\documentclass{amsart}

\usepackage[T1]{fontenc}
\usepackage[utf8]{inputenc}
\usepackage{lmodern}
\usepackage[english]{babel}
\usepackage[autostyle]{csquotes}

\usepackage{amsmath}
\usepackage{amsfonts}
\usepackage{mathtools}
\usepackage{graphicx}
\usepackage{mathrsfs}
\usepackage{amsthm}
\usepackage{amssymb}
\usepackage{centernot}
\usepackage{nicematrix}
\usepackage{enumerate}
\usepackage{colonequals}
\usepackage{mathbbol}
\usepackage[all,cmtip]{xy}
\usepackage{tikz}
\usepackage{tikz-cd}
\usepackage[mathcal]{euscript}
\usepackage[hidelinks]{hyperref}
\usepackage{cleveref}
\usepackage{comment}
\usepackage{mathdots}
\usepackage{varwidth}

\numberwithin{equation}{subsection}

\usepackage[margin = 2cm]{geometry}

\makeatletter
\newcommand{\sigmaop}[1]{\mathop{\mathpalette\@sigmaop{#1}}\slimits@}
\newcommand{\@sigmaop}[2]{%
  \vphantom{\sum}%
  \sbox\z@{$\m@th#1\sum$}%
  \dimen@=\ht\z@ \advance\dimen@\dp\z@
  \dimen\tw@=\wd\z@
  \ifx#1\displaystyle\dimen@=.9\dimen@\fi
  \ooalign{%
    \hidewidth
    $\vcenter{\hbox{$\m@th#1#2$\kern.3\dimen\tw@}%
     \ifx#1\scriptstyle\kern-.25ex\fi}$\hidewidth\cr
    $\vcenter{\hbox{%
      \resizebox{!}{\dimen@}{$\m@th\boxtimes$}%
    }\ifx#1\scriptstyle\kern-.25ex\fi}$\cr
  }%
}
\makeatother

\numberwithin{equation}{subsection}

\newtheorem{theorem}{Theorem}[subsection]
\newtheorem{lemma}[theorem]{Lemma}

\newtheorem{conjecture}[theorem]{Conjecture}

\newtheorem{corollary}[theorem]{Corollary}
\newtheorem{definition}[theorem]{Definition}

\newtheorem{proposition}[theorem]{Proposition}

\newtheorem{assumption}[theorem]{Assumption}

\newtheorem*{thm1}{Theorem 1}

\newtheorem*{prop1}{Proposition 1}
\newtheorem*{prop2}{Proposition 2}
\newtheorem*{prop3}{Proposition 3}

\newtheorem*{thmA}{Theorem A}
\newtheorem*{thmB}{Theorem B}
\newtheorem*{thmC}{Theorem C}
\newtheorem*{thmD}{Theorem D}
\newtheorem*{thmE}{Theorem E}

\theoremstyle{remark}

\newtheorem{remark}[theorem]{Remark}

\newtheorem{exa}[theorem]{Example}
\newtheorem{algorithm}[theorem]{Algorithm}

\newcommand{\colim}{%
  \mathop{\mathpalette\colim@{\rightarrowfill@\scriptscriptstyle}}\nmlimits@
}
\renewcommand{\varprojlim}{%
  \mathop{\mathpalette\varlim@{\leftarrowfill@\scriptscriptstyle}}\nmlimits@
}
\renewcommand{\varinjlim}{%
  \mathop{\mathpalette\varlim@{\rightarrowfill@\scriptscriptstyle}}\nmlimits@
}

\newcommand{\GZip}{\mathop{\text{$G$-{\tt Zip}}}\nolimits}

\newcommand{\procskip}{\vskip\procskipamount}
\newcommand{\interskip}{\vskip\interskipamount}
\newcommand{\refskip}{\vskip\refskipamount}

\newcommand{\procbreak}{\par
   \ifdim\lastskip<\procskipamount\removelastskip
   \penalty-100
   \procskip\fi
   \noindent\ignorespaces}

\newcommand{\titlebreak}{\par%
\ifdim\lastskip<\interskipamount\removelastskip%
\penalty10000%
\interskip\fi%
\noindent}%

\newcommand{\interbreak}{\par%
\ifdim\lastskip<\interskipamount\removelastskip%
\penalty-100%
\interskip\fi%
\noindent\ignorespaces}%

\newcommand{\refbreak}{\par%
\ifdim\lastskip<\refskipamount\removelastskip%
\penalty-100%
\refskip\fi%
\noindent\ignorespaces}%

\newcounter{listcounter}
\newcounter{deflistcounter}
\newcounter{equivcounter}

\newskip{\itemsepamount}
\newskip{\topsepamount}
\newenvironment{assertionlist}{%
  \begin{list}
    {\upshape (\arabic{listcounter})}
    {\setlength{\leftmargin}{18pt}
     \setlength{\rightmargin}{0pt}
     \setlength{\itemindent}{0pt}
     \setlength{\labelsep}{5pt}
     \setlength{\labelwidth}{13pt}
     \setlength{\listparindent}{\parindent}
     \setlength{\parsep}{0pt}
     \setlength{\itemsep}{\itemsepamount}
     \setlength{\topsep}{\topsepamount}
     \usecounter{listcounter}}}
  {\end{list}}

\newenvironment{definitionlist}{%
  \begin{list}
    {\upshape (\alph{deflistcounter})}
    {\setlength{\leftmargin}{18pt}
     \setlength{\rightmargin}{0pt}
     \setlength{\itemindent}{0pt}
     \setlength{\labelsep}{5pt}
     \setlength{\labelwidth}{13pt}
     \setlength{\listparindent}{\parindent}
     \setlength{\parsep}{0pt}
     \setlength{\itemsep}{\itemsepamount}
     \setlength{\topsep}{\topsepamount}
     \usecounter{deflistcounter}}}
  {\end{list}}

\newenvironment{equivlist}{%
  \begin{list}
    {\upshape (\roman{equivcounter})}
    {\setlength{\leftmargin}{18pt}
     \setlength{\rightmargin}{0pt}
     \setlength{\itemindent}{0pt}
     \setlength{\labelsep}{5pt}
     \setlength{\labelwidth}{13pt}
     \setlength{\listparindent}{\parindent}
     \setlength{\parsep}{0pt}
     \setlength{\itemsep}{\itemsepamount}
     \setlength{\topsep}{\topsepamount}
     \usecounter{equivcounter}}}
  {\end{list}}

% ...
\newcommand{\Bcal}{{\mathcal B}}
\newcommand{\Ccal}{{\mathcal C}}

\newcommand{\Ecal}{{\mathcal E}}
\newcommand{\Fcal}{{\mathcal F}}
\newcommand{\Gcal}{{\mathcal G}}
\newcommand{\Hcal}{{\mathcal H}}
\newcommand{\Ical}{{\mathcal I}}
\newcommand{\Jcal}{{\mathcal J}}

\newcommand{\Lcal}{{\mathcal L}}

\newcommand{\Ocal}{{\mathcal O}}
\newcommand{\Pcal}{{\mathcal P}}
\newcommand{\Qcal}{{\mathcal Q}}

\newcommand{\Scal}{{\mathcal S}}
\newcommand{\Tcal}{{\mathcal T}}
\newcommand{\Ucal}{{\mathcal U}}

\newcommand{\Xcal}{{\mathcal X}}

\newcommand{\Zcal}{{\mathcal Z}}

\newcommand{\pfr}{{\mathfrak p}}

\newcommand{\Sfr}{{\mathfrak S}}

\renewcommand{\AA}{\mathbb{A}}

\newcommand{\CC}{\mathbb{C}}
\newcommand{\DD}{\mathbb{D}}

\newcommand{\FF}{\mathbb{F}}
\newcommand{\GG}{\mathbb{G}}

\newcommand{\QQ}{\mathbb{Q}}

\newcommand{\ZZ}{\mathbb{Z}}

\newcommand{\Fscr}{{\mathscr F}}

\newcommand{\Pscr}{{\mathscr P}}

\newcommand{\Uscr}{{\mathscr U}}

\newcommand{\Xscr}{{\mathscr X}}

\newcommand{\leftexp}[2]{{\vphantom{#2}}^{#1}{#2}}

\DeclareMathOperator{\Cent}{Cent}

\DeclareMathOperator{\C}{C}

\DeclareMathOperator{\NC}{NC}
\DeclareMathOperator{\Gal}{Gal}

\DeclareMathOperator{\rank}{rank}
\DeclareMathOperator{\Span}{Span}

\DeclareMathOperator{\Lie}{Lie}

\DeclareMathOperator{\Tr}{Tr}

\DeclareMathOperator{\Stab}{Stab}

\DeclareMathOperator{\pr}{pr}

\DeclareMathOperator{\Ker}{Ker}

\DeclareMathOperator{\rk}{rk}

\DeclareMathOperator{\inv}{inv}
\DeclareMathOperator{\Nm}{Nm}

\DeclareMathOperator{\Sh}{Sh}
\DeclareMathOperator{\spec}{Spec}

\DeclareMathOperator{\Adj}{Adj}

\DeclareMathOperator{\GL}{GL}

\newcommand{\id}{{\rm Id}}

\newcommand{\loccit}{{\em loc.\ cit.\ }}
\newcommand{\loccitn}{{\em loc.\ cit.}}

\newcommand{\diag}{{\rm diag}}

\renewcommand{\Im}{{\rm Im}}

\newcommand{\iw}{\leftexp{I}{W}}

\DeclareMathOperator{\divi}{div}

\DeclareMathOperator{\ima}{im}

\DeclareMathOperator{\val}{val}

\DeclareMathOperator{\Bru}{Br}

\DeclareMathOperator{\iden}{id}

\DeclareMathOperator{\Mat}{Mat}

\newcommand{\relmiddle}[1]{\mathrel{}\middle#1\mathrel{}}

\newcommand{\smallmat}[4]{\bigl(\begin{smallmatrix}#1&#2\\#3&#4\end{smallmatrix}\bigr)}

\newcommand{\xdasharrow}[2][->]{
\tikz[baseline=-\the\dimexpr\fontdimen22\textfont2\relax]{
\node[anchor=south,font=\scriptsize, inner ysep=1.5pt,outer xsep=2.2pt](x){#2};
\draw[shorten <=3.4pt,shorten >=3.4pt,dashed,#1](x.south west)--(x.south east);
}
}

\newcommand{\chara}{\mathsf{char}}
\newcommand{\Ha}{\mathsf{Ha}}

\newcommand{\JS}[1]{{\color{red} [#1]}} %Jean-Stefan: 
\newcommand{\lorenzo}[1]{{\color{blue} [#1]}} % Lorenzo: 
\title{Decidability of singularities in the Ekedahl--Oort stratification} 

\author{Jean-Stefan Koskivirta, Lorenzo La Porta}

\date{}

\makeatletter
\let\c@equation=\c@subsubsection

\let\c@figure=\c@subsubsection

\let\c@table=\c@subsubsection

\begin{document}

\begin{abstract}
For an abelian type Shimura variety and an odd prime \(p\) of good reduction, we characterize the regularity in codimension one of Zariski closures of Ekedahl--Oort strata in terms of the Frobenius action on the root datum. We give an algorithm that detects codimension one singularities for arbitrary Ekedahl--Oort strata. When the Shimura datum is of split type, we relate the singularities of Ekedahl--Oort strata to a stack of $G$-zips over the complex numbers. We study the existence of generalized Hasse invariants on this stack.
%This stack can also detect the presence of singularities in unions of Ekedahl--Oort strata in characteristic \(p\).
\end{abstract}

\maketitle

\tableofcontents

\section{Introduction}
The Ekedahl--Oort (EO) stratification on the special fiber of a Shimura variety was first introduced by Ekedahl and Oort \cite{oortastrat}, in the case of Siegel modular varieties. This was subsequently generalized by the work of many \cite{EOVW, zhang, shen.zhang, imai2024prismaticrealizationfunctorshimura}. It has proved a very effective tool to study the geometry of such moduli spaces. For example, the theory of generalized Hasse invariants on EO strata in \cite{Goldring-Koskivirta-Strata-Hasse, boxer} made it possible to extend the construction of Galois representations attached to more general automorphic representations. The singularities appearing in the EO strata encode interesting arithmetic information. For this reason, it is an important problem to detect and measure such singularities. In this paper, building on the results of \cite{Koskivirta-LaPorta-Reppen-singularities}, we address these questions. %and answer some questions \JS{"some questions" - not precise. I I prefer not using "some"} that were left open in \loccit We achieve these results by introducing new ideas that are likely to have applications beyond the scope of this work.

Let us explain in further detail. Fix \(p\) an odd prime. Write \(S\) for the \(\overline{\FF}_p\)-fiber at some \(p\)-hyperspecial level \(K\) of an abelian type Shimura variety. %When \(p\) and \(K\) are clear from context, we simply write \(S\), instead of \(S_{K, p}\). 
Denote the EO strata in \(S\) by $(S_w)_{w\in {}^I W}$, where ${}^IW$ is a certain subset of the Weyl group of the relevant reductive group $G$, \S\ref{subsubsec-zip-strata}. Here, $I$ denotes the type of the Hodge parabolic $P\subseteq G$ attached to the Shimura datum. The smooth and normal loci of the Zariski closure $\overline{S}_w$ can be written as unions of EO strata. It is a difficult problem to determine which EO strata appear therein. In the recent preprint \cite{Koskivirta-LaPorta-Reppen-singularities}, the authors together with S.\ Reppen gave a criterion for the normality of a locally closed subcheme $U\subseteq S$ which is a union of EO strata. Such a subcheme is called an \emph{elementary $w$-open} if $U=S_w\cup S_{w'}$, where $S_{w'}$ has codimension one in $\overline{S}_w$. This condition is equivalent to $w'\in \Gamma_I(w)$, the set of lower neighbors of $w$ in ${}^IW$ with respect to a certain partial order $\preccurlyeq_I$, see \S\ref{sec-closure}.

For elementary $w$-open subschemes $U$, the authors gave a list of explicit conditions that are equivalent to the smoothness of $U$ (itself equivalent to normality, in this setting). We recall the result below. To explain it, we need to introduce two notions. First, $U$ is \emph{$w$-bounded} if $P_{w'}\subseteq P_w$, where $P_w$, $P_{w'}$ denote the \emph{canonical parabolic subgroups} of $w,w'$ respectively, see \cite[Def.~2.11]{Koskivirta-LaPorta-Reppen-singularities}. Second, for any intermediate parabolic subgroup $B\subseteq P_0\subseteq P$ (where $B$ is a fixed Borel subgroup), we consider the flag space $S^{(P_0)}$, which parametrizes points of $S$ endowed with a $P_0$-subtorsor of the $P$-torsor given by the Hodge filtration. Write $\pi_{P_0}\colon S^{(P_0)}\to S$ for the natural projection. There is a stratification $(S^{(P_0)}_{w})_{w\in {}^{I_0} W}$, where $I_0\subseteq I$ is the type of $P_0$, \S\ref{subsub-flag}. We say that $U$ \emph{admits a separating canonical cover} if the preimage of $U$ under the map 
\begin{equation}
    \pi_{P_w}\colon \overline{S}^{(P_w)}_w \to \overline{S}_w
\end{equation}
coincides with $U^{(P_w)} \colonequals S^{(P_w)}_w\cup S^{(P_w)}_{w'}$.
\begin{thm1}[{\cite[Corollary 4.17]{Koskivirta-LaPorta-Reppen-singularities}}] \label{theorem: previous result}
The following are equivalent:
\begin{enumerate}
    \item The scheme $U$ is smooth.
    \item The scheme $U$ is normal.
    \item The scheme $U$ is $w$-bounded and admits a separating canonical cover.
\end{enumerate}
\end{thm1}
In this paper, we build upon Theorem \ref{theorem: previous result} and uncover several new aspects of the EO stratification. Our main results are twofold:
\begin{enumerate}[(1)]
    \item We give an algorithm which determines whether an elementary $w$-open subset is smooth (equivalently, normal) in terms of the root datum of $G$.
    \item When the group $G$ splits over $\FF_p$, we show that the presence of singularities in the EO stratification is controlled by an object in characteristic zero.
\end{enumerate}
\subsection{Algorithm for smoothness}
Write $\Xcal \colonequals \GZip^\mu$ for the stack of $G$-zips associated to the Shimura datum, \S\ref{subsec-generalities-zip}. It is defined as a quotient stack $[E\backslash G_{\overline{\FF}_p}]$, where $E$ is the attached zip group, \S\ref{subsubsec-G-zips}. Recall that there is a smooth surjective map
\begin{equation}\label{zeta-intro-eq}
    \zeta\colon S\to \Xcal
\end{equation}
constructed by Viehmann--Wedhorn, Zhang and  Imai--Kato--Youcis \cite{EOVW, zhang, imai2024prismaticrealizationfunctorshimura}. The EO stratification is the pullback of the \emph{zip stratification} $(\Xcal_w)_{w\in {}^IW}$ of $\Xcal$, which is itself given by the $E$-orbits $(G_w)_{w\in {}^I W}$ inside $G_{\overline{\FF}_p}$.

Similarly to $\pi_{P_0}\colon S^{(P_0)}\to S$, we have a flag space $\pi_{P_0}\colon \Fcal^{(P_0)}\to \Xcal$ (for any intermediate parabolic $P_0$), carrying a stratification $(\Fcal^{(P_0)}_w)_{w\in {}^{I_0}W}$. For $w\in W$, consider the stratum $\Fcal^{(B)}_w$ in the total flag space $\Fcal^{(B)}$. Its image $\pi_B(\Fcal^{(B)}_w)$ is a union of certain zip strata of $\Xcal$. There exists a unique zip stratum $\Xcal_{v}$ which is open dense in $\pi_B(\Fcal^{(B)}_w)$. This defines a map $w\mapsto v$ that we denote by
\begin{equation}
    \pi_{\overline{\FF}_p}\colon W\to {}^I W.
\end{equation}
The subscript $\overline{\FF}_p$ is used to distinguish this map from an analogous map $\pi_{\CC}\colon W\to {}^I W$ which we define later. The map $\pi_{\overline{\FF}_p}$ is a section of the natural inclusion ${}^I W\subseteq W$. The computation of the map $\pi_{\overline{\FF}_p}$ is in general highly involved. We do not know an algorithm to determine it.

Set $z\colonequals \sigma(w_{0,I})w_0$, where $w_{0,I}\in W_I\colonequals W(L,T)$ and $w_0\in W$ are the longest elements in \(W_I\) and \(W\), respectively, and $\sigma$ is the action of the $p$-power Frobenius on $W$. Write $\Xi\colon G(\overline{\FF}_p)\to {}^I W$ for the map which takes an element $g\in G(\overline{\FF}_p)$ to the unique element $w\in {}^I W$ such that $gz^{-1}\in G_w$. There is an explicit algorithm that computes the restriction of the map $\Xi$ to the subset $W\subseteq G(\overline{\FF}_p)$ (this inclusion is defined by choosing representatives). In the case $G=\GL_{n,\FF_p}$, we also give an algorithm to determine $\Xi$ entirely, \S\ref{subsec-algo-GLn}. For general strata, one needs to know the image by $\Xi$ of infinitely many points in order to compute $\pi_{\overline{\FF}_p}$. This is an obstacle to determining $\pi_{\overline{\FF}_p}$ algorithmically.

We show, however, that the image by \(\pi_{\overline{\FF}_p}\) of certain special elements of $W$ can be computed by an algorithm. We say that $w\in W$ is \emph{small} if $\ell(\pi_{\overline{\FF}_p}(w))=\ell(w)$, where $\ell\colon W\to \ZZ_{\geq 0}$ is the length function. Write $W^{\rm sm}\subseteq W$ for the set of small elements. There is an explicit, combinatorial characterization of smallness in terms of roots, which is computable, \S \ref{subsec-small-roots}. One has the following key result:
\begin{prop1}[{Proposition \ref{prop-small-image}}] \label{prop: pi for small} Let $w\in W^{\rm sm}$.
\begin{enumerate}
    \item The underlying topological space of $\Fcal^{(B)}_w$ is a single point.
    \item One has $\pi_{\overline{\FF}_p}(w)=\Xi(w)$.
\end{enumerate}
 In particular, the restriction of $\pi_{\overline{\FF}_p}$ to $W^{\rm sm}$ can be computed algorithmically.
\end{prop1}

We now explain how to obtain an algorithm to determine the smoothness of elementary $w$-open subsets.
%We conjecture that the maps $\pi_{\CC}$ and $\pi_{\overline{\FF}_p}$ coincide, but we are not, at the moment, able to prove it. 
Let $w\in {}^I W$ and consider $w'$ a lower neighbor of $w$ in ${}^I W$. We have the corresponding elementary $w$-open subsets $U(w,w')$ and $\Ucal(w,w')$ inside $S$ and $\Xcal$, respectively.
%have corresponding elementary $w$-open substacks $U(w,w')$, $\Ucal(w,w')$ and $\Uscr_{\CC}(w,w')$ inside $S$, $\Xcal$ and $\Xscr_\CC$, respectively.
Write $\Gamma_{I_w}(w)$ for the set of lower neighbors of $w$ in the set ${}^{I_w} W$ (with respect to $\preccurlyeq_{I_w}$), where $I_w\subseteq I$ is the type of the canonical parabolic $P_w$. Set $\Gamma_{I_w}^{\rm sm}(w)\colonequals \Gamma_{I_w}(w)\cap W^{\rm sm}$.
% \begin{prop2} \label{prop: 4 condition}
% The three equivalent conditions of Theorem \ref{theorem: previous result} are also equivalent to the following:
% \begin{enumerate}
%     \item[(4)] One has $P_{w'}\subseteq P_w$ and $w'\notin \pi_{\overline{\FF}_p}(\Gamma^{\rm sm}_{I_w}(w)\setminus\{w'\})$.
% \end{enumerate}
% \end{prop2}
We show that the three equivalent conditions of Theorem \ref{theorem: previous result} are also equivalent to the following:
\begin{enumerate}
    \item[(4)] One has $P_{w'}\subseteq P_w$ and $w'\notin \pi_{\overline{\FF}_p}(\Gamma^{\rm sm}_{I_w}(w)\setminus\{w'\})$.
\end{enumerate}
Since the map $\pi_{\overline{\FF}_p}$ can be algorithmically computed on small elements by Proposition 1, we obtain our first main result:
%The determination of the map $\pi_{\overline{\FF}_p}\colon W\to {}^I W$ involves a priori an infinite search space. However, the above shows that it suffices to consider small strata \(w_1\), and check which stratum of $\Xcal$ the element \(w_1\) is sent to. Concretely, one has to determine the $E$-orbit of each element of the form $w_1z^{-1}$.
\begin{thmA}[{Theorem \ref{thm: algo for smooth}}]
There exists an algorithm which determines whether $U(w,w')$ (equivalently, $\Ucal(w,w')$) is smooth (equivalently, normal).
\end{thmA}

In work in progress \cite{Koskivirta-LaPorta-Reppen-n-22} of the authors together with  S.\ Reppen, we are able to completely determine the singularities appearing in the Ekedahl--Oort stratification for unitary Shimura varieties of signature $(n-2,2)$ at a split prime, using this algorithm.

\subsection{Independence from \texorpdfstring{$p$}{p}}
A consequence of the proof of Proposition 1 and Theorem A above %\ref{prop: pi for small} and \ref{prop: 4 condition}\JS{refs not working!}
is that the smoothness of an elementary open \(U\) is entirely determined by the action of the Frobenius on the based root datum \(\Phi\) of \(G\) (with respect to an adequate choice of a Borel pair, see \S\ref{subsubsec-G-zips} and \S\ref{sssec: frob action}). Consider now two odd primes $p,\ell$ of good reduction. Write \(\sigma_p, \sigma_\ell\) for the automorphisms of \(\Phi\) induced by the Frobenius at \(p\) and $\ell$, respectively. We deduce the following.
\begin{thmB}[{Corollary \ref{cor-lp}}]
Assume that \(\sigma_p = \sigma_\ell\). For $w\in \iw$ and $w'\in \Gamma_{I}(w)$, denote by $U_p$ and $U_\ell$ the corresponding $w$-open subsets in the mod $p$ and mod $\ell$ special fibers, respectively.
The following are equivalent:
\begin{enumerate}
\item $U_{p}$ is smooth (equivalently, normal).
\item $U_{\ell}$ is smooth (equivalently, normal).
\end{enumerate}
\end{thmB}

\subsection{Length \texorpdfstring{\(2\)}{2} strata}
We illustrate our results with a concrete application to 2-dimensional strata in unitary Shimura varieties of signature $(r,s)$ at a split prime of good reduction. Let $S$ denote the special fiber of this Shimura variety. We assume that $r\geq s \geq 2$ (in the case $s=1$, the Zariski closure of each EO stratum is smooth by \cite{laporta2023generalised}). In this case, there are exactly two elements of length $2$ in ${}^{I}W$, that we denote by $w_1$, $w_2$ (see \S\ref{subsec-GUrs} for the exact parametrization), and a single length $1$ element, denoted $w'$. Set $U_i\colonequals S_{w_i}\cup S_{w'}$ for $i\in \{1,2\}$. For $x$ coprime to $n$, let $\inv_n(x)$ denote the unique element $k\in \{1,\dots,n-1\}$ such that $kx\equiv 1 \pmod{n}$.
\begin{thmC}[{Theorem \ref{thm-Urs-smooth}}] \ 
\begin{enumerate}
    \item If $\gcd(r,s)>3$, then $U_i$ is not $w_i$-bounded (for $i=1,2$). In particular, $U_i$ is not smooth.
    \item If $\gcd(r,s)\in \{2,3\}$, then $U_i$ is smooth for $i=1,2$.
    \item If $\gcd(r,s)=1$ and we set $m=\inv_n(s)$, one has
    \begin{equation*}
       U_1 \ \textrm{smooth} \ \Longleftrightarrow \ m>\frac{n}{2}, \quad  
        U_2 \ \textrm{smooth} \ \Longleftrightarrow \ m<\frac{n}{2}.
    \end{equation*}
\end{enumerate}
\end{thmC}

%The important aspects of this result are the following:
%\begin{itemize}
%    \item The determination of the map $\pi_{\overline{\FF}_p}$ involves a priori an infinite search space. The above shows that it can be reduced to a finite set.
%    \item The algorithm is independent of $p$.
%\end{itemize}

\subsection{The case of a split prime}
The second main part of this paper shows that for primes of good reduction at which \(G\) splits, the presence of singularities in the EO stratification is controlled by a characteristic zero object. Assume from now on that $G$ admits a split Borel pair $(B,T)$ over $\FF_p$ such that $B\subseteq P$. In particular, $P$ is defined over $\FF_p$. Let $Q$ denote the opposite parabolic of $P$ with respect to $T$ and let $L=P\cap Q$ be the common Levi subgroup. Recall that $\Xcal$ parametrizes tuples $(\Ical,\Ical_P,\Ical_Q,\iota)$ where $\Ical$ is a $G$-torsor, $\Ical_P\subseteq \Ical$ is a $P$-torsor, $\Ical_Q\subseteq\Ical$ is a $Q$-torsor, and $\iota\colon (\Ical_P/R_{\mathrm{u}}(P))^{(p)}\to \Ical_Q/R_{\mathrm{u}}(Q)$ is an isomorphism of $L$-torsors. For any integer $m\geq 0$, let $\Xcal_m$ be the stack parametrizing such tuples, but with $\iota$ now an isomorphism of $L$-torsors $(\Ical_P/R_{\mathrm{u}}(P))^{(p^m)}\to \Ical_Q/R_{\mathrm{u}}(Q)$ (in particular, $\Xcal=\Xcal_1$). Each $\Xcal_m$ admits a zip stratification $(\Xcal_{m,w})_{w\in {}^I W}$. There is an isomorphism $\Xcal_m\simeq [E_m\backslash G_{\overline{\FF}_p}]$ where $E_m\subseteq P\times Q$ is the corresponding zip group, \S\ref{subsubsec-zip-expm}. Via this isomorphism, we can write $\Xcal_{m,w}=[E_m\backslash G_{m,w}]$ for a locally closed subscheme $G_{m,w}\subseteq G$. Using Galois cohomology, we show that $G_{m,w}(\FF_{p^m})=G_{0,w}(\FF_{p^m})$ for all $m\geq 0$ (Proposition \ref{prop-Gw}). In particular, the set $G_{m,w}(\FF_{p})$ is independent of $m\geq 0$. This has the following consequence. Write $\Ucal_m(w,w')$ for the elementary $w$-open subset $\Xcal_{m,w}\cup \Xcal_{m,w'}$, where $w'$ is a lower neighbor in ${}^I W$ of $w$. The case $m=0$ is of particular interest.
\begin{prop2}[Proposition \ref{prop-SCC-m}]
Let $m\geq 1$ be an integer. The following are equivalent:
\begin{enumerate}
    \item $\Ucal_0(w,w')$ is $w$-bounded and admits a separating canonical cover. 
    \item $\Ucal_m(w,w')$ is smooth (equivalently, normal).
\end{enumerate}    
\end{prop2}
Now, assume that $G$ admits a reductive model $\Gcal$ over $\ZZ_p$ which is split over $\ZZ_p$ (we say that $p$ is a prime of \emph{split good reduction}). Since Frobenius is not involved in the construction of $\Xcal_0$, this stack admits a natural model $\Xscr$ over $\ZZ_p$, defined similarly. Again, it carries a stratification $(\Xscr_w)_{w\in {}^I W}$. Choosing an isomorphism $\overline{\QQ}_p\simeq \CC$, we can also consider the stack $\Xscr_\CC$ over $\CC$. The notions of ``$w$-boundedness'' and ``separating canonical cover'', as well as the formation of the flag space, make sense on $\Xscr_\CC$. We also have the analogue of $\pi_{\overline{\FF}_p}$, that we denote by $\pi_\CC\colon W\to {}^I W$.
\begin{prop3}[Proposition \ref{prop-pi-Qp-equals-Fp}]
    For $w\in W^{\rm sm}$, we have $\pi_{\overline{\FF}_p}(w)=\pi_{\CC}(w)$.
\end{prop3}
We conjecture the stronger statement $\pi_{\overline{\FF}_p}=\pi_{\CC}$ but we are unable to prove it at the moment. Write $\Uscr_{\CC}(w,w')$ for the corresponding elementary substack of $\Xscr_{\CC}$, i.e.\ the union of $\Xscr_{\CC,w}$ and $\Xscr_{\CC,w'}$. As a consequence, we obtain:
\begin{thmD}[Theorem \ref{thm-main}]
\label{thm-intro} 
The following are equivalent:
\begin{enumerate}
    \item\label{UKp1-intro} $U(w,w')$ is smooth (equivalently, normal).
    \item \label{UKp2-intro} $\Ucal(w,w')$ is smooth (equivalently, normal). 
    \item \label{UKp3-intro} $\Uscr_{\CC}(w,w')$ is $w$-bounded and admits a separating canonical cover.
\item \label{UKp4-intro} We have $P_{w'}\subseteq P_w$ and $w'\notin \pi_{\CC}(\Gamma^{\rm sm}_{I_w}(w)\setminus \{w'\})$.
\end{enumerate}
\end{thmD}

The equivalence of \eqref{UKp1-intro} and \eqref{UKp2-intro} follows from the smoothness of $\zeta\colon S\to \Xcal$. The important part is the equivalence with \eqref{UKp3-intro}, which involves an object over the complex numbers. Point \eqref{UKp4-intro} is a reformulation of \eqref{UKp3-intro}. %We also obtain an equivalent condition if we replace $\CC$ by $\overline{\FF}_p$ in condition \eqref{UKp4-intro}.

We do not know whether condition \eqref{UKp3-intro} is equivalent to the normality or smoothness of $\Uscr_\CC(w,w')$. It would be desirable to extend the above result beyond the case of elementary $w$-open substacks. Specifically, if $\Ucal\subseteq \Xcal$ is any $w$-open substack, given by $\Ucal=\bigcup_{v\in \Gamma_\Ucal} \Xcal_{v}$ for a certain subset $\Gamma_\Ucal\subseteq {}^I W$, we may define $\Uscr_\CC\colonequals \bigcup_{v\in \Gamma_\Ucal} \Xscr_{\CC,v}$. We conjecture that the smooth (resp.\ normal) loci of $\Ucal$ and $\Uscr_\CC$ correspond to the same unions of strata. This is work in progress by the authors. %Since $\Xscr_\CC$ is independent of the prime $p$, the above result has the following important consequence.

\subsection{Hasse invariants in characteristic zero}
It was proved in \cite{Goldring-Koskivirta-Strata-Hasse} that generalized Hasse invariants exist on all zip strata of $\Xcal$, and hence on all EO strata by pullback via \eqref{zeta-intro-eq}. We explore their existence on the stack $\Xscr_\CC$. %However, for general stacks of $G$-zips which are not of maximal type, they do not always exist (see \cite{Goldring-Koskivirta-zip-flags} for a counter-example).

By condition \eqref{UKp4-intro} of Theorem D, the singularities of elementary substacks are encoded by the map $\pi_\CC\colon W\to {}^I W$. Despite our algorithm to determine $\pi_{\overline{\FF}_p}$ (equivalently, $\pi_\CC$) on small elements, the computation can become quite involved, as shown in the various cases appearing in Theorem C. Moreover, we do not know how to determine $\pi_\CC$ on non-small elements. One motivation for studying Hasse invariants on $\Xscr_\CC$ is that they make it possible to determine $\pi_\CC$ fully and explicitly, when they exist.

For $\lambda\in X^*(L)$, we have a line bundle $\Lcal(\lambda)$ on each of the stacks $\Xcal$, $\Xscr$, $\Xscr_\CC$. We say that a section $\Ha_w\in H^0(\overline{\Xscr}_{\CC,w},\Lcal(\lambda))$ is a generalized Hasse invariant if its nonvanishing locus coincides with $\Xscr_{\CC,w}$. Let $E_w$ be the set of positive roots $\alpha$ such that $ws_\alpha\leq w$ and $\ell(ws_\alpha)=\ell(w)-1$ (where $s_\alpha$ is the reflection with respect to $\alpha$). 
\begin{thmE}[Theorem \ref{thm-Hasse-char0}]
Let $w\in {}^{I}W$ and $\lambda\in X^*(L)$. %Set $z\colonequals w_{0,I}w_0$. 
The following are equivalent:
\begin{enumerate}
\item There exists a generalized Hasse invariant $\Ha_w\in H^0(\overline{\Xscr}_{\CC,w},\Lcal(\lambda))$.
\item There exists $\lambda_0\in X^*(T)$ such that $\lambda = w\lambda_0-z\lambda_0$ and $\langle \lambda_0,\alpha^\vee \rangle <0$ for all $\alpha\in E_w$.
\end{enumerate}
\end{thmE}
In stark contrast with the characteristic $p$ situation (and automorphic forms on Shimura varieties), the trivial line bundle (i.e.\ $\lambda=0$) sometimes admits such generalized Hasse invariants. For example, in the case $G=\GL_{n}$, with $L$ of type $(n-1,1)$, generalized Hasse invariants of weight $\lambda=0$ exist on each stratum (see \S \ref{sub-exampleGL}). This case corresponds to a unitary Shimura variety of signature $(n-1,1)$ at a prime $p$ of good reduction which splits in the associated quadratic imaginary extension $\mathbf{E}/\QQ$. For other signatures, Hasse invariants do not always exist on all strata, see \S\ref{sec-U22}. % These invariants make it possible to determine explicitly the map $\Pi_\CC$ in that case (Example \ref{sub-exampleGL}). \lorenzo{We should spell this out}

\subsection{Acknowledgements}
J-SK was supported by the University of Caen-Normandie. He thanks Wansu Kim, Stefan Reppen and Antoine Szabo for fruitful discussions on topics related to this article.

LLP was supported by the \emph{``Piano di Sviluppo Dipartimentale''} with the title \emph{``Enhancing the Research in the Math Dept'', CUP C93C24000160005}. He wishes to thank the Department of Mathematics ``\emph{Tullio Levi-Civita}'' of the University of Padova and its members for their continued support and friendship.

\section{Review of the theory of \texorpdfstring{$G$}{G}-zips}\label{sec-review-zip}

\subsection{Generalities on stacks of \texorpdfstring{$G$}{G}-zips} \label{subsec-generalities-zip}

For our purposes, we need to work in a setting that allows fields of arbitrary characteristic. We review some basic facts about the stack of $G$-zips from \cite{Pink-Wedhorn-Ziegler-zip-data, Pink-Wedhorn-Ziegler-F-Zips-additional-structure} in the case of a general isogeny $\varphi$. We will later specialize to the cases when $\varphi$ is a power of the Frobenius homomorphism, or the identity map. Let $K$ denote an algebraically closed field.  %\lorenzo{See my comment right below.}\JS{Noted}

\subsubsection{Zip data}\label{subsubsec-zip-data}
%\lorenzo{Wouldn't we want \(G\) to be defined over some subfield of \(K = \overline{K}\)?}\JS{Yes, but in the usual case we need a Frobenius isogeny on $G$, so we must start with an $\FF_p$-group.} 

Fix a connected, reductive group $G$ over $K$ and cocharacter $\mu\colon \GG_{\textrm{m},K}\to G$. Fix also an isogeny $\varphi \colon G\to G$. To the triple $(G,\mu,\varphi)$, we attach a zip datum $\Zcal$ as follows. First, write $P_{\pm}\subset G$ for the pair of opposite parabolic subgroups afforded by $\mu$ (see \cite[\S1.2.1]{Goldring-Koskivirta-Strata-Hasse}), with common Levi subgroup $L=\Cent(\mu)$. Put $P\colonequals P_-$, $Q\colonequals \varphi(P_+)$, and $M\colonequals \varphi(L)$. Set $\Zcal\colonequals (G,P,Q,L,M,\varphi)$. We call $\Zcal$ the zip datum attached to the triple $(G,\mu,\varphi)$.

\subsubsection{Frobenius-type zip data}\label{subsubsec-frob-type-zip}
An important example considered in this article is a zip datum attached to a Frobenius isogeny. In this context, we assume that $G_0$ is a reductive group over $\FF_p$ and that $K=\overline{\FF}_p$ is an algebraic closure of $\FF_p$. We let $\varphi_0\colon G_0\to G_0$ be a positive power of the Frobenius homomorphism of $G_0$. Finally, set $G\colonequals G_{0,K}$, endowed with the isogeny $\varphi\colonequals \varphi_{0,K}$. For any cocharacter $\mu\colon \GG_{\textrm{m},K}\to G$, we may consider the tuple $(G,\mu,\varphi)$ and hence obtain an attached zip datum $\Zcal$, as explained in \S\ref{subsubsec-zip-data}. A zip datum that arises in this way will be called \emph{of Frobenius-type}.

\subsubsection{Stack of $G$-zips}\label{subsubsec-G-zips}
Retain the assumptions of \S\ref{subsubsec-zip-data}. Let $E_{\Zcal}$ denote the zip group of $\Zcal$, i.e.\ the subgroup of $P\times Q$ given by
\begin{equation}
    E_{\Zcal}\colonequals \{(x,y)\in P\times Q, \ \varphi(\overline{x})=\overline{y}\},
\end{equation}
where $\overline{x}\in L$, $\overline{y}\in M$, denote the Levi projections of $x\in P$ and $y\in Q$, respectively. The zip datum $\Zcal$ gives rise to a stack of $G$-zips $\Xcal_{\Zcal}$ over $K$, defined as the quotient stack
\begin{equation}
    \Xcal_{\Zcal}\colonequals \left[ E_{\Zcal} \backslash G \right],
\end{equation}
where $E_\Zcal$ acts on $G$ by $(a,b)\cdot g\colonequals agb^{-1}$ for $(a,b)\in E_\Zcal$, $g\in G$. The zip datum $\Zcal$ is not necessarily orbitally finite (\cite[Def.~7.2]{Pink-Wedhorn-Ziegler-zip-data}), i.e.\ the underlying topological subspace $|\Xcal_\Zcal|$ may be infinite. When $\Zcal$ is of Frobenius-type, $|\Xcal_\Zcal|$ is finite. Fix a Borel pair $(B,T)$ of $G$. For simplicity, we assume:
\begin{enumerate}
    \item $\mu$ factors through $T$,
    \item $B\subseteq P$,
    \item $(B,T)$ is stable by $\varphi$.
\end{enumerate}
The first assumption implies that $T\subseteq L$. Let $W\colonequals W(G,T)$ be the Weyl group of $G$. By our assumptions, $\varphi$ induces an isomorphism of Coxeter systems
\begin{equation}\label{varphi-aut-weyl}
    \varphi\colon W\to W.
\end{equation}
When \(\Zcal\) is of Frobenius-type, we sometimes denote the isomorphism of (\ref{varphi-aut-weyl}) by \(\sigma\) instead of \(\varphi\), using the same notation as the Frobenius action on the root datum. Denote by \(\Phi^+\) the set of positive roots (in our convention, $\alpha \in \Phi^+$ if the $\alpha$-root group $U_{\alpha}$ is contained in $B^+$, the Borel opposite to \(B\) with respect to \(T\)). Write $\Delta$ for the set of simple roots and $I, J\subseteq \Delta,$ for the types of the parabolics $P,Q$, respectively. Note that $I=\Delta_L$ since $B\subseteq P$. Write $W_I\colonequals W(L,T)$ and let $w_0$, $w_{0,I},$ denote the longest elements in $W$ and $W_I$, respectively. Set $z=\varphi(w_{0,I})w_0$. Then, \({}^zB \subseteq Q\) and \(J={z^{-1}}(\varphi(I))\). Moreover, we have a natural isomorphism of Coxeter systems
\begin{align}
    \psi\colon W_I &\longrightarrow W_J, \label{psi-WIWJ} \\
    x &\longmapsto {}^{z^{-1}}\varphi(x).\notag
\end{align}

\subsubsection{Zip strata}\label{subsubsec-zip-strata}
Denote by \(\iw\), respectively \(W^J\), the set of \(w \in W\) such that \(w\) is the shortest element in the coset \(W_Iw \in W_I \backslash W\), respectively \(wW_J \in W/W_J\). For $w\in W$, fix a representative $\dot{w}\in N_G(T)$, such that $(w_1w_2)^\cdot = \dot{w}_1\dot{w}_2$ whenever $\ell(w_1 w_2)=\ell(w_1)+\ell(w_2)$ (this is possible by choosing a Chevalley system \cite[XXIII, \S6]{sga3}). When no confusion occurs, we write $w$ instead of $\dot{w}$. For $w\in {}^I W \cup W^J$, define $G_{\Zcal,w}\colonequals E_\Zcal \cdot (BwBz^{-1})$, i.e.\ the union of all $E_\Zcal$-orbits of elements of $BwBz^{-1}$. Then, \(G_{\Zcal,w}\) is a smooth, locally closed subset of $G$. The families $(G_{\Zcal,w})_{w\in {}^I W}$ and $(G_{\Zcal,w})_{w\in W^J}$ are two parametrizations of the same stratification of $G$, called the \emph{zip stratification} (the one by $W^J$ is called the \emph{dual parametrization}). For $w\in {}^I W \cup W^J$, put
\begin{equation}
    \Xcal_{\Zcal,w}\colonequals \left[E_\Zcal \backslash G_{\Zcal,w}\right].
\end{equation}
This gives a locally closed stratification $(\Xcal_{\Zcal,w})_{w\in {}^I W}$ on $\Xcal_\Zcal$. When $\Zcal$ is of Frobenius-type, $G_{\Zcal,w}$ coincides with a single $E_\Zcal$-orbit, and $|\Xcal_{\Zcal,w}|$ is a single point. We sometimes write $\Xcal$, $\Xcal_w$ for the stack of $G$-zips and the zip strata, if the zip datum $\Zcal$ is clear  from the context.

\subsubsection{Closure relations} \label{sec-closure}
The closure relations between zip strata are described by a partial order $\preccurlyeq_\Zcal$ on ${}^I W$, defined in \cite[Def.~6.1]{Pink-Wedhorn-Ziegler-zip-data}. Write $\leq$ for the Bruhat order on $W$. For two elements $w,w'\in {}^I W$, set
\begin{equation}
    w' \preccurlyeq_\Zcal w \Longleftrightarrow \exists x\in W_I, \ xw'\psi(x)^{-1}\leq w.
\end{equation}
By \loccit Theorem 7.5, we have
\begin{equation}\label{eq-closure}
\overline{\Xcal}_{\Zcal,w}=\bigcup_{w'\preccurlyeq_\Zcal w}\Xcal_{\Zcal,w'}.
\end{equation}
 In general, $\preccurlyeq_\Zcal$ is finer than the Bruhat order $\leq$.
\begin{definition}\label{low-nb-def} Let $w,w'\in {}^I W$.
\begin{enumerate}
\item When $w' \preccurlyeq_{\Zcal} w$ and $\ell(w')=\ell(w)-1$, we say that $w'$ is a lower neighbor of $w$. 
\item When $w' \leq w$ and $\ell(w')=\ell(w)-1$, we say that $w'$ is a Bruhat-lower neighbor of $w$.  
%\item We say that $w'$ is an exceptional lower neighbor of $w$ if it is a lower neighbor that is not a Bruhat-lower neighbor.
\end{enumerate}
\end{definition}
We denote by $\Gamma_{\Zcal}(w)$ the set of all lower neighbors of $w$ in ${}^I W$ with respect to the partial order $\preccurlyeq_\Zcal$. We will also use the notation $\Gamma_I(w)$, especially in \S\ref{subsec-flag-space} below.

\subsubsection{Bruhat strata}\label{subsubsec-bruhat}
The inclusion $E_\Zcal\subseteq P\times Q$ induces a natural projection $\Xcal_\Zcal\to \left[P\backslash G/Q\right]$. The fibers of this map give another stratification on $\Xcal_\Zcal$, which we call the \emph{Bruhat stratification}. In general, the Bruhat stratification is coarser than the zip stratification. We say that a stratum $\Xcal_{\Zcal,w}$ (for $w\in \iw$) is \emph{Bruhat} if it coincides with a Bruhat stratum. This condition is equivalent to the equality $G_{\Zcal,w}=Pwz^{-1}Q$.

\subsection{The flag space}\label{subsec-flag-space}
Fix \(P_0\) a parabolic subgroup such that $B\subseteq P_0 \subseteq P$. %We recall the definition of the flag space $\Fcal_\Zcal^{(P_0)}$. The idea of the flag space first appeared in \cite{ekedahl.geer.cycles.classes} in the setting of Hilbert--Blumenthal Shimura varieties, and was later generalized by Goldring and the first named author in \cite{Goldring-Koskivirta-zip-flags}.

\subsubsection{Definition}
Set $E'_{\Zcal,P_0}\colonequals E_\Zcal\cap (P_0\times G)$ and define the flag space $\Fcal_{\Zcal}^{(P_0)}$ by
\begin{equation}
    \Fcal_{\Zcal}^{(P_0)}\colonequals \left[ E'_{\Zcal,P_0} \backslash G \right].
\end{equation}
There is a natural surjection map induced by the inclusion $E'_{\Zcal,P_0}\subseteq E_\Zcal$:
\begin{equation}
    \pi_{P_0}\colon \Fcal_{\Zcal}^{(P_0)}\to \Xcal_\Zcal.
\end{equation}
When $P_0=P$, we have $\Fcal^{(P)}_\Zcal=\Xcal_\Zcal$ and the projection $\pi_P$ is the identity. For two intermediate parabolic subgroups $P_1\subseteq P_0$, there is a natural projection
\begin{equation}
    \pi_{P_1,P_0}\colon \Fcal^{(P_1)}_{\Zcal}\to \Fcal^{(P_0)}_{\Zcal}
\end{equation}
such that $\pi_{P_0}\circ \pi_{P_1,P_0}=\pi_{P_1}$. We view the family of flag spaces $\Fcal_\Zcal^{(P_0)}$ as a tower of stacks above $\Xcal_\Zcal$.

\begin{remark}\label{rmk-flag-other}
Let $E_\Zcal$ act on $P/P_0$ by left multiplication via the first projection $E_\Zcal\to P$. The flag space $\Fcal_{\Zcal}^{(P_0)}$ is isomorphic to the quotient stack $\left[E_\Zcal\backslash \left(G\times (P/P_0)\right)\right]$ (with respect to the diagonal action of $E_{\Zcal}$). The map $\pi_{P_0}$ is induced by the first projection $G\times (P/P_0)\to G$, which is $E_\Zcal$-equivariant.
\end{remark}

\subsubsection{Flag strata}\label{subsub-flag}
Let $L_0\subseteq P_0$ be the unique Levi subgroup of $P_0$ containing $T$. Set $M_0\colonequals \varphi(L_0)$ and $Q_0\colonequals M_0 ({}^zB)$, which is a parabolic with Levi subgroup $M_0$. Define a new zip datum by $\Zcal_{P_0}\colonequals (G,P_0,Q_0,L_0,M_0,\varphi)$. We obtain a stack $\Xcal_{\Zcal_{P_0}}$ of $G$-zips attached to the datum $\Zcal_{P_0}$. The natural inclusion $E'_{\Zcal,P_0}\subseteq E_{\Zcal_{P_0}}$ yields a projection morphism of $K$-stacks
\begin{equation}
    \psi_{P_0}\colon \Fcal^{(P_0)}_{\Zcal} \to \Xcal_{\Zcal_{P_0}}. 
\end{equation}
Let $I_0\subseteq I$ and $J_0\subseteq J$ denote the types of the parabolic subgroups $P_0,Q_0$ respectively. 

We can apply all results and notation of \S\ref{subsec-generalities-zip} to the stack $\Xcal_{\Zcal_{P_0}}$. In particular, $\Xcal_{\Zcal_{P_0}}$ admits a zip stratification $(\Xcal_{\Zcal_{P_0},w})_{w\in {}^{I_0} W}$, which can also be labeled $(\Xcal_{\Zcal_{P_0},w})_{w\in W^{J_0}}$, using the dual parametrization (\S \ref{subsubsec-G-zips}). Explicitly, we have $\Xcal_{\Zcal_{P_0}, w}=\left[E_{\Zcal_{P_0}}\backslash G_{\Zcal_{P_0},w} \right]$, where $G_{\Zcal_{P_0},w}=E_{\Zcal_{P_0}}\cdot \left( BwBz^{-1}\right)$, for $w\in {}^{I_0}W\cup W^{J_0}$. Define the flag strata in $\Fcal_{\Zcal}^{(P_0)}$ as the pullback of these zip strata. Thus,
\begin{equation}
    \Fcal_{\Zcal,w}^{(P_0)} \colonequals \left[ E'_{\Zcal,P_0} \backslash G_{\Zcal_{P_0},w} \right].
\end{equation}
We also use the notation $G^{(P_0)}_{\Zcal,w}$ or simply $G^{(P_0)}_{w}$ to denote $G_{\Zcal_{P_0},w}$, especially in \S\ref{sec-split} and \S\ref{sec-reduction-p}. We may also speak of Bruhat strata with respect to the stack $\Xcal_{\Zcal_{P_0}}$ (\S\ref{subsubsec-bruhat}). We use the same terminology for the corresponding flag strata $\Fcal_{\Zcal,w}^{(P_0)}$. Write $\preccurlyeq_{I_0}$ for the partial order $\preccurlyeq_{\Zcal_{P_0}}$ on ${}^{I_0}W$. For $w\in {}^{I_0}W$, write $\Gamma_{I_0}(w)$ for the set $\Gamma_{\Zcal_{P_0}}(w)\subseteq {}^{I_0}W$.

\begin{remark} 
Recall the isomorphism $\Fcal_{\Zcal}^{(P_0)}\simeq \left[E_{\Zcal}\backslash \left(G\times (P/P_0)\right) \right]$ from Remark \ref{rmk-flag-other}. In this description, the stratum $\Fcal^{(P_0)}_{\Zcal,w}$ becomes isomorphic to a quotient $[E_\Zcal \backslash H_{\Zcal,w}^{(P_0)}]$, where $H_{\Zcal,w}^{(P_0)}$ is the $E_\Zcal$-stable subset of $G\times (P/P_0)$ defined by
\begin{equation}\label{HwP0-eq}
    H^{(P_0)}_{\Zcal,w} \colonequals \left\{ (g,hP_0)\in G\times (P/P_0) \ \relmiddle| \ h^{-1}g\varphi(\overline{h})\in G_{\Zcal_{P_0},w}  \right\}.
\end{equation}
The dimension of $H^{(P_0)}_{\Zcal,w}$ is given by the formula:
\begin{equation}\label{dim-H}
    \dim(H^{(P_0)}_{\Zcal,w}) = \ell(w)+\dim(P).
\end{equation}
\end{remark}

\begin{comment}
\begin{lemma}[prospective]
    The image by $\pi_{P_0}$ of a flag stratum is a union of zip strata.
    \JS{Assumption needed}
\end{lemma}
\begin{proof}
By definition, $\pi_{P_0}(\Fcal^{(P_0)}_w)=[E \backslash D_w]$ where $D_w$ is the union of all $E$-orbits intersecting $G_{P_0,w}=E_{P_0}\left( BwBz^{-1}\right)$. Hence
\end{proof}
\end{comment}

\subsubsection{Minimal, cominimal strata}
Note that ${}^I W\subseteq {}^{I_0}W$ and $W^{J}\subseteq W^{J_0}$. We call \emph{minimal} the flag strata $\Fcal_{\Zcal,w}^{(P_0)}$ parametrized by elements $w\in {}^I W$. Similarly, the strata $\Fcal^{(P_0)}_{\Zcal,w}$ for $w\in W^{J}$ are called \emph{cominimal}. We also use this terminology for the elements of ${}^I W$ and $W^J$ themselves.

\begin{comment}
\begin{proof}
For the first statement, it suffices to show that $G_{P_0,w}\subseteq G_w$ for $w\in {}^I W$. This amounts to the containment $E_{P_0}\left( BwBz^{-1}\right)\subseteq E\left( BwBz^{-1}\right)$, which is immediate. For the second statement, it suffices to show that for any $x\in G_{P_0,w}$, the quotient $\Stab_{E}(x) / \Stab_{E'_{P_0}}(x) $ is finite etale. For this, it is enough to show that the identity component of $\Stab_{E}(x)$ is contained in $E'_{P_0}$. Note that the first projection $\pr_1\colon E\to P$ induces a closed immersion $\pr_1\colon \Stab_{E}(x)\to P$. View $\Stab_{E}(x)$ as a subgroup of $P$ in this way. We will show that $\Stab_{E}(x)^{\circ}\subseteq B$. It is equivalent to show that the Lie algebra of this stabilizer is contained in $\Lie(B)$. This verification is identical to \cite[]{Koskivirta-Normalization}. \JS{Check details}
\end{proof}
\end{comment}

\subsection{Images of flag strata}\label{subsec-images}
If $\Zcal$ is of Frobenius-type, the image of a stratum $\Fcal^{(P_0)}_{\Zcal,w}$ via the map $\pi_{P_0}\colon \Fcal_{\Zcal}^{(P_0)}\to \Xcal_\Zcal$ is a union of zip strata of $\Xcal_{\Zcal}$. For a general zip datum $\Zcal$, the image of a flag stratum is not necessarily a union of zip strata, as seen in the following example.

\begin{exa}\label{ex-GL2-image}
Consider the group $G=\GL_{2,K}$, endowed with the trivial cocharacter $\mu$ and the identity isogeny $\varphi=\iden$. Then, $\Xcal_\Zcal$ is the quotient of $G$ by itself acting by conjugation. Let $B\subseteq G$ denote the lower triangular Borel subgroup and $T\subseteq B$ the diagonal torus. Identify $W$ with the group $\left\{ 1,w \right\}$ where $w=\smallmat{0}{1}{1}{0}$. Since $I=\Delta$, we have ${}^I W=\{1\}$. Thus, there is a unique zip stratum, which is $\Xcal$ itself. The image of $\Fcal_{\Zcal,w}^{(B)}$ by the projection $\pi_B\colon \Fcal_{\Zcal}^{(B)}\to \Xcal_{\Zcal}$ is the quotient stack $[E_\Zcal\backslash H]$ where $H$ is the union of conjugacy classes intersecting $Bw$. One checks easily that 
\[H=G\setminus \left\{\left(\begin{matrix}\lambda & 0 \\ 0& \lambda \end{matrix}\right) \ \relmiddle| \ \lambda\in K^* \right\}.\]
In particular, $H$ is strictly contained in $G$.
\end{exa}

However, for general zip data, when $w$ is minimal or cominimal, one has
\begin{equation}\label{pi-image-minimal}
  \pi_{P_0}(\Fcal^{(P_0)}_{\Zcal,w}) =\Xcal_{\Zcal,w}. 
\end{equation}
We also deduce $\pi_{P_0}(\overline{\Fcal}^{(P_0)}_{\Zcal,w})=\overline{\Xcal}_{\Zcal,w}$, by properness of $\pi_{P_0}$. Similarly, if $B\subseteq P_1\subseteq P_0 \subseteq P$ are intermediate parabolic subgroups and $w\in {}^{I_0}W$, then
\begin{equation}\label{pi-image-minimal-relative}
  \pi_{P_1,P_0}(\Fcal^{(P_1)}_{\Zcal,w}) =\Fcal^{(P_0)}_{\Zcal,w}.
\end{equation}
Note that this formula makes sense since ${}^{I_0}W\subseteq {}^{I_1}W$. For Frobenius-type zip data, we have a more precise result:

\begin{proposition}[{\cite[Prop.~2.4.3, Rmk.~2.4.4]{Goldring-Koskivirta-Strata-Hasse}}]\label{prop-FP0}
Assume that $\Zcal$ is of Frobenius-type. Let $w\in {}^I W\cup W^J$. 
\begin{enumerate}
    \item\label{prop-FP0-1} The underlying topological subspace of $\Fcal^{(P_0)}_{\Zcal,w}$ is a point. \label{FP0-point}
    \item\label{prop-FP0-2} The preimage of $\Xcal_{\Zcal,w}$ under the map $\pi_{P_0}\colon \overline{\Fcal}^{(P_0)}_{\Zcal,w} \to \overline{\Xcal}_{\Zcal,w}$ coincides with $\Fcal^{(P_0)}_{\Zcal,w}$.
    \item\label{prop-FP0-3} The map $\pi_{P_0}\colon \Fcal^{(P_0)}_{\Zcal,w}\to \Xcal_{\Zcal,w}$ is finite \'{e}tale.
\end{enumerate}
\end{proposition}
These results also extend to a relative setting. For example, if $B\subseteq P_1\subseteq P_0 \subseteq P$ are parabolic subgroups and $w\in {}^{I_0}W$, then the map $\pi_{P_1,P_0}\colon \Fcal^{(P_1)}_{\Zcal,w} \to \Fcal^{(P_0)}_{\Zcal,w}$ (well-defined by \eqref{pi-image-minimal-relative}) is finite étale \cite[Prop.~4.1.4, 4.2.2]{Goldring-Koskivirta-zip-flags}. Assertion \eqref{prop-FP0-1} amounts to the statement that for $w\in {}^I W\cup W^J$, the set $G_{\Zcal_{P_0},w}$ is a single $E'_{\Zcal,P_0}$-orbit.
%Assertion \eqref{FP0-point} amounts to the statement that the $E_{\Zcal_{P_0}}$-orbit $G_{\Zcal_{P_0},w}$ coincides with a single $E'_{\Zcal,P_0}$-orbit. 
\subsection{\texorpdfstring{The maps $\Pi_\Zcal$ and $\pi_\Zcal$}{}}\label{subsec-maps-pi}
For a general zip datum $\Zcal$ (not necessarily of Frobenius-type) and $w\in {}^{I_0}W$, let $\Pi_{\Zcal,P_0}(w)\subseteq {}^I W$ be the subset of elements $w'\in {}^I W$ such that $\pi_{P_0}(\Fcal^{(P_0)}_{\Zcal,w})$ intersects $\Xcal_{w'}$. This defines a map
\begin{equation}
    \Pi_{\Zcal,P_0} \colon {}^{I_0}W\to \Pcal({}^I W).
\end{equation}
\begin{lemma} \label{lem-unique-maxlen}
There exists a unique element in $\Pi_{\Zcal,P_0}(w)$ of maximal length.
\end{lemma}
\begin{proof}
This is clear when \(\Zcal\) is of Frobenius-type, since all zip strata of $\Xcal$ are pointwise and $\pi_{P_0}(\Fcal^{(P_0)}_{\Zcal,w})$ is locally closed, irreducible in $\Xcal$. We prove this in general. Since $\pi_{P_0}(\Fcal^{(P_0)}_{\Zcal,w})$ is irreducible, there exists a unique \(w' \in \iw\) such that \(\Xcal_{w'} \cap \pi_{P_0}(\Fcal^{(P_0)}_{\Zcal,w})\) is open and dense in \(\pi_{P_0}(\Fcal^{(P_0)}_{\Zcal,w})\). Thus, for any other \(w'' \in \Pi_{\Zcal,P_0}(w)\), we have \(\emptyset \neq \Xcal_{w''} \cap \pi_{P_0}(\Fcal^{(P_0)}_{\Zcal,w}) \subseteq \overline{\Xcal}_{w'}=\bigcup_{w'_1\preccurlyeq_{\Zcal}w'} \Xcal_{w'_1}\). This implies \(w'' \preccurlyeq_\Zcal w'\), which completes the proof.
\end{proof}
We denote by $\pi_{\Zcal,P_0}(w)$ (or simply $\pi_{P_0}(w)$, if the choice of the zip datum $\Zcal$ is clear) the unique element afforded by Lemma \ref{lem-unique-maxlen}. We obtain a map
\begin{equation} \label{map-image-W}
    \pi_{\Zcal, P_0} \colon {}^{I_0}W\to {}^I W.
\end{equation}
For $w\in {}^I W$, we have $\Pi_{\Zcal,P_0}(w)=\{w\}$ and $\pi_{\Zcal,P_0}(w)=w$. Thus, $\pi_{\Zcal,P_0}$ is a section of the inclusion ${}^{I}W\subseteq {}^{I_0}W$. Formula \eqref{pi-image-minimal-relative} immediately implies that the above maps are essentially independent of the choice of $P_0$:
\begin{proposition} \label{Pi-P0}
The maps $\Pi_{\Zcal,P_0}\colon {}^{I_0}W\to \Pcal({}^I W)$ and $\pi_{\Zcal,P_0}\colon {}^{I_0}W\to {}^I W$ coincide with the restrictions to ${}^{I_0}W$ of the maps $\Pi_{\Zcal,B}\colon W\to \Pcal({}^IW)$ and $\pi_{\Zcal,B}\colon W\to {}^IW$ respectively.
\end{proposition}
%\begin{proof}
%For $w\in {}^{I_0}W$, it follows easily from the definition of the zip stratifications that $\pi_{B,P_0}(\Fcal^{(B)}_{\Zcal,w})=\Fcal^{(P_0)}_{\Zcal,w}$, similarly to \cite[Proposition 4.2.2]{Goldring-Koskivirta-zip-flags} in the Frobenius-type case. The result follows immediately.
%\end{proof}
We may thus drop the subscript $P_0$ from the notation and simply denote these maps by $\Pi_\Zcal$ and $\pi_\Zcal$. We emphasize that these maps encode information about the geometry of $\Xcal$, for example the singularities appearing in the closures of EO strata, as we will see later (Theorem \ref{thm-FTZD-normality}). We end this section with a lemma on the set of lower neighbors:

\begin{lemma}[{\cite[Lemma 2.8]{Koskivirta-LaPorta-Reppen-singularities}}]\label{lemma-low-nei}
Assume that $\Zcal$ is of Frobenius-type. For any $I_0\subseteq I$ and $w\in {}^I W$, we have the following.
\begin{enumerate}
    \item \label{lemma-low-nei-item1} $\Gamma_{I_0}(w)\cap {}^I W \subseteq \Gamma_{I}(w)$.
    \item \label{lemma-low-nei-item2} $\Gamma_I(w)\subseteq \pi_{\Zcal}(\Gamma_{I_0}(w))$.
\end{enumerate}
\end{lemma}

\subsection{Canonical parabolic}
To $w\in {}^I W$ we attach a standard parabolic subgroup $P_{w}$ of $G$, following \cite[\S 5.1]{Pink-Wedhorn-Ziegler-zip-data}. We stress that we use the convention $B\subseteq P$, whereas \cite{Pink-Wedhorn-Ziegler-zip-data} uses the convention $B\subseteq Q$. This different convention changes the definition of the canonical parabolic.
%\subsubsection{Definition}
For $w\in {}^I W$, there is a largest algebraic subgroup $L_w$ contained in $L$ such that ${}^{wz^{-1}}\varphi(L_w)=L_w$ (as subgroups of $G$). Moreover, $L_w$ is a standard Levi subgroup by \cite[Proposition 5.6]{Pink-Wedhorn-Ziegler-zip-data}. The standard parabolic subgroup $P_w\colonequals L_w B$ is called the \emph{canonical parabolic} of $w$. For $w\in {}^IW$, write $I_w\subseteq I$ for the type of the canonical parabolic $P_w$. Let also $Q_w\colonequals \varphi(L_w){}^z B$. One has the following key relation (\cite[\S 5.2]{Pink-Wedhorn-Ziegler-zip-data}):
\begin{equation}\label{Gw-withPw}
    G_{\Zcal,w} = E_\Zcal\cdot (P_w wz^{-1}).
\end{equation}
%\subsubsection{Canonical flag strata}
Furthermore, the flag stratum $\Fcal_{\Zcal,w}^{(P_w)}$ satisfies a number of important properties that we collect in the following proposition. We endow all substacks with the reduced induced structure.
\begin{proposition}[{\cite[Proposition 2.15]{Koskivirta-LaPorta-Reppen-singularities}}] \label{prop-Pw}\ 
Assume that $\Zcal$ is of Frobenius-type. For $w\in {}^I W$, the following assertions hold.
\begin{enumerate}
    \item \label{Pw-1} The map $\pi_{P_w}\colon \Fcal^{(P_w)}_{\Zcal,w}\to \Xcal_{\Zcal,w}$ is an isomorphism.
 \item \label{Pw-2} The map $\pi_{P_w}\colon \overline{\Fcal}^{(P_w)}_{\Zcal,w}\to \overline{\Xcal}_{\Zcal,w}$ is proper and birational.
 \item \label{Pw-3} $\Fcal^{(P_w)}_{\Zcal,w}$ is a Bruhat stratum. In particular, $\overline{\Fcal}^{(P_w)}_{\Zcal,w}$ is normal and Cohen--Macaulay.
\end{enumerate}
\end{proposition}
We will show in \S\ref{subsec-flag-strata-char0} similar results for zip data which are not of Frobenius-type. The above proposition is a consequence of the following lemma, that we record for later use. Note that $G_{\Zcal_{P_w},w}=P_w(wz^{-1})Q_w$ by \eqref{Pw-3} of Proposition \ref{prop-Pw} above. For $x\in G$, denote by $\Stab_{E_\Zcal}(x)$ the stabilizer of $x$ in $E_\Zcal$. The first projection $\pr_1\colon E_\Zcal\to P$ induces an isomorphism between $\Stab_{E_\Zcal}(x)$ and a subgroup of $P$. We identify $\Stab_{E_\Zcal}(x)$ as a subgroup of $P$ in this way.

\begin{lemma}\label{stab-contained-Pw}
Assume that $\Zcal$ is of Frobenius-type. For any $w\in {}^I W$ and $x\in G_{\Zcal_{P_w},w}$, one has \[\Stab_{E_\Zcal}(x)\subseteq P_w.\]
\end{lemma}
\begin{proof}
When $x=wz^{-1}$, this follows from \cite[Lemma~2.14]{Koskivirta-LaPorta-Reppen-singularities}. The general case is immediately deduced from this and the fact that $G_{\Zcal_{P_w},w}$ consists of a single $E'_{\Zcal,P_w}$-orbit.
\end{proof}

\subsection{\texorpdfstring{Representatives of elements of $W$}{}}
Assume in this section that $\Zcal$ is of Frobenius-type. Let $w'\in W$ be an arbitrary element. We explain here how to determine the unique element $w\in {}^I W$ such that $w'z^{-1}$ lies in the zip stratum $G_{\Zcal,w}$, i.e.\ such that $w'z^{-1}\in E_\Zcal\cdot(wz^{-1})$. Recall the map $\psi\colon W_I\to W_J$ from \eqref{psi-WIWJ}.
\begin{proposition}[{\cite[\S9.1, Corollary~9.18]{Pink-Wedhorn-Ziegler-zip-data}}]\label{prop-W-representative}
For any $w'\in W$, the following are equivalent:
\begin{enumerate}
    \item $w'z^{-1}$ lies in the zip stratum $G_{\Zcal,w}$.
    \item There exists $a\in W_I$ and $x\in W_{I_w}$ such that $w'=axw\psi(a)^{-1}$.
\end{enumerate}
\end{proposition}

%The equation $w'=axw\psi(a)^{-1}$ can be rewritten as $w'z^{-1}=axwz^{-1}\sigma(a)^{-1}$. 
The above proposition gives an algorithm to determine the unique representative $w\in {}^I W$ such that $w'z^{-1}\in G_{\Zcal,w}$, since the groups $W_I$, $W_{I_w}$ are finite. For $g\in G(K)$, write $\Xi_{\Zcal}(g)$ for the unique element $w\in {}^I W$ such that $gz^{-1}\in G_{\Zcal,w}$. This gives a map
\begin{equation}\label{Xi-map}
    \Xi_{\Zcal}\colon G(K)\to {}^I W.
\end{equation}
%Note that the analogous map defined in the introduction is slightly different (for simplicity, we did not twist it by $z^{-1}$). %\lorenzo{TODO: check that this is still true when we have re-written the intro}.
Proposition \ref{prop-W-representative} gives an algorithm (used later in \S\ref{General algorithm for smoothness}) for the determination of $\Xi_{\Zcal}|_{W}$, where $W$ is viewed as a subset of $G(K)$ via the choice of representatives $\dot{w}$ (\S\ref{subsubsec-zip-strata}). In \S\ref{subsec-algo-GLn}, we compute algorithmically the whole map $\Xi_\Zcal$ in the case $G=\GL_{n,\FF_p}$.

\section{Singularities of zip strata}

\subsection{Previous results}
We review some of the results of \cite{Koskivirta-LaPorta-Reppen-singularities} regarding the singularities that appear in the zip stratification. In this section, most of the results will be stated with the assumption that the zip datum $\Zcal$ is of Frobenius-type, following \loccit Throughout this section, we set $K=\overline{\FF}_p$. We fix a zip datum $\Zcal$ attached to a tuple $(G,\mu,\varphi)$ as in \S\ref{subsubsec-zip-data}. Recall that when $\Zcal$ is of Frobenius-type, $G$ is the base change to $K$ of a reductive group over $\FF_p$, and $\varphi$ is a positive power of the Frobenius homomorphism of $G$. %Write $\Xcal_\Zcal$ for the attached stack of $G$-zips.

\begin{comment}
    \subsubsection{Orbitally finite zip data}
We say that the zip datum $\Zcal_\mu$ is \emph{orbitally finite} (\cite{Pink-Wedhorn-Ziegler-zip-data} ) if the set of $E_\Zcal$-orbits in $G$ is finite. In this case, each zip stratum $G_w$ (for $w\in {}^IW$) coincides with the $E_\Zcal$-orbit of $wz^{-1}$. For example, let $G$ be a connected reductive group over a finite field $\FF_q$ endowed with a cocharacter $\mu\colon \GG_{\textrm{m},k}\to G_k$ (where $k$ is an algebraic closure of $\FF_q$). Let $\varphi\colon G\to G$ be the $q$-power Frobenius homomorphism. Then, the zip datum attached to $(G,\mu,\varphi)$ is orbitally finite by \cite{Pink-Wedhorn-Ziegler-zip-data} .
\end{comment}

\subsubsection{\texorpdfstring{$w$}{w}-Open substacks}
Fix an element $w\in {}^I W$ and let $\Ucal\subseteq \Xcal_\Zcal$ be a locally closed substack. We say that $\Ucal$ is \emph{$w$-open} if it is an open substack of $\overline{\Xcal}_{\Zcal,w}$ (the Zariski closure of $\Xcal_{\Zcal,w}$) which can be written as a union
\begin{equation}
    \Ucal=\bigcup_{w'\in \Gamma_\Ucal} \Xcal_{\Zcal,w'}
\end{equation}
for a certain subset $\Gamma_\Ucal\subseteq {}^I W$. When $\Zcal$ is of Frobenius-type, a $w$-open subset is the same as an open subset of the Zariski closure $\overline{\Xcal}_{\Zcal,w}$, because the underlying topological space of a zip stratum is a single point. For example, $\overline{\Xcal}_{\Zcal,w}$ is a $w$-open substack. An \emph{elementary $w$-open substack} is a substack of the form $\Ucal(w,w')\colonequals \Xcal_{\Zcal,w}\cup \Xcal_{\Zcal,w'}$ where $w'$ is any lower neighbor of $w$ in ${}^I W$ (Definition \ref{low-nb-def}).

\subsubsection{Canonical covers}
Let $w\in {}^I W$ and let $\Ucal\subseteq \Xcal_\Zcal$ be a $w$-open substack. Let $P_w$ be the canonical parabolic attached to $w$. Define a (set-theoretical) subset $\Ucal^{(P_w)}\subseteq \Fcal_\Zcal^{(P_w)}$ as follows:
\begin{equation}
    \Ucal^{(P_w)}\colonequals \bigcup_{w'\in \Gamma_\Ucal} \Fcal^{(P_w)}_{\Zcal,w'}.
\end{equation}
We say that $\Ucal$ admits a \emph{canonical cover} if $\Ucal^{(P_w)}$ is contained in $\overline{\Fcal}^{(P_w)}_{\Zcal,w}$ and is open in it. We say that $\Ucal$ admits a \emph{separating canonical cover} if the preimage of $\Ucal$ by $\pi_{P_w}\colon \overline{\Fcal}^{(P_w)}_{\Zcal,w}\to \overline{\Xcal}_{\Zcal,w}$ coincides with $\Ucal^{(P_w)}$. 
\begin{exa}
    Let $w'\in {}^IW$ be a Bruhat lower neighbor of $w$. Then, the elementary subset $\Ucal(w,w')$ admits a canonical cover (but not necessarily a separating canonical cover).
\end{exa}

\subsubsection{\texorpdfstring{$w$}{w}-Bounded substacks}
Let $\Ucal\subseteq \Xcal_\Zcal$ be a $w$-open substack. We say that $\Ucal$ is \emph{$w$-bounded} if for all $w'\in \Gamma_\Ucal$, we have an inclusion $P_{w'}\subseteq P_w$ between the canonical parabolic subgroups.

\subsubsection{Normality of \texorpdfstring{$w$}{w}-open substacks}
We collect the statements of \cite[\S 4]{Koskivirta-LaPorta-Reppen-singularities} in the following proposition.

\begin{proposition} \label{prop: results about normality} Assume that $\Zcal$ is of Frobenius-type. Let $\Ucal$ be a \(w\)-open substack. 
\begin{enumerate}
    \item If $\Ucal$ is $w$-bounded and admits a canonical cover, then $\pi_{P_w}\colon \Ucal^{(P_w)}\to\Ucal$ is bijective, birational and unramified.
    \item \label{prop-normal-2}Assume that $\Ucal$ is normal and admits a canonical cover. Then, $\Ucal$ is $w$-bounded.
    \item \label{prop-normal-3} Assume that $\Ucal$ is $w$-bounded and admits a separating canonical cover. Then, the map $\pi_{P_w}\colon \Ucal^{(P_w)}\to\Ucal$ is an isomorphism, and $\Ucal$ is normal and Cohen--Macaulay.
\end{enumerate}
\end{proposition}

\subsubsection{The case of elementary $w$-opens}
%In the case of elementary $w$-opens, note that normality is equivalent to smoothness, as normal stacks are smooth in codimension one.
For elementary $w$-opens, we have a stronger result, namely Assertion \eqref{prop-normal-3} admits a converse. Write again $\pi_\Zcal\colon W\to {}^I W$ for the map defined in \S\ref{subsec-maps-pi}.
\begin{theorem}\label{thm-FTZD-normality}
Assume that $\Zcal$ is of Frobenius-type. For $w'\in \Gamma_I(w)$, the following are equivalent:
\begin{enumerate}
    \item \label{thm-normal-1} $\Ucal(w,w')$ is normal.
    \item \label{thm-normal-2} $\Ucal(w,w')$ is smooth.
    \item \label{thm-normal-3} $\Ucal(w,w')$ is $w$-bounded and admits a separating canonical cover.
    \item \label{thm-normal-4} $P_{w'}\subseteq P_w$ and $w'\notin \pi_\Zcal(\Gamma_{I_w}(w)\setminus \{w'\})$.
\end{enumerate}
\end{theorem}
The equivalence between \eqref{thm-normal-1} and \eqref{thm-normal-2} in Theorem \ref{thm-FTZD-normality} above simply follows from the fact that normal schemes are regular in codimension one, and \eqref{thm-normal-4} is simply a reformulation of \eqref{thm-normal-3}.

\subsection{Pointwise and quasi-pointwise strata}\label{sec: qpoint and point}\
%\lorenzo{I tried to generalise this section to remove the hypothesis that \(G\) is split. TODO: check again}
In this section, we assume that $\Zcal$ is of Frobenius-type unless mentioned otherwise. 
\begin{definition}\label{def: quasi pointwise}\
\begin{enumerate}
    \item We say that a flag stratum $\Fcal_{\Zcal,w}^{(P_0)}$ is \emph{pointwise} if its underlying topological space is reduced to a single \(K\)-point.
    \item We say that a flag stratum $\Fcal_{\Zcal,w}^{(P_0)}$ is \emph{quasi-pointwise} if its underlying topological space admits a dense $K$-point.
\end{enumerate}
\end{definition}
By Proposition \ref{prop-FP0}~\eqref{FP0-point}, minimal and cominimal strata are pointwise. Note that the dense point in a quasi-pointwise stratum is necessarily unique, by the irreducibility of $\Fcal_{\Zcal,w}^{(P_0)}$.  Equivalently, $\Fcal_w^{(P_0)}$ is quasi-pointwise if the set $G_{\Zcal_{P_0},w}=E_{\Zcal_{P_0}}\cdot (wz^{-1})$ admits an open dense $E'_{P_0}$-orbit. We show that the property of being pointwise, resp.\ quasi-pointwise, is not affected if we replace \(P_0\) by a smaller parabolic \(P_1\) with \(B\subseteq P_1 \subseteq P_0\). First, we record a lemma about algebraic group actions.

\begin{lemma}\label{lem-transitiveG}
Let $H$ be an algebraic group over $K$, and $X,Y$ two irreducible $K$-varieties equipped with an action of $H$. Let $f\colon X\to Y$ be a $H$-equivariant and generically quasi-finite morphism. Assume that $H$ acts transitively on $Y$. Then, $H$ acts transitively on $X$.
\end{lemma}

\begin{proof}
If $C\subseteq X$ is any $H$-orbit, the restriction of $f$ to $f\colon C\to Y$ is surjective, by $H$-equivariance and the fact that $H$ acts transitively on $Y$. Hence, $\dim(C)\geq \dim(Y)$. Moreover, we have $\dim(X)=\dim(Y)$ since $f$ is generically quasi-finite. It follows that $\dim(C)=\dim(X)$ for any $H$-orbit $C\subseteq X$. Since $X$ is irreducible, the action of $H$ on $X$ is transitive.
\end{proof}

\begin{lemma}
\label{lem: tower invariance of qpointwiseness}
Let $w\in {}^{I_0}W$ be an element and $B\subseteq P_0 \subseteq P$ an intermediate parabolic subgroup.
\begin{enumerate}
    \item $\Fcal^{(P_0)}_{w}$ is pointwise if and only if $\Fcal^{(B)}_{w}$ is pointwise.
    \item $\Fcal^{(P_0)}_{w}$ is quasi-pointwise if and only if $\Fcal^{(B)}_{w}$ is quasi-pointwise.
\end{enumerate}
\end{lemma}

\begin{proof}
 The map $\pi_{B,P_0}\colon \Fcal^{(B)}\to \Fcal^{(P_0)}$ satisfies $\pi_{B,P_0}(\Fcal^{(B)}_{w})=\Fcal^{(P_0)}_{w}$, see \eqref{pi-image-minimal-relative}. Thus, if $\Fcal^{(B)}_{w}$ is pointwise (resp.\ quasi-pointwise), then so is $\Fcal^{(P_0)}_{w}$. Assume now that $\Fcal^{(P_0)}_{w}$ is pointwise. The map $\pi_{B,P_0}\colon \Fcal^{(B)}_{w} \to \Fcal^{(P_0)}_{w}$ is finite étale (by the relative version of Proposition \ref{prop-FP0}~\eqref{prop-FP0-3}, explained below the proposition).
 %: working with the zip datum $\Zcal_{P_0}$, finite by \cite[Prop.~3.2.2]{stratif-funct}, étale by \cite[Prop.~2.2.1.(2)]{Koskivirta-Normalization}, % (where we take \(\Zcal = \Zcal_{P_0}\) and \(P_0 = B\)), a base change given by \cite[Prop.~3.1.4]{stratif-funct} and the cancellation law for étale morphisms. %(\cite[\href{https://stacks.math.columbia.edu/tag/02GW}{Lemma 02GW}]{stacks-project}).
Hence, the same is true for the map of schemes $\pi_{B,P_0}\colon H^{(B)}_{w} \to H^{(P_0)}_{w}$. By Lemma \ref{lem-transitiveG}, $E_\Zcal$ acts transitively on $H^{(B)}_{w}$, hence $\Fcal^{(B)}_w$ is pointwise. Finally, assume that $\Fcal^{(P_0)}_{w}$ is quasi-pointwise. %Then, \(H^{(B)}_{w}\) contains an open dense \(E_\Zcal\)-orbit \(U\). Since the morphism $\pi_{B,P_0}\colon H^{(B)}_{w} \to H^{(P_0)}_{w}$ is étale, it is also open. In particular, \(\pi_{B, P_0}(U)\) is an open dense \(E_\Zcal\)-orbit in \(H^{(P_0)}_{w}\). This shows that \(\Fcal^{(P_0)}_{w}\) is quasi-pointwise. Conversely, assume that \(\Fcal^{(P_0)}_{w}\) is quasi-pointwise and 
Let \(V \subseteq H^{(P_0)}_{w}\) be the unique dense \(E_\Zcal\)-orbit and let $V'\subseteq H^{(B)}_{w}$ denote the preimage of $V$ by $\pi_{B, P_0}$. Applying Lemma \ref{lem-transitiveG} to \(\pi_{B, P_0} \colon V' \to V\), we conclude that \(\Fcal^{(B)}_{w}\) is quasi-pointwise.
\end{proof}

Lemma \ref{lem: tower invariance of qpointwiseness} lets us use the terminology "pointwise", "quasi-pointwise" unambiguously to refer to elements of $W$. Assume now that the torus $T$ splits over $\FF_q$, some finite extension of \(\FF_p\). This assumption implies that all zip strata $\Xcal_w$ and all flag strata $\Fcal^{(P_0)}_w$ are defined over $\FF_q$.

\begin{lemma} \label{lem-dense}
The following are equivalent:
\begin{enumerate}
    \item The stratum $\Fcal_{\Zcal,w}^{(P_0)}$ is quasi-pointwise.
    \item There exists a point $a\in G_{\Zcal_{P_0},w}(\FF_q)$ such that the $E'_{\Zcal,P_0}$-orbit of $a$ is dense in $G_{\Zcal_{P_0},w}$.
\end{enumerate}
\end{lemma}

\begin{proof}
The second condition clearly implies the first one. Conversely, assume that $\Fcal_{\Zcal,w}^{(P_0)}$ is quasi-pointwise. Denote by $x\in \Fcal_{\Zcal,w}^{(P_0)}$ the unique dense point. Since the Galois group $\Gal(K/\FF_q)$ acts on $\Fcal_{\Zcal,w}^{(P_0)}$, we must have $\sigma(x)=x$ by unicity of $x$, where $\sigma$ is the $q$-power Frobenius element. This shows that $x$ is defined over $\FF_q$. Now, $x$ corresponds to an $E'_{\Zcal,P_0}$-orbit $C\subseteq G_{\Zcal_{P_0},w}$ defined over $\FF_q$. Lemma \ref{lem: lang and rational points} shows $C(\FF_q)\neq \emptyset$, which concludes the proof. 
\end{proof}

\begin{lemma}\label{lem: lang and rational points}
Let $H$ be a connected algebraic group over $\FF_q$ acting transitively on an $\FF_q$-variety $X$. Then, $X(\FF_q)$ is nonempty.
\end{lemma}
\begin{proof}
This follows from the triviality of \(H^1(\Gal(K/\FF_q), H(K))\). Let us explain in further detail. Fix a point $a\in X(K)$. By assumption, we can write $\sigma(a)=h_\sigma\cdot a$ for some $h_\sigma\in H(K)$. Since \(H^1(\Gal(K/\FF_q), H(K))\) is trivial, we can find $g\in H(K)$ such that $\sigma(g)^{-1}g=h_\sigma$. Thus $\sigma(g)^{-1}g \cdot a = h_\sigma\cdot a=\sigma(a)$, hence $\sigma(g\cdot a)=g\cdot a$. This shows that $g\cdot a\in X(\FF_q)$.
\end{proof}

\subsection{Hasse invariants}
Assume that $\Zcal$ is of Frobenius-type. Recall that any character $\lambda\in X^*(L)$ gives rise to a line bundle $\Lcal(\lambda)$ on $\Xcal$, see for instance \cite[\S3.1]{Goldring-Koskivirta-Strata-Hasse}. For $w\in {}^I W$, a \emph{Hasse invariant} for the stratum $\Xcal_w$ is a section $\Ha_w\in H^0(\overline{\Xcal}_w,\Lcal(\lambda))$ whose non-vanishing locus coincides with $\Xcal_w$. When $\Zcal$ is of Frobenius-type, such Hasse invariants exist in most cases:

\begin{theorem}\label{thm-Hasse}
Assume that $\varphi$ is the $p^m$-power Frobenius isogeny of $G$, for some $m\geq 1$. Assume that $G$ admits a representation with central kernel $\rho\colon G\to \GL_N$ such that $\rho\circ \mu$ has $d$ different weights, and that $p^m\geq d$. Then Hasse invariants exist for all zip strata of $\Xcal$.
\end{theorem}
\begin{proof}
Although \cite[Corollary 6.2.1]{Goldring-Koskivirta-zip-flags} states a slightly weaker result, the statement follows immediately from the proof of \loccit
\end{proof}

In particular, Hasse invariants exist for zip data coming from abelian type Shimura varieties. For an arbitrary reductive $\FF_p$-group $G$, they exist when $p^m$ is large enough. Under the assumption of the existence of Hasse invariants, we show that the image of a quasi-pointwise stratum is given by a single zip stratum.

\begin{lemma}\label{lem-qpoint-single-im}
Assume that all zip strata in $\Xcal$ admit Hasse invariants. Let \(w \in W\) be a quasi-pointwise element and set \(w_1 = \pi_\Zcal(w) \in {}^IW\). %Suppose that there exists a character \(\lambda \in X^\ast(L)\) and a section \(h_{w_1} \in H^0(\overline{\Xcal}_{w_1}, \Lcal(\lambda))\) whose non-vanishing locus is exactly \({\Xcal}_{w_1}\). 
Then, \(\Pi_\Zcal(w) = \{w_1\}\).
\end{lemma}

\begin{proof} 
Let $\Ha_{w_1}\in H^0(\overline{\Xcal}_{w_1},\Lcal(\lambda))$ be a Hasse invariant for $w_1$. Replacing \(\lambda\) with \(N\lambda\) and \(\Ha_{w_1}\) with \(\Ha_{w_1}^N\), for some \(N > 0\) if necessary, we can assume that \(H^0(\Xcal_{\Zcal_{B}, w}, \Lcal(\lambda))\) is one-dimensional (\cite[Theorem 3.1]{Koskivirta-compact-hodge}). Since \(\Fcal_{w}^{(B)}\) admits a dense \(K\)-point, \(H^0(\Fcal_{w}^{(B)}, \Lcal(\lambda))\) is at most one-dimensional. Thus, since \(\pi_{B}^\ast(\Ha_{w_1}) \in H^0(\Fcal_{w}^{(B)}, \Lcal(\lambda))\) is non-zero, we have \(\dim_K(H^0(\Fcal_{w}^{(B)}, \Lcal(\lambda))) = 1\). On the other hand, there exists a nonzero section \(s \in H^0(\Xcal_{\Zcal_{B}, w}, \Lcal(\lambda))\), hence by dimensionality $\psi_{B}^{*}(s)=\pi_{B}^\ast(\Ha_{w_1})$ up to a nonzero scalar. If there were some \(w_2 \in \Pi_\Zcal(w) \setminus\{w_1\}\), then \(\pi_{B}^\ast(\Ha_{w_1})\) would vanish on some points of \(\Fcal_{w}^{(B)}\). On the other hand, \(s \in H^0(\Xcal_{\Zcal_{B}, w}, \Lcal(\lambda))\) is nowhere-vanishing on $\Xcal_{\Zcal_{B}}$, hence \(\psi_{B}^{*}(s)\) is nowhere-vanishing on \(\Fcal_{w}^{(B)}\). This is a contradiction.
%Let \(s \in H^0(\Xcal_{\Zcal_{B}, w}, \Lcal(\lambda))\) be any non-zero section, which must therefore be nowhere-vanishing on \(\Xcal_{\Zcal_{B}, w}\). Then, \(\psi_{B}^{*}(s) \in H^0(\Fcal_{w}^{(B)}, \Lcal(\lambda))\) is nowhere-vanishing too and, by dimensionality, it must be a non-zero multiple of \(\pi_{B}^\ast(\Ha_{w_1})\), which must also be nowhere-vanishing on \(\Fcal_{w}\). 
Therefore, we must have \(\Pi_\Zcal(w) = \{w_1\}\). 
\end{proof}

\subsection{Small strata}\label{subsec-small-strata}
In this section, we define the important notion of small strata. We assume that $\Zcal$ is of Frobenius-type unless mentioned otherwise. We first show the following key inequality.

\begin{lemma}\label{lem-length-ineq}
For any $w\in W$, we have $\ell(\pi_{\Zcal}(w))\leq \ell(w)$.
\end{lemma}
\begin{proof}
Set $w_1\colonequals \pi_{\Zcal}(w)$. Then, $\pi_{B}$ induces a map $\Fcal_{\Zcal,w}^{(B)}\to \overline{\Xcal}_{\Zcal,w_1}$. Since $\Zcal$ is of Frobenius-type, $\pi_B(\Fcal_{\Zcal,w}^{(B)})$ is a union of zip strata. Hence $\Xcal_{\Zcal, w_1}$ is contained in the image of the map $\pi_B\colon \Fcal_{\Zcal,w}^{(B)}\to \overline{\Xcal}_{\Zcal,w_1}$. We obtain similarly a map of schemes $\pi_{B}\colon H^{(B)}_w\to \overline{G}_{w_1}$ whose image contains $G_{w_1}$. By comparing dimensions, we have $\dim(G_{w_1})\leq \dim(H^{(B)}_w)$, hence $\ell(w_1)\leq \ell(w)$ by \eqref{dim-H}. 
\end{proof}

\begin{definition}\label{def-small}
For a general zip datum $\Zcal$, we say that $w\in W$ is small if $\ell(\pi_{\Zcal}(w)) = \ell(w)$.
\end{definition}

We use the terminology \emph{$\Zcal$-small} if there is an ambiguity on the choice of the zip datum. We write $W^{\rm sm}\subseteq W$ for the subset of small elements. Assume that $\Zcal$ is of Frobenius-type and that $w$ is small. Set $w_1\colonequals \pi_{\Zcal}(w)$. For any intermediate parabolic $B\subseteq P_0 \subseteq P$, we obtain a dominant morphism $\pi_{P_0} \colon \Fcal_{\Zcal,w}^{(P_0)}\to \overline{\Xcal}_{\Zcal,w_1}$ between stacks of the same dimension. Hence, this map is generically quasi-finite.

\begin{proposition} \label{prop-same-l}
If $w\in W$ is small, the stratum $\Fcal^{(B)}_{\Zcal,w}$ is quasi-pointwise. 
\end{proposition}
%\lorenzo{Does this argument show that \(\Fcal^{(P_{w_1})}_{\Zcal, w}\) contains a dense open isomorphic copy (a section of the projection) of \(\Xcal_{\Zcal, w_1}\), when \(w \in W\) is small and \(\pi(w) = w_1\)?}\JS{maybe not "isomorphic", but maybe finite \'{e}tale over it}

%\lorenzo{this bit must be reformulated and possibly moved elsewhere}\JS{Which part?} Moreover, if \(\Fcal^{(P_0)}_{\Zcal,w}, \Xcal_{\pi_{P_0, P}(w)}\) and \(\Xcal_{\pi_{P_0, P}(w)}(k)\) are defined over a subfield \(\kappa \leq k\) such that \(k/\kappa\) is Galois, then the unique dense \(k\)-point of \(\Fcal^{(P_0)}_{\Zcal,w}\) is defined over \(\kappa\).
%\JS{I changed the formulation and assumptions, sorry! I moved the density result to Lemma \ref{lem-dense}. Also, I changed your $\kappa$ to $\FF_p$. If we want to have a general finite field, we can change later the setting to $\kappa$ everywhere.}

\begin{proof}
%\lorenzo{A lot of the material discussed here should be in the section about \(\GF\).}\JS{I added a remark about this}
Recall that \(\Fcal^{(B)}_{\Zcal,w}\) is identified with \([E_\Zcal \backslash H^{(B)}_{\Zcal,w}]\).
%Write \(H_w = \tilde{\gamma}^{-1}(C_w), C_w = BwBz^{-1},\) as in \Cref{ssec: equivalent def of zip strata}. The inclusion \(G \to G\times P/B, g \mapsto (g, B),\) induces a natural isomorphism \(\Fcal^{(P_0)} \cong [E\backslash(G\times P/B)]\). This identifies \(\Fcal^{(B)}_w\) with \([E \backslash H_w]\). The \(E\)-quotient of the projection \(G\times P/B \to G\) is identified to \(\pi \colon \Fcal^{(B)} \to \Xcal\). Thus, 
The first projection $\pr_1\colon G\times (P/B)\to G$ yields a proper, surjective, $E_\Zcal$-equivariant map $\pr_1\colon \overline{H}^{(B)}_{\Zcal,w}\to \overline{G}_{\Zcal,w_1}$.
%pullback of \(\pi_{B}\) via \(\overline{G}_{w_1} \to \overline{\Xcal}_{w_1}\) is proper and its restriction \(\pi \colon \overline{H}_w^{(B)} \to \overline{G}_{w_1}\) to \(\overline{H}_w^{(B)}\) is proper and surjective.
%Moreover, by the hypotheses on the length, \(\dim(H_w) = \dim(\overline{G}_{w_1}) = \dim({G}_{w_1})\). 
Let $x\in H^{(B)}_{\Zcal,w}$ be any point which maps to a point of $G_{\Zcal,w_1}$ and let $S\colonequals E_\Zcal\cdot x$ be the orbit of $x$. By $E_\Zcal$-equivariance, the map $\pr_1\colon S\to G_{\Zcal,w_1}$ is surjective, hence $\dim(S)\geq \dim(G_{\Zcal,w_1})=\dim(H^{(B)}_{\Zcal,w})$, by our assumption that $w$ is small. Hence $S$ must be open dense in $H^{(B)}_{\Zcal,w}$. It corresponds to an open dense point in $\Fcal^{(B)}_{\Zcal,w}$. \qedhere

\end{proof}

Choose a finite extension $\FF_q$ of \(\FF_p\) over which $T$ splits. If $w$ is small, Lemma \ref{lem-dense} shows the existence of $a\in G_{\Zcal_{P_0},w}(\FF_q)$ such that the $E'_{\Zcal,P_0}$-orbit of $a$ is dense in $G_{\Zcal_{P_0},w}$. Moreover, the image of $a$ by $\pi_{P_0}$ must be contained in the stratum $\Xcal_{w_1}$. In other words, we have
\begin{equation}
    a\in (G_{\Zcal_{P_0},w} \cap G_{\Zcal,w_1})(\FF_q).
\end{equation}
Note that \(G_{\Zcal_{B},w} = BwBz^{-1}\). The above discussion also shows the following corollary:
\begin{corollary}\label{coro-small-Fp}
For $w\in {}^{I_0}W$, the following are equivalent:
\begin{enumerate}
    \item $w$ is small.
    \item There exists $w_1\in {}^{I}W$ such that $\ell(w)=\ell(w_1)$ and $G_{\Zcal_{P_0},w} \cap G_{\Zcal,w_1}\neq \emptyset$.
    \item There exists $w_1\in {}^{I}W$ such that $\ell(w)=\ell(w_1)$ and $\left(G_{\Zcal_{P_0},w} \cap G_{\Zcal,w_1}\right)(\FF_q)\neq \emptyset$.
    \item There exists $w_1\in {}^{I}W$ such that $\ell(w)=\ell(w_1)$ and $\left(BwBz^{-1} \cap G_{\Zcal,w_1}\right)(\FF_q)\neq \emptyset$. 
\end{enumerate}
\end{corollary}

%By Proposition \ref{prop-same-l}, we deduce:
%\begin{corollary}\label{cor-fp-points}
%If $\Ucal(w,w')$ does not admit a separating canonical cover, then there exists $v\in \Gamma_{I_w}(w)\setminus\{w'\}$ such that $\left(BvBz^{-1} \cap G_{\Zcal,w'}\right)(\FF_q)\neq \emptyset$.
%\end{corollary}

We end this section with a lemma on the set of small lower neighbors. Let $B\subseteq P_0 \subseteq P$ be an intermediate parabolic subgroup, of type $I_0\subseteq I$. For $w\in {}^{I_0}W$, put $\Gamma_{I_0}^{\rm sm}(w)\colonequals \Gamma_{I_0}(w)\cap W^{\rm sm}$.

\begin{lemma}\label{lem-sm-lownb}
For any $w\in {}^{I}W$, one has $|\Gamma^{\rm sm}_{I_0}(w)|\geq |\Gamma_I(w)|$. In particular, $\Gamma^{\rm sm}_{I_0}(w)\neq\emptyset$ unless $w=1$.
\end{lemma}

\begin{proof}
Consider the proper map $\pi_{P_0}\colon \overline{\Fcal}^{(P_0)}_{w}\to \overline{\Xcal}_w$. This map is surjective and maps $\Fcal^{(P_0)}_{w}$ to $\Xcal_w$. By the inequality of Lemma \ref{lem-length-ineq}, we deduce that for any lower neighbor $w'\in \Gamma_I(w)$, there exists an element $w_1\in \Gamma_{I_0}(w)$ such that the image of the stratum $\Fcal^{(P_0)}_{w_1}$ intersects $\Xcal_{w'}$. In particular, $w_1$ is small. Since $w'$ is uniquely determined by $w_1$ (because $\pi_{P_0}(\overline{\Fcal}^{(P_0)}_{w_1})=\overline{\Xcal}_{w'}$), we obtain the desired inequality.
\end{proof}

\subsection{Exponent raising}\label{sec-expo-raising}
When Hasse invariants exist on all zip strata, we can combine Proposition \ref{prop-same-l} with Lemma \ref{lem-qpoint-single-im} to conclude that $\Pi_\Zcal(w)=\{w_1\}$ where $w_1=\pi_\Zcal(w)$. In this section, we use what we call the ``exponent raising method'' to show that the assumption on the existence of Hasse invariants is superfluous.

Let $\Zcal$ be a Frobenius-type zip datum. Fix an integer $r\geq 1$ such that $T$ splits over $\FF_{p^r}$. We then have $\varphi^r(L)=L$. Hence, for any integer $m\equiv 1\pmod{r}$, we have $\varphi^m(L)=M$. Thus, we can consider the zip datum $\Zcal_m\colonequals(G,P,Q,L,M,\varphi^m)$ for each such $m\geq 1$.

\begin{lemma}\label{lem-same-stratum-m}
Let $x\in G$ and $w\in {}^I W$. Assume that $x\in G_{\Zcal,w}$. Then there exists $s\geq 1$ such that $x\in G_{\Zcal_m,w}$ for all $m\equiv 1\pmod{rs}$.
\end{lemma}
\begin{proof}
We can write $x=awz^{-1}b^{-1}$ for some $(a,b)\in E_{\Zcal}$. Let $s\geq 1$ such that $\overline{a}\in L$ is defined over $\FF_{p^s}$. We thus have $\varphi^m(\overline{a})=\overline{b}$ for all $m\equiv 1\pmod{rs}$, hence $(a,b)\in E_{\Zcal_m}$. The result follows.
\end{proof}

\begin{proposition}\label{prop-same-stratum-m}
Let $w\in W$ and $w_1\in \Pi_{\Zcal}(w)$. There exists $r_1\geq 1$, multiple of $r$, such that $w_1\in \Pi_{\Zcal_m}(w)$ for all $m\equiv 1\pmod{r_1}$. 
\end{proposition}

\begin{proof}
By assumption, $\Xcal_{\Zcal,w_1}$ intersects $\pi_{\Zcal,B}(\Fcal^{(B)}_{\Zcal,w})$. Hence there exists $x\in G_{\Zcal,w_1}\cap (BwBz^{-1})$. By Lemma \ref{lem-same-stratum-m}, there exists $r_1$, a multiple of $r$, such that $x\in G_{\Zcal_m,w_1}\cap (BwBz^{-1})$ for all $m\equiv 1\pmod{r_1}$. It follows that $\Xcal_{\Zcal_m,w_1}$ intersects the image of $\Fcal^{(B)}_{\Zcal_m,w}$. The result follows.
\end{proof}

\begin{corollary}\label{coro-small-m}
If $w\in W$ is $\Zcal$-small, there exists $r_1\geq 1$, multiple of $r$, such that $w$ is $\Zcal_m$-small for all $m\equiv 1\pmod{r_1}$.    
\end{corollary}

We obtain the following result, available for all zip data of Frobenius-type $\Zcal$, without any assumption on the existence of Hasse invariants.

\begin{proposition}\label{prop-small-image}
If $w\in W$ is $\Zcal$-small, then the following hold:
\begin{enumerate}
    \item\label{prop-small1} $\Pi_\Zcal(w)=\{w_1\}$ for $w_1=\pi_\Zcal(w)$.
    \item\label{prop-small2} $\Fcal_{\Zcal,w}^{(B)}$ is pointwise.
\end{enumerate}
\end{proposition}

\begin{proof}
Let $w_2\in \Pi_{\Zcal}(w)$. By combining Proposition \ref{prop-same-stratum-m} and Corollary \ref{coro-small-m}, there exists $r_2\geq 1$ multiple of $r$, such that for all $m\equiv 1 \pmod{r_2}$, the element $w$ is $\Zcal_m$-small and one has $w_1, w_2\in \Pi_{\Zcal_m}(w)$. In particular, $w$ is quasi-pointwise with respect to $\Zcal_m$ and $w_1=\pi_{\Zcal_m}(w)$. Choosing $m$ large enough, we can apply Lemma \ref{lem-qpoint-single-im} to deduce $w_2=w_1$. For the second assertion, note that \eqref{prop-small1} gives a map $\pi_{\Zcal,B}\colon H_{\Zcal,w}^{(B)}\to G_{\Zcal,w_1}$. Since $w$ is small, this map is generically quasi-finite. By Lemma \ref{lem-transitiveG}, we deduce that $E_\Zcal$ acts transitively on $H_{\Zcal,w}^{(B)}$, hence $\Fcal^{(B)}_{\Zcal,w}$ is pointwise.
\end{proof}

\subsection{Characterizations of smallness in terms of roots}\label{subsec-small-roots}
%\lorenzo{do we need Frobenius-type here?}\JS{Yes, I think so. In the general case, the compact $w$-cycles (i.e.\ cycles that only involve compact roots) have to be accounted for. Intuitively, I would say that in the case $\varphi=\id$, the dimension of the stabilizer is "number of compact $w$-cycles" + number of $w$-sequences of type $\NC^+ \xrightarrow{C} \NC^-$. The former does not account for positive dimension in the Frobenius case because $\varphi^m(x)=x$ has only finitely many solutions when $\varphi=$ Frobenius}
Assume here that \(\Zcal\) is of Frobenius-type. Recall that we write $\sigma$ for the action of the corresponding Frobenius homomorphism on \(W\) and \(\Phi\). We use the following terminology: elements of $\Phi_L$ (resp.\ $\Phi\setminus \Phi_L$) are called \emph{compact} (resp.\ \emph{non-compact}) roots. 
%In what follows, the letter "$\C$" will stand for "compact", and "$\NC$" will mean "non-compact". Moreover, we add the superscript $+$ (resp.\ $-$) to denote positive (resp.\ negative) roots. For example, $\NC^+$ denotes non-compact positive roots, i.e.\ elements of $\Phi^+\setminus \Phi^+_L$. Any of the symbols $\C,\NC,\C^+,\C^-,\NC^+,\NC^-, \Phi^+,\Phi^-$ will be called types of roots.
Fix $w\in W$. A sequence of roots $S=(\beta_0,\dots,\beta_n)$ is called a \emph{$w$-sequence} if
\begin{enumerate}
    \item $\beta_{i+1}=w \sigma(\beta_i)$ for all $0\leq i \leq n-1$,
    \item $\beta_0$ is non-compact positive.
    \item $\beta_n$ is non-compact.
    \item $\beta_1,\dots, \beta_{n-1}$ are compact.
\end{enumerate}
For each positive non-compact root $\beta$, there is a unique $w$-sequence such that $\beta_0=\beta$, we denote it by $S_w(\beta)$. We say that $S_w(\beta)$ is \emph{positive} (resp.\ \emph{negative}) if the last root in the sequence $\beta_n$ is positive (resp.\ negative). Write $\Scal_w^{+}$ (resp.\ $\Scal_w^-$) for the set of positive (resp.\ negative) $w$-sequences. Note that $|\Scal_w^+|+|\Scal_w^-|$ is the number of positive non-compact roots, i.e.\ $\dim(R_{\textrm{u}}(Q)) =\dim(R_{\textrm{u}}(P))$. For any $w_1\in W_I$ and $w\in W$, one has $S_{w_1w\sigma(w_1)^{-1}}(w_1 \beta)=w_1 S_w(\beta)$ (note that $w_1\beta$ is again positive non-compact since $w_1\in W_I$). It follows that
\begin{equation}\label{Splus-eq}
\Scal^+_{w_1w\varphi(w_1)^{-1}}=w_1\Scal^+_w
\end{equation}
and similarly for negative sequences. We will use repeatedly the fact that an element $w\in W$ lies in ${}^I W$ if and only if $w^{-1}(\Phi^+_L)\subseteq \Phi^+$.

%Let $T_1,T_2,T_3$ be three types of roots. We say that a $w$-sequence of roots $S=(\beta_0,\dots,\beta_n)$ has type $T_1 \xrightarrow{T_2} T_3$ if it satisfies the following conditions:
%\begin{enumerate}
%    \item $\beta_0$ is of type $T_1$,
%    \item $\beta_n$ is of type $T_3$,
%    \item $\beta_1,\dots, \beta_{n-1}$ are of type $T_2$.
%\end{enumerate}
%For example, we will consider $w$-sequences of type $\NC^+ \xrightarrow{\C} \NC^-$.

\begin{lemma}\label{lem-cpt-neg}
If $w=w'z^{-1}$ with $w'\in {}^I W$, then all compact roots appearing in a $w$-sequence are negative.
\end{lemma}

\begin{proof}
Let $(\beta_0,\dots,\beta_n)$ be a $w$-sequence. We may assume $n\geq 2$. Note that $w'^{-1}\beta_1=z^{-1}\sigma(\beta_0)=w_0\sigma(w_{0,I}\beta_0)$. Since $\beta_1$ is compact and $w'\in {}^I W$, the roots $w'^{-1}(\beta_1)$ and \(\beta_1\) have the same sign. Since $\beta_0$ is positive non-compact, so is $w_{0,I}\beta_0$. It follows that $\beta_1$ is negative. Assume now that $\beta_i\in \Phi^-_L$ with $i+1<n$. Again $w'^{-1}(\beta_{i+1})=w_0\sigma(w_{0,I}\beta_i)$. Since $\beta_i$ is compact negative, $w'^{-1}(\beta_{i+1})$ is negative, and then so is $\beta_{i+1}$ since $w'\in {}^I W$.
\end{proof}

%Write $\Stab_E(g)$ for the stabilizer of an element $g\in G$ with respect to the action of $E\colonequals E_\Zcal$ on $G$.

\begin{lemma}\label{lem-dim-stab-1}
For any $w\in W$, the dimension of $\Stab_E(w)$ is equal to $|\Scal_w^-|$.
\end{lemma}
\begin{proof}
Write $w=w'z^{-1}$ for $w'\colonequals wz$. Assume first that $w'\in {}^IW$ (i.e.\ that $w'$ is minimal). In this case, $\pi_B(\Fcal^{(B)}_{w'})=\Xcal_{w'}$, and the map $\pi_B\colon \Fcal^{(B)}_{w'}\to \Xcal_{w'} \simeq [1/\Stab_E(w)]$ is finite (Proposition \ref{prop-FP0}). Moreover, $\Fcal^{(B)}_{w'}=[E'\backslash Bw'Bz^{-1}]$, hence $\dim(\Stab_E(w))=\dim(E')-\dim(Bw'Bz^{-1})=\dim(R_{\textrm{u}}(P))-\ell(w')$. Thus, it suffices to show that $\ell(w')=|\Scal^+_{w}|$. Since $w'\in {}^IW$, one has $w'^{-1}(\Phi^+_L)\subseteq \Phi^+$. Thus, if the roots $w'^{-1}(\beta)$ and $\beta$ have different signs, then $\beta$ is non-compact. By Lemma \ref{lem-cpt-neg}, all compact roots appearing in a $w$-sequence are negative. Let $(\beta_0,\dots,\beta_n)$ be a $w$-sequence (with $n\geq 1$). Then, $w'^{-1}(\beta_n)=w_0\sigma(w_{0,I}\beta_{n-1})$ is negative (even for $n=1$). Moreover, any positive non-compact root is either the tail of a positive $w$-sequence or the opposite of the tail of a negative $w$-sequence. This shows that the tails of positive $w$-sequences are exactly the positive roots whose sign is changed by $w'^{-1}$. Thus $\ell(w')=\ell(w'^{-1})=|\Scal^+_w|$. % \JS{Note that any positive non-compact root is either the tail of a positive $w$-sequence or the opposite of a negative $w$-sequence. This shows that the tails of positive $w$-sequences are exactly the positive roots whose sign is changed by $w'^{-1}$ Thus $\ell(w')=\ell(w'^{-1})=|\Scal^+_w|$.} Conversely, if \(\beta \in \Phi\) is a root such that \(w'^{-1}(\beta)\in \Phi^-\), consider the unique \(\beta\) such that \(\beta = w\sigma(\beta)\). Then, \(\sigma(\beta) = \sigma(w_{0, I})w_0w'^{-1}(\beta)\) and can find a unique \(\gamma \in \Phi^+\) such that \(\beta = w_{0, I}(\gamma)\). In particular, \(\beta\) is positive if and only if it is non-compact. 

Now, for general $w$, let $v'\in {}^I W$ be the unique element such that $w\in G_{v'}$, and set $v\colonequals v'z^{-1}$. By Proposition \ref{prop-W-representative}, we can write $w=w_1xv\sigma(w_1)^{-1}$ for some $x\in W_{I_{v'}}$ and $w_1\in W_I$. Since $\dim(\Stab_E(w))=\dim(\Stab_E(v))$, it suffices to show that $|\Scal^-_w|=|\Scal^{-}_{v}|$. The transformation $w\mapsto w_1w\sigma(w_1)^{-1}$ does not change the number of negative sequences, by \eqref{Splus-eq}. It remains to check that $|\Scal^-_{xv}|=|\Scal^{-}_{v}|$. Note that the tail of the $v$-sequence $S_v(\beta)$ is the root $v\sigma(v)\cdots \sigma^{n-1}(v)\sigma^{n}(\beta)$ for some $n\geq 1$. Since $v\sigma(I_{v'})=I_{v'}$, we have $v\sigma(x)=yv$ for some $y\in  W_{I_{v'}}$. This implies (by induction on $n$) that we can write
\begin{equation}
    (xv)\sigma(xv)\cdots \sigma^{n-1}(xv) = y v\sigma(v)\cdots \sigma^{n-1}(v)
\end{equation}
for some $y\in W_{I_{v'}}$. We deduce that the $v$-sequence $S_v(\beta)$ is negative if and only if the $xv$-sequence $S_{xv}(\beta)$ is negative. This terminates the proof.
\end{proof}

\begin{remark}

Let $X_w$ denote the set of positive non-compact roots $\beta$ such that the $w$-sequence $S_w(\beta)$ is negative. For $\beta\in X_w$ and $u\in U_\beta$, let $m\geq 1$ be the smallest integer such that $w^m (\beta)$ is non-compact. Consider the pair
\begin{equation}
 \epsilon_\beta(u)\colonequals \left((wuw^{-1})(w^2\varphi(u)w^{-2})\cdots (w^{m}\varphi^{m-1}(u)w^{-m}), u(w\varphi(u)w)\cdots (w^{m-1}\varphi^{m}(u)w^{-(m-1)}) \right).
\end{equation}
It is easy to see that $\epsilon_\beta(u)\in \Stab_E(w)^\circ$. Let $U_X\colonequals \prod_{\beta\in X} U_\beta$. The above construction yields a map $U_X\to \Stab_E(w)^\circ$. We believe that this map is bijective (but we will not need this fact).
\end{remark}

%\JS{I find this surprising, does $\pi_Q(\Stab_E(w))$ decompose as a product of $U_\beta$'s ?} 
%\lorenzo{I think it should. I wanted to rely on \cite[Thm.~8.1]{Pink-Wedhorn-Ziegler-zip-data} to conclude this, but I don't think it is trivial to show that theorem applies in this case. I reckon that this decomposition still holds, though. I'd have to think about this. Does it sound reasonable to you?}\JS{I don't have any intuition. I thought it could sit "diagonally" in the product, if you understand what I mean, i.e.\ one component is $x$, and some other component $x^p$ for example - this would not decompose as a product then. I have to look at some examples.} \lorenzo{Maybe you are right.}
%\lorenzo{Actually, I think this decomposition works by induction using \cite[Prop.~4.11]{Pink-Wedhorn-Ziegler-zip-data}. I can write down the argument inside the proof, if you want.}
%\JS{Thanks. I am computing an example right now.}

%\begin{lemma}\label{lem-dim-stab-2}
%For any $w\in W$, the dimension of the group $\Stab_{E'}(w)$ is equal to the number of $w$-sequences of type $\NC^+ \xrightarrow{C^-} \NC^-$.
%\end{lemma}
%\begin{proof}
%The proof is analogous to that of Lemma \ref{lem-dim-stab-1}.
%\end{proof}

As a consequence of Lemma \ref{lem-dim-stab-1}, we find that for any $w\in W$,
\[\dim(E\cdot (wz^{-1}))=\dim(E)-\dim(\Stab_E(wz^{-1}))=\dim(P)+|\Scal^+_{wz^{-1}}|.\]

\begin{lemma}
For any $w\in W$, we have $|\Scal^+_{wz^{-1}}|\leq \ell(w)$. %\lorenzo{TODO: double check}
\end{lemma}

\begin{proof}
For any positive $wz^{-1}$-sequence $(\beta_0,\dots,\beta_n)$ there is at least one  $0 \leq i < n$ such that $z^{-1}(\sigma(\beta_i))$ and $\beta_{i+1}=wz^{-1}(\sigma(\beta_i))$ have different signs. The result follows.
\end{proof}

\begin{proposition}\label{prop-small-by-chains}
Let $w\in W$. The following are equivalent:
\begin{enumerate}
    \item The element $w$ is small.
    \item \label{prop-small-by-chains4} $|\Scal^+_{wz^{-1}}|=\ell(w)$.
\end{enumerate}    
\end{proposition}
\begin{proof}
By Proposition \ref{prop-small-image}~\eqref{prop-small2}, $w$ is small if and only if the dimension of $\Fcal^{(B)}_w$ coincides with that of the stack $[E\backslash E\cdot (wz^{-1})]\simeq [1/\Stab_E(wz^{-1})].$ By Lemma \ref{lem-dim-stab-1}, this amounts to $\ell(w)+\dim(B)-\dim(E')=-|\Scal^-_{wz^{-1}}|$, which is equivalent to $\ell(w)=\dim(R_{\textrm{u}}(P))-|\Scal^-_{wz^{-1}}|=|\Scal^+_{wz^{-1}}|$.
\end{proof}

\begin{remark}\label{rmk-2conditions-small}
Using the characterization of Proposition \ref{prop-small-by-chains}, one easily checks that $w\in W^{\rm sm}$ if and only if the following conditions are satisfied:
\begin{enumerate}
    \item In any $wz^{-1}$-sequence $(\beta_0,\dots,\beta_n)$, there is no $1\leq i <n$ such that $\beta_i$ is positive and  $\beta_{i+1}$ is negative.
%    \item If $(\beta_0,\dots,\beta_n)\in \Scal_{wz^{-1}}^+$, then there exists $1\leq i \leq n$ such that $\beta_1,\dots, \beta_{i}$ are negative and $\beta_{i+1},\dots,\beta_{n}$ are positive.
%    \item If $(\beta_0,\dots,\beta_n)\in \Scal_{wz^{-1}}^-$, then all roots $\beta_1,\dots, \beta_{n-1}$ are negative.
    \item If $C$ is an orbit under the operator $wz^{-1}\sigma$ entirely contained in $\Phi_L$, all roots in $C$ have the same sign. 
\end{enumerate}
\end{remark}

\begin{remark}\label{rmk-algo}
Consider \(w \in W\). Proposition \ref{prop-small-by-chains}.(\ref{prop-small-by-chains4}) gives an algorithm to check whether \(w\) is small or not. If we write \(n = \rank(G)\), its complexity is \(\lvert \Phi\rvert = O(n^2)\).
\end{remark}

%\lorenzo{Do we still need this definition, given that we have the one you gave above?}
%Define \emph{a negative compact chain} as a sequence of roots $\Ccal=(a_0,...,a_n)$ with $n>1$ such that
%\begin{enumerate}
%    \item $a_0\in \Phi^+$.
 %   \item $a_1,..., a_{n-1}\in \Phi^{-}_L$
%\item $a_{i+1}= zw^{-1}(a_i)$ for each $0\leq i <n$.
%\item $a_n\notin \Phi^{-}_L$.
%    \end{enumerate}

\section{Decidability of singularities}
In this section we consider an arbitrary zip datum $\Zcal$ of Frobenius-type. % Let \((B, T, z),\) satisfying the hypotheses of Section \ref{sec-split}.
%In particular, recall that \(\varphi\) is the Frobenius isogeny and \(T\) is the base change to $K=\overline{\FF}_p$ of a maximal torus defined over \(\FF_p\). 
Let \(w \in \iw, \ w' \in \Gamma_\Zcal(w),\) and consider the elementary \(w\)-open substack \(\Ucal(w, w') \subseteq \Xcal_\Zcal\). By Theorem \ref{thm-FTZD-normality}, we know that \(\Ucal(w, w')\) is smooth if and only if it is both \(w\)-bounded and admits a separating canonical cover. 

\subsection{General algorithm for smoothness}
\label{General algorithm for smoothness}

%\lorenzo{The general algorithm for smoothness requires a result (Proposition \ref{prop-small-single-im}) that comes from the section on the split case, which should not come \emph{after} this!}\JS{thanks, fixed}
We start with a proposition which gives a more precise statement than Theorem \ref{thm-FTZD-normality}~\eqref{thm-normal-4}. We first explain the connection between small strata and separating canonical covers of zip strata. Assume that the elementary substack $\Ucal(w,w')$ does not admit a separating canonical cover. By definition, this means that there exists an element $v\in \Gamma_{I_w}(w)\setminus\{w'\}$ such that $\pi_{P_w}(\Fcal^{(P_w)}_v)$ intersects $\Xcal_{w'}$. We must necessarily have $\pi_{\Zcal}(v)=w'$ by Lemma \ref{lem-length-ineq}. In particular,
\begin{equation}\label{eq-length}
    \ell(\pi_{\Zcal}(v))=\ell(w')=\ell(w)-1=\ell(v).
\end{equation}
Hence $v$ is small. This implies:

\begin{proposition}\label{prop-smooth-small}
The following are equivalent:
\begin{enumerate}
    \item \label{prop-smooth-small-1} $\Ucal(w,w')$ is smooth (equivalently, normal).
   \item \label{prop-smooth-small-2} $P_{w'}\subseteq P_w$ and $w'\notin \pi_\Zcal(\Gamma^{\rm sm}_{I_w}(w)\setminus \{w'\})$.
\end{enumerate}
\end{proposition}

\begin{proof}
It is clear that \eqref{prop-smooth-small-1} implies \eqref{prop-smooth-small-2}. Conversely, if the substack $\Ucal(w,w')$ does not admit a separating canonical cover, by the discussion above, then there exists $w_1\in \Gamma^{\rm sm}_{I_w}(w)\setminus \{w'\}$ such that $w'=\pi_{\Zcal}(w)$. This shows that \eqref{prop-smooth-small-2} implies \eqref{prop-smooth-small-1}.
\end{proof}

We show that both conditions of Proposition \ref{prop-smooth-small}~\eqref{prop-smooth-small-2} are algorithmically decidable.

\subsubsection{\texorpdfstring{\(w\)}{w}-Boundedness}\label{wbounded-algo}
%Recall that \(\Ucal(w, w')\) is \(w\)-bounded if and only if \(P_{w'} \subseteq P_w\). 
Since $P_w$, $P_{w'},$ are standard parabolic subgroups, the containment $P_{w'}\subseteq P_w$ is satisfied if and only if \(I_{w'} \subseteq I_w\). These finite sets can be computed explicitly by the formula
\begin{equation}
I_w=\bigcap_{m\geq 0} \varphi^m_w(I)    
\end{equation}
where $\varphi_w$ is the operator $wz^{-1}\sigma$, see \cite[\S~2.2.6]{Koskivirta-LaPorta-Reppen-singularities}, \cite[Prop.~5.6]{Pink-Wedhorn-Ziegler-zip-data}.

\subsubsection{Images of small elements}
We explain the procedure to determine whether the second condition $w'\notin \pi_\Zcal(\Gamma^{\rm sm}_{I_w}(w)\setminus \{w'\})$ in Proposition \ref{prop-smooth-small} is satisfied. Recall the map $\Xi_\Zcal\colon G(K)\to {}^IW$ defined in \eqref{Xi-map}. Proposition \ref{prop-W-representative} gives an algorithm to determine the restriction $\Xi_\Zcal|_W$. An immediate consequence of Proposition \ref{prop-small-image} is the following:
\begin{proposition}\label{prop-pi-Xi}
For $w\in W^{\rm sm}$, one has $\pi_\Zcal(w)=\Xi_\Zcal(w)$.
\end{proposition}
Since $\Xi_\Zcal|_W$ is algorithmically computable, so is $\pi_\Zcal(w)$ for $w\in W^{\rm sm}$. The algorithm to decide if an elementary substack admits a separating canonical cover is the following:

%\(\Ucal(w, w')\) admits a separating canonical cover or not and we then show that it is effective.
\begin{algorithm}\label{gen-algo}\

    \begin{enumerate}
    \item Enumerate \(\Gamma_{I_w}(w)\).
    \item Determine which elements of $\Gamma_{I_w}(w)$ are small, using the characterization explained in \S\ref{subsec-small-roots}. This gives the set $\Gamma^{\rm sm}_{I_w}(w)$.
    \item For each \(w'' \in \Gamma^{\rm sm}_{I_w}(w)\setminus \{w'\}\), compute $\pi_\Zcal(w'')=\Xi_\Zcal(w'')$ (Proposition \ref{prop-pi-Xi}). This gives the set $\pi_\Zcal(\Gamma^{\rm sm}_{I_w}(w)\setminus \{w'\})$.
    %If \(w_1 = w\), then \(\Ucal(w, w')\) does not admit a separating canonical cover and the algorithm terminates.
    \item Check if $w'\in \pi_\Zcal(\Gamma^{\rm sm}_{I_w}(w)\setminus \{w'\})$.
    %If \(w_1 \neq w\) for all \(w'' \in \Gamma_{I_w}(w)\setminus \{w'\}\), then \(\Ucal(w, w')\) admits a separating canonical cover and the algorithm terminates.
\end{enumerate}
\end{algorithm}
\begin{theorem}\label{thm: algo for smooth}
Algorithm \ref{gen-algo} gives an effective procedure. Combined with \S\ref{wbounded-algo}, this gives an effective algorithm to determine whether $\Ucal(w,w')$ is smooth (equivalently, normal).
\end{theorem}

\subsection{Dictionary for \texorpdfstring{\(\GL_n\)}{general linear groups}}\label{subsec-algo-GLn}
%\JS{How to pass between F,V and $g\in G$. I just copy-paste the relevant part from my previous article}
We explain the connection between $\GL_{n,\FF_p}$-zips and Dieudonn\'{e} spaces, following computations from \cite{Koskivirta-Normalization}. Let $\varphi$ be the $p$-power isogeny of $\GL_{n,\FF_p}$.

\subsubsection{\texorpdfstring{Dieudonn\'{e} spaces and $F$-zips}{}}\label{subsec Dieudonne zips}

We consider the case $G=\GL_{n,\FF_p}$ with a cocharacter of type $(r,s)$, where $r+s=n$. Let $S$ be the special fiber of a unitary Shimura variety of signature \((r, s)\) at $p$ and assume that splits in the associated imaginary quadratic field $\mathbf{E}$. If $A$ is the underlying abelian variety of a point of $S$, then $A$ is endowed with an action of $\Ocal_{\mathbf{E}}$ and a compatible polarization. Write $\DD(A)$ for the Dieudonne module over $W(K)$ of $A$, and $\overline{\DD}(A)\colonequals \DD(A)\otimes_{W(K)} K$ for the associated Dieudonne space.

The $\Ocal_{\mathbf{E}}$-action yields a splitting $\DD(A)=\DD(A)_\tau\oplus \DD(A)_{\overline{\tau}}$, where $\Gal(\mathbf{E}/\QQ)=\{\tau,\overline{\tau}\}$. Since $p$ is split in $\mathbf{E}$, each subspace is stable under the action of the operators $F,V$ (Frobenius and Verschiebung), and similarly for $\overline{\DD}(A)=\overline{\DD}(A)_\tau\oplus \overline{\DD}(A)_{\overline{\tau}}$. We set $D\colonequals \overline{\DD}(A)_\tau$, which is an $n$-dimensional $K$-vector space endowed with a $\sigma$-linear endomorphism $F \colon D\to D$, a $\sigma^{-1}$-linear endomorphism $V\colon D\to D$ satisfying the conditions:
\begin{enumerate}[(1)]
\item $\Ker(F)=\Im(V)$,
\item $\Ker(V)=\Im(F)$,
\item $\rk(F)=r$, $\rk(V)=s$.
\end{enumerate}
The triple $(D,F,V)$ completely characterizes the Ekedahl--Oort stratum of a point. We call such a triple a \emph{Dieudonne space of signature $(r,s)$}.

\subsubsection{$E$-orbits}\label{subsub-E-orbits}
On the other hand, recall that the Ekedahl--Oort strata of $S$ are parametrized by the $E$-orbits in $G_K=\GL_{n,K}$, where $E$ is the zip group attached to the cocharacter $\mu\colon \GG_{\textrm{m},K}\to G_K$, $t\mapsto \diag(t\mathbf{1}_r,\mathbf{1}_s)$. We explain how these two interpretations are related. We consider the $K$-valued points of all geometric objects, but we drop $K$ from the notation for simplicity.

Let $\Mat^{(d)}_n$ denote the set of matrices in $\Mat_n$ of rank $d$. After choosing a $K$-basis of $D$, we may write $F=a\otimes \sigma$ and $V=b\otimes \sigma^{-1}$, where $(a,b)$ lies in the set
\[\Pscr:=\{(a,b)\in \Mat^{(r)}_n\times \Mat^{(s)}_n, \ a\sigma(b)=\sigma(b)a=0\}.\]
Note that for $(a,b)\in \Pscr$, we have $\Ker(a)=\Im(\sigma(b))=\sigma(\Im(b))$ and $\Im(a)=\Ker(\sigma(b))=\sigma(\Ker(b))$.

Two pairs $(a,b)$ and $(a',b')$ in $\Pscr$ yield isomorphic $F$-zips if and only if there is $M\in \GL_n$ such that
\begin{equation}
a'=Ma\sigma(M)^{-1} \quad \textrm{ and } \quad b'=Mb\sigma^{-1}(M)^{-1}.
\end{equation}
This defines an action of $\GL_n$ on $\Pscr$ and we obtain a bijection between isomorphism classes of Dieudonne spaces of signature $(r,s)$ and $\GL_n$-orbits in $\Pscr$. Let $(e_1,...,e_n)$ the canonical basis of $K^n$ and define $W_1\colonequals \Span(e_1,...,e_{r})$ and $W_2\colonequals\Span(e_{r+1},...,e_n)$. The parabolic subgroups attached to $\mu$ are given by $P:=\Stab(W_2)$, $Q:=\Stab(W_1)$, with common Levi subgroup $L:=P\cap Q$. Consider the set
\begin{equation}
\Pscr_1:=\{ (a,b)\in \Pscr, \ \Ker(a)=W_2\}.
\end{equation}
The action of $\GL_n$ on $\Pscr$ restricts to an action of $P$ on $\Pscr_1$ and the inclusion $\Pscr_1 \subset \Pscr$ induces a bijection between $P$-orbits in $\Pscr_1$ and $\GL_n$-orbits in $\Pscr$. Next, we define a natural map
\begin{equation}\label{Phi-map}
   \Phi \colon \GL_n\to \Pscr_1.
\end{equation}
Let $f\in \GL_n$ be a matrix, viewed as an automorphism of $K^n$. For $i=1,2$ denote by $p_i\colon K^n\to W_i$ the projection onto $W_i$ with respect to the decomposition $K^n=W_1\oplus W_2$. Set also $W'_i\colonequals f(W_i)$. Let $q_i\colon K^n\to W'_i$ denote the projection with respect to the decomposition $K^n=W'_1\oplus W'_2$, i.e.\ $q_i=f\circ p_i\circ f^{-1}$. Define $a,b\colon K^n\to K^n$ by
\begin{align*}
    &a\colonequals f\circ p_1 \\
    &b\colonequals {}^{\sigma^{-1}}(f^{-1}\circ q_2)={}^{\sigma^{-1}}(p_2\circ f^{-1}).
\end{align*}
It is an easy verification that $(a,b)\in \Pscr_1$. This gives a map $\Phi$ as in \eqref{Phi-map}. By \cite[Lemma 4.1.1]{Koskivirta-Normalization}, $\Phi$ induces a bijection 
\begin{equation}
\Phi \colon \GL_n/R_{\textrm{u}}(Q)\to \Pscr_1.    
\end{equation}
Moreover, let $P$ act on the left on $\GL_n/R_{\textrm{u}}(Q)$ by the rule $x\cdot (gR_{\textrm{u}}(Q))\colonequals xg\varphi(\overline{x})^{-1}R_{\textrm{u}}(Q)$, where $\overline{x}\in L$ denotes the Levi projection $P\to L$ modulo $R_{\mathrm{u}}(P)$. Note that this group action is well-defined. By construction, the $P$-orbits in $\GL_n/R_{\textrm{u}}(Q)$ correspond to the $E$-orbits in $\GL_n$. We obtain (\cite[Proposition 4.1.2]{Koskivirta-Normalization})

\begin{proposition}\label{prop-E-orb-Y}
The map $\Phi$ is $P$-equivariant and induces a bijection
\begin{equation}
E\backslash \GL_n \longrightarrow  P\backslash \Pscr_1.
\end{equation}
Therefore, the above construction yields a bijection between isomorphism classes of Dieudonne spaces of signature $(r,s)$ and $E$-orbits in $\GL_n$.
\end{proposition}

For $f\in \GL_n$, let $\Phi(f)=(a,b)$ and write $D_f\colonequals (K^n,F,V)$ for the corresponding Dieudonne space of signature $(r,s)$, where $F=a\otimes \sigma$ and $V=b\otimes \sigma^{-1}$. By the above, $f_1,f_2,$ lie in the same $E$-orbit if and only if $D_{f_1}\simeq D_{f_2}$ in the category of Dieudonne spaces.

\subsection{The canonical filtration}
Let $(D,F,V)$ be a Dieudonne space of signature $(r,s)$. If $M\subseteq D$ is a subspace, we consider the two subspaces $F(M)$ and $V^{-1}(M)$ obtained by applying $F$ and taking inverse image under $V$. There is a unique coarsest filtration of $D$ by subspaces which are $F$- and $V^{-1}$-stable. It is called \emph{the canonical filtration of $D$} and is obtained by applying all finite combinations of $F,V^{-1},$ to the flag $0\subset D$. %The stabilizer of the canonical filtration coincides with the canonical parabolic $P_w$, where $w\in {}^I W$ corresponds to the isomorphism class of $(D,F,V)$ via the correspondence obtained from Proposition \ref{prop-E-orb-Y} and the parametrization of $E$-orbits by ${}^I W$.
%Denote the set of subspaces appearing in this filtration by $\Ccal(D)$.

\begin{comment}
    
The following lemma makes it possible to compute the canonical filtration of $D$ directly in terms of the matrix $f$:

\begin{lemma}
For any subspace $W\subset K^n$, one has the following relations:
\begin{align}
V(W)&=W_2 \cap \left(\sigma^{-1}(f^{-1}W)+W_1\right) \\
F^{-1}(W)&=W_2+\left(\sigma^{-1}(f^{-1}W)\cap W_1\right).
\end{align}
\end{lemma}
\end{comment}

\subsection{\texorpdfstring{The map $\Xi_\Zcal$}{}}\label{sub-determine-w} %\lorenzo{TODO: are we sure that everything we do here is effectively computable? Some of the sub-steps (computing the rel.\ pos.\ of flags) are probably computable, but I don't know references for algorithms. }\JS{It only depends on the dimensions of the intersections}
We give an algorithm to compute the full map
\begin{equation}
  \Xi_\Zcal \colon \GL_n(K)\to {}^I W.
\end{equation}
Let $f\in \GL_n(K)$ and set $\Phi(f)=(a,b)$. Consider the associated Dieudonne space $(K^n,F,V)$, where $F=a\otimes \sigma$ and $V=b\otimes \sigma^{-1}$. Let $\Fcal_f$ be the canonical filtration of $D_f$, written as
\begin{equation}
    0\subsetneq D_{f,1} \subsetneq \dots \subsetneq D_{f,h-1} \subsetneq D_{f,h}\colonequals D_f.
\end{equation}
Note that this filtration is obtained after at most $n$ iterations of the operators $V,F^{-1},$ applied successively, starting with the flag $0\subseteq D_f$. We consider the relative position of the flag $\Fcal_f$ with the flag $\Gcal_f$ given by $0\subseteq V(D_f) \subseteq D_f$. Specifically, denote by $S(\Fcal_f)$ and $S(\Gcal_f)$ the parabolic subgroups of $\GL_{n}$ defined as the stabilizers of the filtrations $\Fcal_f$, $\Gcal_f$ respectively. Choose $x,y\in \GL_{n}$ such that ${}^x B\subseteq S(\Fcal_f)$ and ${}^y B\subseteq S(\Gcal_f)$. Consider the double quotient $P\backslash \GL_n /B$. The natural inclusion ${}^I W\subseteq \GL_n$ gives rise to a bijection ${}^I W\simeq P\backslash \GL_n /B$. For a matrix $M\in \GL_n$, write $[M]\in {}^I W$ for its representative in ${}^I W$. The following is an easy consequence of the properties of the canonical filtration:
\begin{proposition} \ 
\begin{enumerate}
    \item The image of the element $y^{-1}x$ in the double quotient $P\backslash \GL_n /B$ is independent of all choices.
    \item One has $\Xi_\Zcal(f)=[y^{-1}x]$.
\end{enumerate}
\end{proposition}
The relative position of $\Fcal_f$, $\Gcal_f$ only depends on the integers $\dim(V(D_f)\cap D_{f,i})$ for $i=1,\dots, h$. In particular, it can easily be computed algorithmically.

\subsection{Example: length \texorpdfstring{$2$}{two} strata of \texorpdfstring{$\GL_{n}$}{general linear groups}} \label{subsec-GUrs} %\lorenzo{TODO: there's an issue with notations; here \(\alpha_i\) denotes the roots of \(\Phi\) with the usual indices; in \S\ref{subsec-small-roots} \(\alpha_i\) denotes the \(i\)-th element of a \(w\)-sequence. We should fix this.}\JS{Good point. I changed all $\alpha$'s to $\beta$'s in \S\ref{subsec-small-roots}}
The smoothness of length one strata was studied in full in \cite{Koskivirta-LaPorta-Reppen-singularities} for general Hodge-type Shimura varieties. Here, we illustrate the results of the present paper by studying the $2$-dimensional strata in Shimura varieties attached to simple unitary groups at a split prime of good reduction. This corresponds to the reductive group $G=\GL_{n,\FF_p}$. Denote by $(r,s)$ the signature of the unitary group, where $r+s=n$. Write $\mu\colon \GG_{\textrm{m}}\to G$ for the cocharacter $\mu(t)=\diag(t\textbf{1}_r,\textbf{1}_s)$. We may assume $r\geq s$, since the stacks of $G$-zips corresponding to $(r,s)$ and $(s,r)$ are isomorphic.
\begin{enumerate}
    \item When $s=0$, the stack $\Xcal$ is a single point since ${}^I W=\{1\}$.
    \item When $s=1$, it was proved in \cite[Proposition 5.5]{laporta2023generalised} that the Zariski closure of any EO stratum is smooth.
\end{enumerate}
We therefore assume that $s\geq 2$ in the following. Let $B\subseteq P$ denote the lower-triangular Borel subgroup and $T$ the diagonal torus. Identify characters of $T$ with $\ZZ^n$ in the usual way. Let $(e_1,\dots, e_n)$ be the standard basis of $\ZZ^n$. The simple roots are given by $\alpha_i \colonequals e_i-e_{i+1}$ for $1\leq i \leq n-1$, and $\Delta \setminus I =\{\alpha_r\}$. There is a unique element of length one in ${}^I W$, namely $w'\colonequals s_{\alpha_r}$. There are exactly two elements in ${}^IW$ of length $2$, namely:
\begin{equation}\label{w12-eq}
    w_1\colonequals \left( \begin{matrix}
        \textbf{1}_{r-1} &&&& \\
       &0&0&1& \\
       &1&0&0& \\
       &0&1&0& \\
       &&&&\textbf{1}_{s-2} \\
    \end{matrix} \right) \quad , \quad
     w_2\colonequals \left( \begin{matrix}
        \textbf{1}_{r-2} &&&& \\
       &0&1&0& \\
       &0&0&1& \\
       &1&0&0& \\
       &&&&\textbf{1}_{s-1} \\
    \end{matrix} \right).
\end{equation}
One has $w_1=s_{\alpha_{r}} s_{\alpha_{r+1}}$ and  $w_2=s_{\alpha_{r}} s_{\alpha_{r-1}}$. We obtain two elementary subsets $\Ucal_i\colonequals \Ucal(w_i,w')$ for $i=1,2$. We determine below the pairs $(r,s)$ for which $\Ucal_i$ is smooth. Since $w'$ is the only lower neighbor of $w_i$ in ${}^IW$, the substack $\Ucal_i$ is smooth if and only if it is $w_i$-bounded and $\Gamma^{\rm sm}_{I_{w_i}}(w_i) = \Gamma_I(w_i)=\{w'\}$ (Proposition \ref{prop-smooth-small} or Corollary \ref{cor-closure} below). A key property is the following relation:
\begin{equation}\label{rel-s-alpha}
    s_{\alpha_{r+1}}(\alpha_r) = s_{\alpha_{r}}(\alpha_{r+1}).
\end{equation}

For $i=1,2$, the element $w_i$ has two Bruhat lower neighbors in $W$, namely $\{w',w'_i\}$, where $w'_1\colonequals s_{\alpha_{r+1}}$ and $w'_2\colonequals s_{\alpha_{r-1}}$. Therefore, the set $\Gamma_{I_{w_i}}(w_i)$ has at most two elements by Lemma \ref{lemma-low-nei}~\eqref{lemma-low-nei-item2}. Since $w'$ is a Bruhat lower neighbor of $w_i$, it is also a lower neighbor of $w_i$ with respect to $\preccurlyeq_{I_i}$, which shows that
$w'\in \Gamma^{\rm sm}_{I_{w_i}}(w_i)$. The relevant strata are represented in the diagram below. The connecting lines represent the Bruhat order $\leq$.
\begin{equation}
\xymatrix{
&w_1 \ar@{-}[dr] \ar@{-}[dl] &&w_2 \ar@{-}[dr] \ar@{-}[dl]& \\
w'_1 \ar@{-}[drr]&&w'\ar@{-}[d]&&w'_2 \ar@{-}[dll] \\
&&\iden&&
}
\end{equation}

For an integer $k\in \ZZ$ which is prime to $n$, write $\inv_n(k)$ for the unique integer $k'\in \{1,\dots, n-1\}$ such that $kk'\equiv 1 \pmod{n}$. 

\begin{proposition}\label{prop-small}
Let $i\in \{1,2\}$. The following are equivalent:
\begin{enumerate}
    \item $w'_i$ is small.
    \item $\gcd(r,s)=1$.
\end{enumerate}
\end{proposition}

\begin{proof}
%We prove the result for $i=1$, the proof is completely similar in the case $i=2$. 
It suffices to check whether $w'_i$ satisfies the two conditions (1), (2) of Remark \ref{rmk-2conditions-small}. Note that the only compact root whose sign is changed by $w'_1z^{-1}$ (resp.\ $w'_2z^{-1}$) is $\alpha_{r-s+1}$ (resp.\ $\alpha_{r-s-1}$, if \(r>s+1\)). %\JS{The sentence begins with "compact root", so the case $r=s+1$ needs to be done separately}.
(If \(r = s+1\) and \(i = 2\), \(\alpha_{r-1} = w'_2z^{-1}(e_1-e_n), e_1-e_n \in \Phi\setminus \Phi_L,\) and \(w'_2z^{-1}\) does not change the sign of any compact root.) Hence, if $w'_1$ (resp.\ $w'_2$) is small, then the $w'_1z^{-1}$-orbit ($w'_2z^{-1}$-orbit) of $\alpha_{r-s+1}$ (resp.\ $\alpha_{r-s-1}$, if \(r>s+1\); \(e_1-e_n\), if \(r = s + 1\)) is not contained in $\Phi_L$. Conversely, if this condition is satisfied, then condition (2) of Remark \ref{rmk-2conditions-small} holds. Moreover, the sign of any non-compact root is changed by $w'_iz^{-1}$ (except when \(r = s+1\) and \(i = 2\)). It follows that condition (1) is satisfied by the $w'_1z^{-1}$-sequence (resp.\ $w'_2z^{-1}$-sequence) containing $\alpha_{r-s+1}$ (resp.\ $\alpha_{r-s-1}$ or \(e_1-e_n\)). Thus it is satisfied by all $w'_iz^{-1}$-sequences. We have shown that $w'_1$ (resp.\ $w'_2$) is small if and only if the $w'_1z^{-1}$-orbit (resp.\ $w'_2z^{-1}$-orbit) containing $\alpha_{r-s+1}$ (resp.\ $\alpha_{r-s-1}$ or \(e_1-e_n\)) is not contained in $\Phi_L$.

Assume $i=1$.
We claim that if this orbit is contained in $\Phi_L$, then $\gcd(r,s)>1$. For a contradiction, assume $\gcd(r,s)=1$ and set $m\colonequals \inv_n(s)$. Note that $w'_1z^{-1}(\alpha_i)=\alpha_{i+s}$ for all compact simple root $\alpha_i$, except for $i \in \{r-s,r-s+1,r-s+2\}$. Let $k_0$ be the smallest positive integer $k$ such that $r+(k-1)s+1\in \{r-s,r-s+1,r-s+2\}$, where these integers are taken modulo $n$. Since $\gcd(r,s)=1$, we have either $r+(k_0-1)s+1=r-s$ or $r+(k_0-1)s+1=r-s+2$, which amounts to $k_0=n-m$ and $k_0=m$, respectively (if $r+(k_0-1)s+1=r-s+1$, then \(k_0 = n\), which is impossible). In the first case, we obtain a non-compact root in the orbit (since \(w_1'z^{-1}(\alpha_{r-s}) = \alpha_r + \alpha_{r+1}\)), hence a contradiction. It follows that $k_0=m<\frac{n}{2}$ and the root $-\alpha_{r-s+2}$ lies in the orbit. The next roots in the sequence are $-(\alpha_{r+1}+\alpha_{r+2})$ and $-(\alpha_1+\alpha_2)=e_3-e_1$. The following roots in the sequence are obtained by adding $s$ to each index, until one of the integers $1+ks$ or $3+ks$ lies in $\{r-s+1,r-s+2\}$ (taken modulo $n$). The smallest integer for which this happens is $k=\min\{m,n-2m\}-2$. If $m>\frac{n}{3}$, then $k=n-2m-2$ and $3+ks=r-s+1$. The root $e_{r-s+1}-e_{r-s-1}$ lies in the orbit. The next root is $e_{r+2}-e_{r-1}$, which is non-compact, a contradiction. It follows that $m<\frac{n}{3}$. Continuing this way, we get $m<\frac{n}{d}$ for any $d=2,3,4,\dots$, which is impossible. This proves the claim.

Conversely, assume that $\gcd(r,s)>1$. We show that the $w'_1z^{-1}$-orbit containing $\alpha_{r-s+1}$ is contained in $\Phi_L$. Let $k$ be the smallest integer such that either $1+ks$ or $2+ks$ lies in $\{r-s+1,r-s+2\}$ (taken modulo $n$). Since $\gcd(r,s)>1$, we must have $1+ks\equiv r-s+1 \pmod{n}$. We deduce that the sequence takes once more the value $-\alpha_{r-s+1}$ without having passed through any non-compact roots. Hence, it is entirely contained in $\Phi_L$. 
The case \(i = 2\) is proved similarly.
\end{proof}

To simplify notation, we identify $I$ with a subset of $\{1,\dots, n-1\}$ via the map $i\mapsto \alpha_i$.

\begin{proposition} \label{Pw-leng2}
Assume that $\gcd(r,s)=1$. Set $m\colonequals \inv_n(s)$.
\begin{enumerate}
\item If $m>\frac{n}{2}$, then one has 
\[I_{w_1}=\{1+ks \ | \ 0\leq k \leq n-m-2\}\cup \{r+1\} \quad \textrm{and} \quad I_{w_2}=\emptyset.\]
\item If $m< \frac{n}{2}$, then 
\[I_{w_1}=\emptyset \quad \textrm{and} \quad I_{w_2}=\{-1+ks \ | \ 1\leq k \leq m-1\}\cup \{r-1\}.\]
\end{enumerate}
\end{proposition}
\begin{proof}
We only consider the element $w_1$, the proof for $w_2$ is similar. The computation is analogous to the one carried out in \cite[\S 5.2.1]{Koskivirta-LaPorta-Reppen-singularities}. Identify $\Delta$ with the set $\{1,\dots,n-1\}$ via the bijection $i\mapsto \alpha_i$, and view it as a subset of $\ZZ/n\ZZ$ via the map $i\mapsto i \pmod{n}$. Define a map $\gamma_1\colon \ZZ/n\ZZ\to \ZZ/n\ZZ$ by
\begin{equation}
    \gamma_1(k)=\begin{cases}
0 & \quad \textrm{if } k\in \{r-s-1,r-s+1,r-s+2,r\}, \\
r+1 & \quad \textrm{if } k=r-s, \\
k+s & \quad \textrm{otherwise.}
    \end{cases}
\end{equation}
Similarly to \loccitn, $I_{w_1}\cup \{0\}=\bigcap_{m\geq 1}\ima(\gamma^m_1)$. Consider the sequence $u_k=r+1+ks$ for $k\geq 0$. Let $u_{k_0}$ be the first element of the sequence such that $r-s-1\leq u_{k_0}\leq r-s+2$. It follows that:
\begin{enumerate}
    \item If $u_{k_0}=r-s$, then $I_{w_1}=\{u_k \ | \ 0\leq k \leq k_0 \}=\{u_k \ | \ 1\leq k \leq k_0 \}\cup \{r+1\}$.
    \item Otherwise, $I_{w_1}=\emptyset.$
\end{enumerate}
The smallest integer $k$ such that $u_k=r-s$ is $n-m-1$. The smallest integer $k$ such that $u_k=r-s-1$ is $n-2m-1$ (if $m<\frac{n}{2}$) or $2n-2m-1$ (if $m>\frac{n}{2}$). The smallest integer $k$ such that $u_k=r-s+1$ is $n-1$. The smallest integer $k$ such that $u_k=r-s+2$ is $m-1$. Hence, $I_{w_1}$ is nonempty if and only if $n-m<m$, i.e.\ $m>\frac{n}{2}$. The result follows.
\end{proof}

\begin{corollary}\label{cor: length 2 small calc}
Assume that $\gcd(r,s)=1$.  Set $m\colonequals \inv_n(s)$.
\begin{enumerate}
\item If $m>\frac{n}{2}$, then $\Gamma^{\rm sm}_{I_{w_1}}(w_1)=\{w'\}$ and $\Gamma^{\rm sm}_{I_{w_2}}(w_2)=\{w',w'_2\}$.
\item If $m<\frac{n}{2}$, then $\Gamma^{\rm sm}_{I_{w_1}}(w_1)=\{w',w'_1\}$ and $\Gamma^{\rm sm}_{I_{w_2}}(w_2)=\{w'\}$.
\end{enumerate}   
\end{corollary}

\begin{proof}
We carry out the proof for $w_1$, the argument for $w_2$ is analogous. When $m<\frac{n}{2}$, we have $I_{w_1}=\emptyset$, hence $\Gamma^{\rm sm}_{I_{w_1}}(w_1)=\Gamma^{\rm sm}(w_1)=\{w',w'_1\}$ by Proposition \ref{prop-small} and the assumption $\gcd(r,s)=1$. Assume now that $m>\frac{n}{2}$. It suffices to show that $w'_1z^{-1}$ and $w'z^{-1}$ lie in the same orbit under the action of the group $E_{\Zcal,P_{w_1}}$.

First, by the proof of Proposition \ref{Pw-leng2}, the roots $\alpha_{r}$, $\alpha_{r+2}$ do not lie in $I_{w_1}$. This implies that the roots $s_{\alpha}$ for $\alpha\in I_{w_1}$ commute with one another. In particular, using the relation $w_1z^{-1}(I_{w_1})=I_{w_1}$, we deduce that $w_1z^{-1}$ commutes with $w_{0,I_{w_1}}$. Writing $w'_1z^{-1}=s_{\alpha_r}w_1z^{-1}$, we obtain:
\begin{align*}
w_{0,I_{w_1}} w'_1 z^{-1} w_{0,I_{w_1}}^{-1} &=(w_{0,I_{w_1}} s_{\alpha_r} w_{0,I_{w_1}}^{-1}) (w_{0,I_{w_1}} w_1 z^{-1} w_{0,I_{w_1}}^{-1}) \\
&=(w_{0,I_{w_1}} s_{\alpha_r} w_{0,I_{w_1}}^{-1}) w_1 z^{-1} = s_{w_{0,I_{w_1}}(\alpha_r)}  w_1 z^{-1}.
\end{align*}
Finally, since $\alpha_{r+1}\in I_{w_1}$ and $\alpha_{r-1},\alpha_r \notin I_{w_1}$, we find $w_{0,I_{w_1}}(\alpha_r)=s_{\alpha_{r+1}}(\alpha_r) = s_{\alpha_r}(\alpha_{r+1})$ by \eqref{rel-s-alpha}. We have showed:
\begin{equation}\label{relat-Iw1}
    w_{0,I_{w_1}} w'_1 z^{-1} w_{0,I_{w_1}}^{-1} = s_{\alpha_r}s_{\alpha_{r+1}}s_{\alpha_r}w_1z^{-1}=w'z^{-1}.
\end{equation}
This shows the claim that $w'_1z^{-1}$ and $w'z^{-1}$ lie in the same $E_{\Zcal,P_{w_1}}$-orbit and terminates the proof.
\end{proof}

\begin{comment}
    
For a subset $I_0\subseteq I$, note that we have
\begin{align*}
w'_1\in {}^{I_0}W \quad &\Longleftrightarrow \quad r+1\notin I_0, \\
w'_2\in {}^{I_0}W \quad  &\Longleftrightarrow \quad r-1\notin I_0.
\end{align*}
Hence, in the case $m>\frac{n}{2}$, we have $w'_2\in \Gamma_{I_{w_2}}^{\rm sm}(w_2)$, which implies that $\Ucal_2$ does not admit a sep. Similarly, in the case $m< \frac{n}{2}$, we have $w'_1\in \Gamma_{I_{w_1}}^{\rm sm}(w_1)$, hence $\Ucal_1$ is not smooth. 
\end{comment}

\begin{remark}\label{rmk: can type gcd 234}
In the case where $\gcd(r,s)>1$, the computation of $I_{w_i}$ is more difficult. One can show that if $\gcd(r,s)\in \{2,3,4\}$, then $I_{w_1}=\emptyset$.
\end{remark}

We now study the $w_i$-boundedness of $\Ucal_i$.

\begin{proposition}\label{Ui-bounded-prop}
Let $i\in \{1,2\}$. The following are equivalent:
\begin{enumerate}
    \item \label{Ui-bounded-1} $\Ucal_i$ is $w_i$-bounded.
    \item \label{Ui-bounded-2} $P_{w'}=B$.
    \item \label{Ui-bounded-3} $\gcd(r,s)\leq 3$.
\end{enumerate}
\end{proposition}

\begin{proof}
It was showed in \cite[Corollary 5.10]{Koskivirta-LaPorta-Reppen-singularities} that \eqref{Ui-bounded-2} and \eqref{Ui-bounded-3} are equivalent. Moreover, it is clear that \eqref{Ui-bounded-2} implies \eqref{Ui-bounded-1}. We prove the converse implication for $w_1$. The case of $w_2$ is similar. Assume that $\Ucal_1$ is $w_1$-bounded and suppose that $\delta\colonequals \gcd(r,s)>3$. By \cite[Proposition 5.9]{Koskivirta-LaPorta-Reppen-singularities}, the type of the parabolic subgroup of $w'$ is
\begin{equation}
    I_{w'}= \{\alpha_i \ | \ i\not\equiv -1,0,1 \pmod{\delta}\}.
\end{equation}
In particular, the root $\alpha_{r-s+2}$ lies in $I_{w'}$. On the other hand, one sees easily that $w_1z^{-1}(\alpha_{r-s+2})$ is not simple, hence it does not lie in $I_{w_1}$. Thus $I_{w'}\not\subset I_{w_1}$, which contradicts the $w_1$-boundedness of $\Ucal_1$. This terminates the proof.
\end{proof}

The following theorem gives a completely determines when the substack $\Ucal_i$ is smooth in terms of the signature $(r,s)$.

\begin{theorem}\label{thm-Urs-smooth} \ 
\begin{enumerate}
    \item\label{thm-long2-1} If $\gcd(r,s)>3$, then $\Ucal_i$ is not $w_i$-bounded (for $i=1,2$). In particular, $\Ucal_i$ is not smooth.
    \item\label{thm-long2-2} If $\gcd(r,s)\in \{2,3\}$, then $\Ucal_i$ is smooth for $i=1,2$.
    \item\label{thm-long2-3} If $\gcd(r,s)=1$ and we set $m=\inv_n(s)$, one has
    \begin{equation*}
       \Ucal_1 \ \textrm{smooth} \ \Longleftrightarrow \ m>\frac{n}{2}, \quad
        \Ucal_2 \ \textrm{smooth} \ \Longleftrightarrow \ m<\frac{n}{2}.
    \end{equation*}
\end{enumerate}
\end{theorem}

\begin{proof}
Assertion \eqref{thm-long2-1} follows from Proposition \ref{Ui-bounded-prop}. If $\gcd(r,s)\in \{2,3\}$, then $w'_1, w'_2$ are not small by Proposition \ref{prop-small}, hence $\Gamma_{I_{w_i}}^{\rm sm}(w_i)=\{w'\}$. Moreover, $\Ucal_i$ is $w_i$-bounded ($i=1,2$), thus $\Ucal_i$ is smooth by Proposition \ref{prop-smooth-small}, which shows \eqref{thm-long2-2}. Finally, Assertion \eqref{thm-long2-3} follows from Proposition \ref{Pw-leng2}.
\end{proof}

\begin{remark}\label{rmk-Urs-Uibar}
    In the case $\gcd(r,s)=1$, $m>\frac{n}{2}$, the Zariski closure $\overline{\Ucal}_1=\Xcal_{w_1}\cup \Xcal_{w'}\cup \Xcal_{\iden}$ is $w_1$-bounded (Proposition \ref{Ui-bounded-prop} and \cite[Cor.~5.8]{Koskivirta-LaPorta-Reppen-singularities}) and admits a separating canonical cover (by the proof of Corollary \ref{cor: length 2 small calc}, \(w'z^{-1}\) and \((w_1'z^{-1})\) have the same \(E_{\Zcal,P_{w_1}}\)-orbit), hence is normal and Cohen--Macaulay, by Proposition \ref{prop: results about normality}.(\ref{prop-normal-3}). Similarly, when $\gcd(r,s)=1$, $m<\frac{n}{2}$, the closed substack $\overline{\Ucal}_2=\Xcal_{w_2}\cup \Xcal_{w'}\cup \Xcal_{\iden}$ is normal and Cohen--Macaulay. However, if $\gcd(r,s)\in \{2,3\}$, $\overline{\Ucal}_i$ is not \(w_i\)-bounded and does not admit a canonical cover. The techniques developed in this paper are not enough to determine whether it is normal or not.
\end{remark}

\section{The split case} \label{sec-split}

In this section, we assume that $G$ is split over $\FF_p$. We compare the various stacks of $G$-zips attached to all positive powers of the Frobenius isogeny $\varphi$, as well as the identity isogeny (which we view as the zero-th power of $\varphi$). Our main goal is to show that several results established in the case of Frobenius-type zip data in previous sections also hold in the case $\varphi=\id_G$.

\subsection{Notation} \label{subsec-notation}
\subsubsection{Group theory}
Throughout, $m$ will denote a non-negative integer. We let $K=\overline{\FF}_p$ be an algebraic closure of $\FF_p$. Let $\sigma\in \Gal(K/\FF_p)$ be the $p$-power arithmetic Frobenius homomorphism. We let $G_0$ be a connected, split reductive group over $\FF_p$ and set $G\colonequals G_{0,K}$. The assumption that $G_0$ is $\FF_p$-split is necessary for our purpose, which is to compare the stacks of $G$-zips for different powers of the Frobenius isogeny (especially $m=0$), and eventually to a characteristic zero object. For this reason, we need the Galois action to be trivial on all attached group-theoretical objects. 

Write $\varphi$ for the $p$-power geometric Frobenius isogeny of $G_0$, base changed to $G$. We define $\varphi^0$ to be the identity isogeny of $G$. Fix a cocharacter $\mu\colon \GG_{\textrm{m},\FF_p}\to G_{0}$. Let $P\subseteq G$ denote the parabolic subgroup attached to $\mu$ (\S \ref{subsubsec-zip-data}). We fix a Borel pair $(B,T)$ in $G_0$ satisfying the same assumptions as in \S\ref{subsubsec-G-zips} and such that \(T\) splits over \(\FF_p\). 

%and make following assumptions:
%\begin{enumerate}
%    \item $(B,T)$ is defined over $\FF_p$.
%    \item $T$ is split over $\FF_p$.
%    \item $B\subseteq P$.
%    \item $\mu$ factors through $T$.
%\end{enumerate}
%With the exception of the second assumption, all other assumptions can easily be achieved after changing $\mu$ to a conjugate cocharacter. 

\subsubsection{The stack of $G$-zips with exponent $m$} \label{subsubsec-zip-expm}
Let $\Zcal_m$ denote the zip datum attached to the triple $(G,\mu,\varphi^m)$ for $m\geq 0$ (\S\ref{subsubsec-zip-data}). It is of Frobenius-type if and only if $m\geq 1$. Note that the groups $(P,Q,L,M)$ appearing in the zip datum $\Zcal_m$ are independent of $m\geq 0$, by our assumption that $T$ splits over $\FF_p$. Note also that $L=M=P\cap Q$. %Specifically, we have  $P=P_-$, $Q=P_+$, $L=M=P\cap Q$. 
We set $z\colonequals w_{0,I}w_0$ throughout. We adopt a simplified notation to denote the various objects attached to the zip datum $\Zcal_m$. Set $\Xcal_m\colonequals \Xcal_{\Zcal_m}$ and $E_m\colonequals E_{\Zcal_m}$. For $w\in {}^I W\cup W^J$, let $G_{m,w}\colonequals G_{\Zcal_m,w}$ and $\Xcal_{m,w}\colonequals \Xcal_{\Zcal_m,w}$. We call $\Xcal_m$ the stack of $G$-zips of exponent $m$ attached to $(G_0,\mu)$.

Note also that the partial order $\preccurlyeq_{\Zcal_m}$ on the set ${}^I W$ is independent of $m\geq 0$, by our assumption. We denote it simply by $\preccurlyeq$. For an intermediate parabolic $P_0$ and $w\in {}^{I_0}W$, write $\Gamma_{I_0}(w)$ for the set $\Gamma_{\Zcal^{(P_0)}_m}(w)$. Note that this set is independent of $m\geq 0$.

%We say that $(G,\mu)$ is an $\FF_p$-cocharacter datum if $\mu$ is defined over $\FF_p$ \JS{better terminology?} \lorenzo{How about "\(\FF_p\)-rational" or "split cocharacter datum"?}. Write $\varphi\colon G\to G$ for the geometric Frobenius homomorphism.

\subsubsection{The flag space}
For an intermediate parabolic $B\subseteq P_0 \subseteq P$, we set $\Fcal_m^{(P_0)}\colonequals \Fcal_{\Zcal_m}^{(P_0)}$. Similarly, for $w\in {}^{I_0}W\cup W^{J_0}$, write $\Fcal_{m,w}^{(P_0)} \colonequals\Fcal_{\Zcal_m,w}^{(P_0)}$. The zip datum $(\Zcal_{m})_{P_0}$ will be denoted by $\Zcal_{P_0}^{(m)}$. For $w \in {}^{I_0}W\cup W^{J_0}$, write $G^{(P_0)}_{m,w}\colonequals G_{\Zcal^{(m)}_{P_0},w}$. Similarly, the partial order $\preccurlyeq_{\Zcal_{P_0}^{(m)}}$ on ${}^{I_0} W$ is independent of $m\geq 0$, we simply denote it by $\preccurlyeq_{I_0}$. Its restriction to the subset ${}^I W$ is coarser that $\preccurlyeq$. Set $E'_{m,P_0} \colonequals E'_{\Zcal_m,P_0}$ and $E_{m,P_0}\colonequals E_{\Zcal^{(m)}_{P_0}}$. Therefore, we have $\Fcal_{m,w}^{(P_0)}=[E'_{m,P_0}\backslash G^{(P_0)}_{m,w}]$, for $w\in {}^{I_0}W\cup W^{J_0}$, and recall that
\begin{equation}
    G^{(P_0)}_{m,w} = E_{m,P_0}\cdot \left( BwBz^{-1} \right).
\end{equation}
Note that for $w\in {}^I W$, the canonical parabolic $P_w$ (which a priori depends on the zip datum $\Zcal_m$) is the same for all $m\geq 0$, hence the notation $P_w$ is unambiguous.

\subsection{Stabilizers}
For any $x\in G$, identify $\Stab_{E_m}(x)$ with a closed subgroup of $P$ via the first projection $\pr_1\colon E_m\to P$. Note that $G^{(P_w)}_{m,w}= P_w (wz^{-1}) Q_w$ for all $m\geq 0$.

\begin{lemma}\label{stab-Pw}
For any $x\in P_w (wz^{-1}) Q_w$ and any $m\geq 0$, one has $\Stab_{E_m}(x)\subseteq P_w$.
\end{lemma}

\begin{proof}
The case $m\geq 1$ is exactly the statement of Proposition \ref{stab-contained-Pw}. We now consider the case $m=0$. Let $(x,y)\in \Stab_{E_0}(wz^{-1})$. By definition, one has $\overline{x}=\overline{y}$. Choose an integer $m\geq 1$ such that $\overline{x}\in L(\FF_{p^m})$. One then has $\overline{y}=\varphi^m(\overline{x})$, hence $(x,y)\in \Stab_{E_m}(wz^{-1})$. By the case $m\geq 1$, we deduce $x\in P_w$.
\end{proof}

Viewed as a subgroup of $P\times Q$, we deduce $\Stab_{E_m}(wz^{-1})\subseteq P_w\times Q_w$, since $E_m\cap (P_w\times G)\subseteq P_w\times Q_w$.

\subsection{\texorpdfstring{The $\FF_p$-rational points}{}}
%We present a technical tool regarding the $\FF_p$-rational points of the various zip strata.
By our assumption, the locally closed subset $G_{m,w}$ is defined over $\FF_p$. In particular, we may speak of the $\FF_p$-points of this variety. Let $\Sigma_m\colonequals(G_{m,w})_{w\in {}^I W}$ denote the zip stratification of exponent $m\geq 0$ of $G$. The proposition below shows that the various stratifications $\Sigma_m$ (for $m\geq 0$) coincide on $G(\FF_p)$.

\begin{proposition}\label{prop-Gw} \ 
\begin{enumerate}
    \item We have $G_{m,w}(\FF_{p^m})=G_{0,w}(\FF_{p^m})$ for all $m\geq 0$. \label{Gw-item1}
    \item The set $G_{m,w}(\FF_p)$ is independent of $m\geq 0$.\label{Gw-item2}
\end{enumerate}
\end{proposition}

\begin{proof}
First, we show $G_{m,w}(\FF_{p^m})\subseteq G_{0,w}(\FF_{p^m})$ for any $m\geq 1$. Let $g\in G_{m,w}(\FF_{p^m})$ and $(a,b)\in E_m$ such that $g=awz^{-1}b^{-1}$. Since $g\in G(\FF_{p^m})$, we obtain $awz^{-1}b^{-1}=\varphi^m(a)wz^{-1}\varphi^m(b)^{-1}$, hence $(a^{-1}\varphi^m(a),b^{-1}\varphi^m(b))$ lies in the stabilizer in $E_m$ of $wz^{-1}$. By Lemma \ref{stab-Pw}, this stabilizer is contained is $E'_{m,P_w} = E_m\cap(P_w\times Q_w)$. Since $E'_{m,P_w}$ is connected, Lang's Theorem shows that we can write
\begin{equation}
    (a^{-1}\varphi^m(a),b^{-1}\varphi^m(b))=(x^{-1}\varphi^m(x),y^{-1}\varphi^m(y))
\end{equation}
for $(x,y)\in E'_{m,P_w}$. We deduce that $(ax^{-1},by^{-1})\in E_m(\FF_{p^m})$. Moreover, note that we have $E_m(\FF_{p^m})=E_0(\FF_{p^m})$. We can write
\begin{equation}
    g=awz^{-1}b^{-1} = (ax^{-1})(xwz^{-1}y^{-1})(y b^{-1}).
\end{equation}
Since $(ax^{-1},by^{-1})\in E_0$, this implies that $g\in E_0\cdot (P_w wz^{-1}Q_w)=G_{0,w}$, by \eqref{Gw-withPw}. This shows the inclusion $G_{m,w}(\FF_{p^m})\subseteq G_{0,w}(\FF_{p^m})$. Since
\begin{equation}
    G(\FF_{p^m})=\bigsqcup_{w\in {}^I W} G_{m,w}(\FF_{p^m})=\bigsqcup_{w\in {}^I W} G_{0,w}(\FF_{p^m}),
\end{equation}
we must have $G_{m,w}(\FF_{p^m})= G_{0,w}(\FF_{p^m})$. %The same proof applies for the reverse inclusion, since Lemma \ref{stab-Pw} also holds for $m=0$. 
This shows \eqref{Gw-item1}. For the second assertion, we have for $m\geq 1$:
\begin{equation}
    G_{m,w}(\FF_{p})=G_{m,w}(\FF_{p^m})\cap G(\FF_p) = G_{0,w}(\FF_{p^m})\cap G(\FF_p) = G_{0,w}(\FF_{p}).\qedhere
\end{equation}
\end{proof}

%We illustrate the result of Proposition \ref{prop-Gw} in the following example.

\begin{exa}
Consider the case of $G=\GL_{3,\FF_p}$, endowed with the cocharacter $\mu\colon t\mapsto \diag(t,t,1)$. This corresponds to the cocharacter datum arising from to a Picard modular surface attached to a quadratic, totally imaginary extension $\mathbf{E}/\QQ$ at a split prime $p$ of $\mathbf{E}$ of good reduction. For any $m\geq 0$, the zip stratification is linear and consists of exactly $3$ strata. The unique open stratum $U\subseteq \GL_{3,k}$ is the same for all stratifications $\Sigma_m$ ($m\geq 0$), and is given by the matrices of the form
\begin{equation}
   x = \left( \begin{matrix}
        x_{11}&x_{12}&x_{13}\\
        x_{21}&x_{22}&x_{23}\\
        x_{31}&x_{32}&x_{33}\\
    \end{matrix}\right) \quad \textrm{such that} \quad  \delta(x)\colonequals \left|\begin{matrix}
        x_{11}&x_{12}\\
        x_{21}&x_{22}
    \end{matrix}\right|\neq 0.
\end{equation}
The unique stratum of codimension one consists of matrices $x\in \GL_{3,k}\setminus U$ such that $\Ha^{(m)}\neq 0$, where
\begin{equation}
    \Ha^{(m)}(x)=x_{11}\left| \begin{matrix}
        x_{11}&x_{12}\\
        x_{31}&x_{32}
    \end{matrix}\right|^{p^m-1}+x_{22}\left| \begin{matrix}
        x_{21}&x_{22}\\
        x_{31}&x_{32}
    \end{matrix}\right|^{p^m-1}.
\end{equation}
This formula still holds in the case $m=0$, where we simply have $\Ha^{(0)}(x)=x_{11}+x_{22}$. The section $\Ha^{(0)}$ is an example of a generalized Hasse-invariant for a zip datum which is not of Frobenius-type. It lifts to a characteristic zero analogous section. See \S\ref{sec-Hasse-0} for more details on such Hasse invariants.

The third zip stratum, which is the unique closed stratum, is the complement of the above two strata. It is defined by the vanishing of both $\delta$ and $\Ha^{(m)}$. In particular, we see on this example that the various stratifications of $G$ strongly depend on the choice of $m\geq 0$. However, for $x\in \GL_3(\FF_p)$, we have $\Ha^{(m)}(x)=x_{11}+x_{22}=\Ha^{(0)}(x)$. This shows that these stratifications coincide on $\GL_3(\FF_p)$.
\end{exa}

We have a similar result for flag strata. Let $B\subseteq P_0\subseteq P$ be an intermediate parabolic. Let $w\in {}^{I_0}W$ and consider the stratum $G_{m,w}^{(P_0)}$, which is the zip stratum parametrized by $w$ with respect to the zip datum $\Zcal^{(m)}_{P_0}$. We may apply Proposition \ref{prop-Gw} to this zip datum, and we obtain the following corollary:
\begin{corollary}\label{cor-Gw-P0} \ 
\begin{enumerate}
    \item We have $G^{(P_0)}_{m,w}(\FF_{p^m})=G^{(P_0)}_{0,w}(\FF_{p^m})$ for all $m\geq 0$. \label{Gw-item1-P0}
    \item The set $G^{(P_0)}_{m,w}(\FF_p)$ is independent of $m\geq 0$.\label{Gw-item2-P0}
\end{enumerate}
\end{corollary}

\subsection{Exponent raising}\label{sec-expo-raising-split}
We compare the images of strata for various values of $m\geq 0$. Recall that we defined maps $\Pi_{\Zcal}\colon  W\to \Pcal({}^I W)$ and $\pi_{\Zcal}\colon  W\to {}^I W$ in \S\ref{subsec-images}. For $\Zcal=\Zcal_m$, write $\pi_{m}$ and $\Pi_{m}$ for these maps (for $m\geq 0$). 

\begin{proposition}\label{prop-exp-raising}
Let $m\geq 0$ and $w\in W$.
\begin{enumerate}
    \item For $w'\in \Pi_{m}(w)$, there exists $d\in \ZZ_{\geq 1}$ such that $w'\in \Pi_{m+dk}(w)$ for all $k\geq 0$.
    \item There exists $D\in \ZZ_{\geq 1}$ such that $\Pi_{m}(w)\subseteq \Pi_{m+Dk}(w)$ for all $k\geq 0$.
\end{enumerate}
\end{proposition}

We first state the following lemma, whose proof is entirely similar to Lemma \ref{lem-same-stratum-m} (also for $m=0$):

\begin{lemma}\label{lem-Em-raising}
Let $m\geq 0$ and $a\in G$. If $x\in E_m\cdot a$, then there exists an integer $d\geq 1$ (depending on $x$) such that $x\in E_{m+dk}\cdot a$ for all $k\geq 0$.
\end{lemma}

%\begin{proof}
%By assumption, there exists $(\alpha,\beta)\in E_m$ such that $\alpha a\beta^{-1}=x$. We can write $\alpha=us$, $\beta=\varphi^m(s)u'$ for $s\in L$, $u\in R_{\mathrm{u}}(P)$ and $u'\in R_{\mathrm{u}}(Q)$. Let $d\geq 1$ be an integer such that $s\in L(\FF_{p^d})$, hence $\varphi^d(s)=s$. It follows that $\varphi^{m+dk}(s)=\varphi^m(s)$ for all $k\ge 0$. We thus have $(\alpha,\beta)\in E_{m+dk}$, hence $x\in E_{m+dk}\cdot a$ for all $k\geq 0$.
%\end{proof}

We now prove Proposition \ref{prop-exp-raising}. In the case $m\geq 1$, it simply follows from Proposition \ref{prop-same-stratum-m}.

\begin{proof}
Recall that $\Pi_{m}(w)$ is the set of $w'\in {}^{I}W$ such that the zip stratum $G_{m,w'}\cap BwBz^{-1} \neq \emptyset$. Therefore, for $w'\in \Pi_m(w)$, there exists $a\in G$ such that the $E_m$-orbit of $a$ intersects both $BwBz^{-1}$ and $Bw'Bz^{-1}$. Therefore, there exist $x,y\in G$ with $x\in (E_m\cdot a)\cap BwBz^{-1}$ and $y\in (E_m\cdot a)\cap Bw'Bz^{-1}$. By Lemma \ref{lem-Em-raising}, there exist $d,d'\geq 1$ such that $x\in E_{m+dk}\cdot a$ and $y\in E_{m+d'k}\cdot a$ for all $k\geq 0$. Hence for all $k\geq 0$, we have $x\in E_{m+dd'k}\cdot a$ and $y\in E_{m+dd'k}\cdot a$. This shows that the $E_{m+dd'k}$-orbit of $a$ intersects both $BwBz^{-1}$ and $Bw'Bz^{-1}$. In other words, $w'\in \Pi_{m+dd'k}(w)$. This terminates the proof of the first assertion of Proposition \ref{prop-exp-raising}. For the second assertion, write $d(w')$ to be the integer afforded by the first part for $w'\in \Pi_{m}(w)$. Define $D$ to be the least common multiple of all $d(w')$ when $w'$ varies in the set $\Pi_{m}(w)$. Then $\Pi_{m}(w)\subseteq \Pi_{m+Dk}(w)$ for all $k\geq 0$.
\end{proof}

We can use this result to extend the content of Lemma \ref{lem-length-ineq} to the case $m=0$, which is not of Frobenius-type.

\begin{corollary}\label{lem-ineq-0}
We have $\ell(\pi_0(w))\leq \ell(w)$ for all $w\in W$.
\end{corollary}

\begin{proof}
Set $w_1\colonequals \pi_{0}(w)$. In particular, $w_1\in \Pi_0(w)$. Hence, there exists $m>0$ such that $w_1\in \Pi_m(w)$. In particular, $\ell(w_1)\leq \ell(\pi_m(w))\leq \ell(w)$. 
\end{proof}

\begin{comment}

\begin{corollary}
Let $m,m'\geq 0$ be two integers and $w\in W$. There exists $M\geq 0$ such that 
\begin{equation}
    \Pi_m(w)\cup \Pi_{m'}(w)\subseteq \Pi_M(w).
\end{equation}
\end{corollary}
\begin{proof}
Let $D,D'$ be integers such that $\Pi_m(w)\subseteq \Pi_{m+Dk}(w)$ and $\Pi_{m'}(w)\subseteq \Pi_{m'+D'k'}(w)$ for all $k,k'\geq 0$. Take
\end{proof}

For $w\in W$, define a set $\Pi_\infty(w)$ as the union
\begin{equation}
    \Pi_\infty(w)\colonequals \bigcup_{m\geq 0} \Pi_m(w).
\end{equation}
\end{comment}

\subsection{Consequences}
We conjecture the following:
\begin{conjecture} \label{indep-conj}
The maps $\Pi_{m}$ and $\pi_{m}$ are independent of $m\geq 0$.
\end{conjecture}
The first part of the conjecture immediately implies the second part. We prove a weak version of this conjecture in the following lemma. Recall that $w$ is $\Zcal_m$-small if $\ell(\pi_m(w))=\ell(w)$ (Definition \ref{def-small}). This definition applies for any $m\geq 0$.

\begin{lemma}\label{lem-small-indep}
Assume that $w\in W$ is $\Zcal_m$-small, for some $m\geq 0$. Then:
\begin{enumerate}
    \item $w$ is $\Zcal_{m'}$-small for all $m'\geq 0$.
    \item The element $\pi_{m'}(w)\in {}^I W$ is independent of $m'\geq 0$.
\end{enumerate}
\end{lemma}

\begin{proof} %\lorenzo{This proof can be simplified using Proposition \ref{prop-W-representative}}\JS{I agree, but the proof is not too long. Also, I am not sure we can apply that proposition in the case m=0, because we have not shown that the image is pointwise in that case}
In the case $m=0$, Proposition \ref{prop-exp-raising} implies that $\Fcal_{m_1,w}^{(P_0)}$ is $\Zcal_{m_1}$-small for some $m_1>0$, hence we may assume $m>0$. By Corollary \ref{coro-small-Fp}, there exists $w_1\in {}^I W$ of length $\ell(w)$ such that $\left(G^{(B)}_{m,w} \cap G_{m,w_1}\right)(\FF_p)\neq \emptyset$. Since this set is independent of $m\geq 0$, we deduce $G^{(B)}_{m',w} \cap G_{m',w_1}\neq \emptyset$ for all $m'\geq 0$, hence $w_1\in \Pi_{m'}(w)$. By Lemma \ref{lem-length-ineq} (in the case $m'\geq 1$) and Corollary \ref{lem-ineq-0} (in the case $m'=0$), we must have $w_1=\pi_{m'}(w)$. This terminates the  proof.\qedhere

%again by Corollary \ref{coro-small-Fp} (which only applies to Frobenius-type zip data) that $\Fcal_{m',w}^{(P_0)}$ is small for all $m'\geq 1$. We also have $w_1\in \Pi_m(w)$, hence $\pi_{m',P_0}(w)=w_1$ by Lemma \ref{lem-length-ineq}, which shows part of the second assertion. It remains to prove that $\pi_{0,P_0}(w)=w_1$. Since $\left(G^{(P_0)}_{0,w} \cap G^{(P_0)}_{0,w_1}\right)(\FF_p)\neq \emptyset$, the stratum $\Xcal_{0,w_1}$ intersects the image of $\Fcal^{(P_0)}_{0,w}$ via $\pi_{P_0}$. Since $w_1\in \Pi_0(w)$ and $\ell(\pi_0(w))\leq \ell(w)$ (Lemma \ref{lem-ineq-0}), we must have $\pi_0(w)=w_1$. This terminates the  proof.
\end{proof}

By Lemma \ref{lem-small-indep}, the set of $\Zcal_m$-small elements of $W$ is independent of $m\geq 0$. Write simply $W^{\rm sm}$ for this set. We have shown that the restriction of the map $\pi_m$ to $W^{\rm sm}\subseteq W$ is independent of $m$. This gives a partial answer to Conjecture \ref{indep-conj}.

\subsection{Separating canonical covers}
For $w\in {}^I W$ and $w'\in \Gamma_I(w)$, denote by
\begin{equation}
    \Ucal_m(w,w')\subseteq \Xcal_m
\end{equation}
the corresponding elementary $w$-open substack inside $\Xcal_m$. %The results of the previous section make it possible to give a characterisation for the existence of a separating canonical cover for elementary subsets.
%Write $\Gamma^{\rm sm}_{I_0}(w)\colonequals \Gamma_{I_0}(w)\cap W^{\rm sm}$. 
For a subset $A\subseteq W$, we simply write $\Pi_m(A)$ for the subset $\bigcup_{w\in A}\Pi_m(w)$.
\begin{proposition}\label{prop-scc-charac}
For any $m\geq 0$, the following are equivalent:
    \begin{enumerate}
        \item \label{scc1} $\Ucal_m(w,w')$ admits a separating canonical cover.
        \item \label{scc2} One has $w'\notin \Pi_m(\Gamma_{I_w}(w)\setminus \{w'\})$.
            \item \label{scc3} One has $w'\notin \pi_m(\Gamma_{I_w}(w)\setminus \{w'\})$.
            \item \label{scc4} One has $w'\notin \pi_m(\Gamma^{\rm sm}_{I_w}(w)\setminus \{w'\})$.
    \end{enumerate}
\end{proposition}
\begin{proof}
For $m\geq 0$, $\Ucal_m(w,w')$ admits a separating canonical cover if and only if 
\begin{enumerate}
    \item[(a)] $w'\preccurlyeq_{I_w} w$ (i.e.\ it admits a canonical cover), and
     \item[(b)] for all $v\preccurlyeq_{I_w} w$ such that $v\notin \{w,w'\}$, we have $\Pi_{m}(v)\cap \{w,w'\}=\emptyset$.
\end{enumerate}
Condition (b) is equivalent to $w'\notin \Pi_m(\Gamma_{I_w}(w)\setminus \{w'\})$ by Lemmas \ref{lem-length-ineq} (for $m\ge 1$) and \ref{lem-ineq-0} (for $m = 0$). Moreover, since $\pi_{P_w}(\overline{\Fcal}_{m,w})=\overline{\Xcal}_{m,w}$, we must have $w'\in \Pi_m(\Gamma_{I_w}(w))$. It follows that if $w'\notin \Pi_m(\Gamma_{I_w}(w)\setminus \{w'\})$, then $w'\in \Gamma_{I_w}(w)$. Hence, condition (b) implies condition (a). It follows that \eqref{scc1} and \eqref{scc2} are equivalent. Clearly, \eqref{scc2} implies \eqref{scc3}, which implies \eqref{scc4}. Conversely, suppose $w'\in \Pi_m(v)$, for some $v\in \Gamma_{I_w}(w)\setminus \{w'\}$. Since $\ell(w')=\ell(v)$, we must have $w'=\pi_m(v)$. Thus $v$ is small, and $w'\in \pi_m(\Gamma^{\rm sm}_{I_w}(w)\setminus \{w'\})$. Hence \eqref{scc4} implies \eqref{scc2}. This terminates the proof.
\end{proof}

We now explain the main result of this section. Note that condition \eqref{scc4} of Proposition \ref{prop-scc-charac} is independent of $m\geq 0$, by Lemma \ref{lem-small-indep}. We deduce the following theorem.

\begin{proposition}\label{prop-SCC-m}
Let $m\geq 1$ be an integer. The following are equivalent:
\begin{enumerate}
   \item $\Ucal_m(w,w')$ admits a separating canonical cover.
    \item $\Ucal_0(w,w')$ admits a separating canonical cover.
\end{enumerate}
\end{proposition}

\begin{comment}
\begin{proof}
First, note that $\Ucal_0(w,w')$ admits canonical cover if and only if $\Ucal_m(w,w')$ does, since $\preccurlyeq_{\Zcal_m}$ is independent of $m\geq 0$. Assume that $\Ucal_m(w,w')$ does not admit a separating canonical cover. Since $m\geq 1$, the zip datum $\Zcal_m$ is of Frobenius-type. Therefore, Corollary \ref{cor-fp-points} shows that there exists $v\in \Gamma_{I_w}(w)\setminus \{w'\}$ such that $w'\in \Pi_{m}(v)$. Since $\ell(v)=\ell(w')$, we have $\pi_m(v)=w'$ and $v$ is small. Hence $\pi_0(v)=w'$, which shows that $\Ucal_0(w,w')$ does not admit a separating canonical cover. 

Conversely, assume that $\Ucal_0(w,w')$ does not admit a separating canonical cover. In this case, there exists $v\in \Gamma_{I_w}(w)\setminus \{w'\}$ such that $w'\in \Pi_0(v)$. By Proposition \ref{prop-exp-raising}, there exists $m'\geq 1$ such that $w'\in \Pi_{m'}(v)$. This implies that $\Ucal_{m'}(w,w')$ does not admit a separating canonical cover. Moreover, $v$ is small and $\pi_{m'}(v)=w'$. By Lemma \ref{lem-small-indep}, we know that $\pi_{m}(v)=w'$, and hence $\Ucal_m(w,w')$ does not admit a separating canonical cover.
\end{proof}
\end{comment}

Finally, using Theorem \ref{thm-FTZD-normality}, we obtain the following corollary:

\begin{corollary}\label{Um-smooth}
Let $m\geq 1$ be an integer. The following are equivalent:
\begin{enumerate}
    \item $\Ucal_m(w,w')$ is smooth (equivalently, normal).
    \item $\Ucal_0(w,w')$ is $w$-bounded and admits a separating canonical cover.
    \item $P_{w'}\subseteq P_w$ and $w'\notin \pi_0(\Gamma^{\rm sm}_{I_w}(w)\setminus \{w'\})$.
\end{enumerate}    
\end{corollary}
In particular, the smoothness of $\Ucal_m(w,w')$ is independent of $m\geq 1$. We do not know whether these conditions are also equivalent to the smoothness of the substack $\Ucal_0(w,w')$. We say that a scheme or stack $X$ is \emph{smooth in codimension one} if the complement of the smooth locus of $X$ has codimension $\geq 2$.

\begin{corollary}\label{cor-closure}
Let $w\in {}^I W$ and $m\geq 1$. The following are equivalent:
\begin{enumerate}
    \item\label{cor-closure-item1} $\overline{\Xcal}_m(w)$ is smooth in codimension one.
    \item\label{cor-closure-item2}  $P_{w'}\subseteq P_w$ for all $w'\in \Gamma_I(w)$ and $\Gamma_{I_w}^{\rm sm}(w)=\Gamma_I(w)$.
    \item\label{cor-closure-item3}  $P_{w'}\subseteq P_w$ for all $w'\in \Gamma_I(w)$ and $\Gamma_{I_w}^{\rm sm}(w)\subseteq\Gamma_I(w)$.
\end{enumerate}
\end{corollary}

\begin{proof}
We apply Corollary \ref{Um-smooth} to each lower neighbor of $w$ in ${}^I W$. Assume \eqref{cor-closure-item1}.  Let $w'\in \Gamma_{I_w}^{\rm sm}(w)$ and write $w_1\colonequals \pi_0(w')$. Since $\Ucal_m(w,w_1)$ is smooth, we must have $w_1=w'$ hence $w'\in \Gamma_I(w)$. Hence $\Gamma_{I_w}^{\rm sm}(w)\subseteq \Gamma_I(w)$. Moreover, for any $w'\in \Gamma_I(w)$, the substack $\Ucal_m(w,w')$ admits a canonical cover, hence $w'\preccurlyeq_{I_w} w$. This shows $\Gamma_I(w)\subseteq \Gamma_{I_w}(w)$, and thus $\Gamma_I(w)\subseteq \Gamma^{\rm sm}_{I_w}(w)$, since elements of ${}^IW$ are small by Proposition \ref{prop-Pw}. Hence \eqref{cor-closure-item2} holds. \eqref{cor-closure-item2} implies \eqref{cor-closure-item3} trivially. Finally, assume \eqref{cor-closure-item3}, and let $w'\in \Gamma_I(w)$. Since $\pi_0$ is the identity on ${}^I W$, we have $w'\notin \pi_0(\Gamma^{\rm sm}_{I_w}(w)\setminus\{w'\})$ for each $w'\in \Gamma_I(w)$. Therefore $\bigcup_{w'\in \Gamma_I(w)}\Ucal_m(w,w')$ is smooth by Corollary \ref{Um-smooth}.
\end{proof}

\section{Lifting to characteristic zero}\label{sec-lift-char0}

\subsection{\texorpdfstring{$\Gcal$-zips over $\ZZ_p$}{}}
\subsubsection{Setting}
Let $\Gcal$ be a reductive group over $\ZZ_p$ with connected special fiber, endowed with a cocharacter
\begin{equation}
    \mu\colon \GG_{\textrm{m},\ZZ_p}\to \Gcal.
\end{equation}
It gives rise to a pair of opposite parabolic subgroups $\Pcal=\Pcal_-$ and $\Qcal=\Pcal_+$, with common Levi subgroup $\Lcal=\Pcal\cap \Qcal$, similarly to \S\ref{subsubsec-zip-expm}. Assume that there is a Borel pair $(\Bcal,\Tcal)$ over $\ZZ_p$ satisfying:
\begin{enumerate}
    \item $\Tcal$ is split over $\ZZ_p$.
    \item $\Bcal\subseteq \Pcal$.
    \item $\mu$ factors through $\Tcal$.
\end{enumerate}
Define the zip group $\Ecal_0\subseteq \Pcal\times \Qcal$ by
\begin{equation}
    \Ecal_0(R)=\{(x,y)\in \Pcal(R)\times \Qcal(R) \ | \ \overline{x}=\overline{y} \} \qquad \textrm{for any $\ZZ_p$-algebra $R$},
\end{equation}
where $\overline{x}, \overline{y}\in \Lcal(R)$ are the Levi projections of $x,y$ respectively. Define the $\ZZ_p$-stack $\Xscr$ by
\begin{equation}\label{Xscr-def}
    \Xscr \colonequals [\Ecal_0 \backslash \Gcal].
\end{equation}
It is clear that $\Xscr$ is smooth over $\ZZ_p$. We can apply the results of \S \ref{sec-review-zip} to both $\Xscr_{\overline{\QQ}_p}$ and $\Xscr_{\overline{\FF}_p}$.

\subsubsection{Stratifications}\label{subsub-stratif}
Let $W\colonequals W(\Gcal,\Tcal)$ be the Weyl group of $\Gcal$. Let $\Phi^+$, $\Delta$, denote the sets of positive and simple roots, respectively, with respect to the opposite Borel \(\Bcal^+\). Write $I\subseteq \Delta$ for the type of $\Pcal$. Set $z\colonequals w_{0,I}w_0$. For $w\in W$, set $\Ccal_w\colonequals \Bcal w \Bcal z^{-1}$, which is locally closed in $\Gcal$. Our first goal is to define a zip stratification $(\Xscr_w)_{w\in {}^IW}$ on $\Xscr$. Consider the map $\gamma \colon  \Gcal\times \Pcal \to \Gcal$ defined by $\gamma(g,x)=x^{-1}g \overline{x}$. It induces a well-defined map
\begin{equation}
    \widetilde{\gamma} \colon  \Gcal\times \left(\Pcal/\Bcal\right) \to \Bcal\backslash \Gcal/{}^z\Bcal
\end{equation}
which is smooth and surjective. For $w\in W$, define $\Hcal_w\colonequals \widetilde{\gamma}^{-1}(\Ccal_w)$. Furthermore, let $\pr_1\colon \Gcal\times \left(\Pcal/\Bcal\right)\to \Gcal$ be the first projection, which is a proper smooth map. Define for $w\in {}^I W$
\begin{equation}
    \Gcal_w\colonequals \pr_1(\Hcal_w), \quad \Gcal^*_w\colonequals \pr_1(\overline{\Hcal}_w).
\end{equation}
For $w\in W$, we have $\overline{\Hcal}_w=\bigcup_{w'\leq w}\Hcal_{w'}$. Indeed, this follows from the smoothness of the map $\widetilde{\gamma}$ and a similar formula for the closure of $\Ccal_w$. We deduce
\begin{equation}
    \Gcal^*_w = \bigcup_{w'\leq w} \pr_1(\Hcal_{w'}).
\end{equation}
Let $\Zcal_{\overline{\QQ}_p}$ and $\Zcal_{\overline{\FF}_p}$ be the zip data attached to the generic and special fiber of $(\Gcal,\mu)$, respectively. Since $\Gcal$ is split over $\ZZ_p$, the twisted orders on ${}^I W$ attached to $\Zcal_{\overline{\QQ}_p}$ and $\Zcal_{\overline{\FF}_p}$ can be identified. We simply denote it by $\preccurlyeq$. 

\begin{proposition} \ 
\begin{enumerate}
    \item $\Gcal$ is the disjoint union of the sets $\Gcal_w$ for $w\in {}^IW$ .
    \item The set $\Gcal_w$ is locally closed in $\Gcal$.
    \item The set $\Gcal^*_w$ is closed in $\Gcal$ and coincides with the Zariski closure $\overline{\Gcal}_w$ of $\Gcal_w$.
\end{enumerate}
\end{proposition}

\begin{proof}
For the first statement, it suffices to check it over $\overline{\QQ}_p$ and $\overline{\FF}_p$ where the result is already known. We claim that
\begin{equation}\label{eq-Gstar}
    \Gcal^*_w=\bigcup_{w'\preccurlyeq w}\Gcal_{w'}.
\end{equation}
Indeed, this identity can be checked set-theoretically on the special and generic fibers, where we already know it holds. Since $\pr_1$ is proper, $\Gcal^*_w$ is closed in $\Gcal$ for any $w\in {}^I W$. We can write
\begin{equation}
    \Gcal_w=\Gcal^*_w\setminus \bigcup_{w'\prec w} \Gcal_{w'}=\Gcal^*_w\setminus \bigcup_{w'\prec w} \Gcal^*_{w'}.
\end{equation}
The second equality follows from \eqref{eq-Gstar} applied to each $\Gcal_{w'}$ for $w'\prec w$. This shows that $\Gcal_w$ is locally closed. Finally, since $\pr_1$ is proper, we have $\Gcal^*_w=\pr_1(\overline{\Hcal}_w)=\overline{\pr_1(\Hcal_w)}=\overline{\Gcal}_w$. This concludes the proof.
\end{proof}

\subsection{Flag strata}
For an intermediate parabolic $\Bcal\subseteq \Pcal_0\subseteq \Pcal$, define the flag space $\Fscr^{(\Pcal_0)}$ as the quotient stack
\begin{equation}
    \Fscr^{(\Pcal_0)} = [\Ecal'_{0,\Pcal_0} \backslash \Gcal]
\end{equation}
where $\Ecal'_{0,\Pcal_0}=\Ecal_{0}\cap (\Pcal_0\times \Gcal)$, similarly to \S\ref{subsec-flag-space}. An equivalent definition is $\Fscr^{(\Pcal_0)}=[\Ecal_0\backslash \left(\Gcal\times (\Pcal/\Pcal_0)\right)]$, where $\Ecal_0$ acts on $\Gcal\times (\Pcal/\Pcal_0)$ by $(x,y)\cdot (g,h\Bcal)=(xgy^{-1},xh\Bcal)$ for $(x,y)\in \Ecal_0$, $g\in \Gcal$ and $h\in \Pcal$. The natural map $\pi_{\Pcal_0}\colon \Fscr^{(\Pcal_0)}\to \Xscr$ then corresponds to the $\Ecal_0$-equivariant map $\pr_1\colon \Gcal\times (\Pcal/\Pcal_0)\to \Gcal$.

Write $\Lcal_0\subseteq \Pcal_0$ for the unique Levi subgroup containing $\Tcal$. Define the stack $\Xscr_{\Pcal_0}$ similarly to $\Xscr$ after replacing $\Pcal$ with $\Pcal_0$ and $\Qcal$ with $\Qcal_0\colonequals \Lcal_0{}^z\Bcal$. Write $I_0\subseteq I$ for the type of $\Pcal_0$. For $w\in {}^{I_0}W$, let $\Xscr_{\Pcal_0,w}$ be the corresponding zip stratum of $\Xscr_{\Pcal_0}$. Define the flag stratum $\Fscr^{(\Pcal_0)}_w$ as the inverse image of $\Xscr_{\Pcal_0,w}$ by the natural map \begin{equation}
    \psi_{\Pcal_0}\colon \Fscr^{(\Pcal_0)} \to \Xscr_{\Pcal_0}
\end{equation}
induced by the inclusion $\Ecal'_{0,\Pcal_0}\subseteq \Ecal_0$.

\subsection{Images of strata}
For $w\in W$, write $\Pi_{\ZZ_p}(w)$ for the set of $w'\in {}^I W$ such that $\pi_{\Bcal}(\Fscr^{(\Bcal)}_{w})\cap \Xscr_{w'}\neq \emptyset$, equivalently $\Fscr^{(\Bcal)}_{w}\cap \pi_{\Bcal}^{-1}(\Xscr_{w'})\neq \emptyset$. Similarly, write $\Pi_{\QQ_p}(w)$ and $\Pi_{\FF_p}(w)$ for the analogues of this set for the generic and special fiber respectively. For $R\in \{\ZZ_p,\FF_p,\QQ_p\}$, write $\pi_R(w)$ for the longest element in $\Pi_R(w)$. We obtain maps
\begin{equation}\label{piR-maps}
    \Pi_R\colon W\to \Pcal({}^IW), \quad \pi_R\colon W\to {}^IW.
\end{equation}
Denote also by $W_{R}^{\rm sm}$ the set of small elements over $R$, i.e.\ the set of $w'\in W$ such that $\ell(\pi_R(w))=\ell(w)$. %We clearly have $W_{\ZZ_p}^{\rm sm}=W_{\QQ_p}^{\rm sm}\cup W_{\FF_p}^{\rm sm}$.
It is clear that $\Pi_{\ZZ_p}(w)=\Pi_{\QQ_p}(w)\cup \Pi_{\FF_p}(w)$.

\begin{lemma}\label{lem-Fp-Qp} 
For $w\in W$, we have $\pi_{\FF_p}(w)\preccurlyeq \pi_{\QQ_p}(w)$ and $\pi_{\ZZ_p}(w)= \pi_{\QQ_p}(w)$. 
\end{lemma}

\begin{proof}
The second assertion immediately follows from the first one. Write $w'\colonequals \pi_{\QQ_p}(w)$. Consider the proper map $\pi_\Bcal \colon \Fscr^{(\Bcal)}_w \to \Xscr$. We have
\begin{equation}
\pi_\Bcal(\overline{\Fscr}_{\QQ_p,w}^{(\Bcal)})=\overline{\pi_\Bcal(\Fscr^{(\Bcal)}_{\QQ_p,w})}\subseteq \overline{\Xscr}_{w'} = \bigcup_{w''\preccurlyeq w'}\Xscr_{w''}.  
\end{equation}
Since $\Fscr^{(\Bcal)}_{\QQ_p,w}$ is dense in $\Fscr^{(\Bcal)}_w$ (as $\Ccal_w=\Bcal w\Bcal z^{-1}$ is integral), we deduce that $\pi_{\FF_p}(w)\preccurlyeq w'$.
\end{proof}

\subsection{\texorpdfstring{Reduction to characteristic $p$}{}}

The goal of this section is to prove the following proposition, using a reduction mod $p$ argument. Say that a $\ZZ_p$-scheme is \emph{rational} if it is birational to an affine space $\AA^n_{\ZZ_p}$. For example, the set $\Bcal w\Bcal$ (called a Bruhat cell) is rational (because $\Tcal$ is split over $\ZZ_p$). Note that the set $\AA^n(\overline{\ZZ}_p)$ is Zariski dense in $\AA^n$, hence the same is true for irreducible rational $\ZZ_p$-schemes. Similarly, if a $\ZZ_p$-scheme $X$ admits a surjective map $Y\to X$ from an irreducible rational $\ZZ_p$-scheme, then $X(\overline{\ZZ}_p)$ is Zariski dense in $X$.

\begin{proposition}\label{prop-pi-Qp-equals-Fp}
For any $w\in W$, one has $\pi_{\FF_p}(w) = \pi_{\QQ_p}(w)$.
\end{proposition}

\begin{proof}
The set $\Hcal_w$ considered in \ref{subsub-stratif} can be written as follows:
\begin{equation}
\Hcal_w \colonequals \{(g,xB)\in \Gcal\times (\Pcal/\Bcal) \ | \ x^{-1}g\overline{x}\in \Ccal_w\}.    
\end{equation}
There is a natural surjective $\Ecal'_0$-invariant map $\Ecal_0 \times \Ccal_w\to \Hcal_w$, $((a,b),x)\mapsto (axb^{-1},a\Bcal)$. Since the $\ZZ_p$-scheme $\Ecal_0 \times \Ccal_w$ is rational, it follows that $\Hcal_w(\overline{\ZZ}_p)$ is Zariski dense in $\Hcal_w$. The map $\pi_\Bcal\colon \Fscr^{(\Bcal)}_w\to \Xscr$ corresponds to the first projection $\pr_1\colon \Hcal_w\to \Gcal$, which is $\Ecal_0$-equivariant. 

If we write $w'=\pi_{\QQ_p}(w)$, then $\pr_1(\Hcal_w)\subseteq \overline{\Gcal}_{w'}$ since $\pi_{\QQ_p}(w)=\pi_{\ZZ_p}(w)$. Hence, the preimage of $\Gcal_{w'}$ inside $\Hcal_w$ is open in $\Hcal_w$ and nonempty (because it contains a $\overline{\QQ}_p$-point), hence dense in $\Hcal_w$ by irreducibility. By the above, it follows that it contains a $\overline{\ZZ}_p$-point. Hence, there exists $x\in \Hcal_w(\overline{\ZZ}_p)$ such that $\pr_1(x)\in \Gcal_{w'}$ by $\pr_1$. In particular, $(\Hcal_w\cap \Gcal_{w'})(\overline{\FF}_p)\neq \emptyset$, hence $w'\in \Pi_{\FF_p}(w)$. This implies $\ell(\pi_{\QQ_p}(w))\leq \ell(\pi_{\FF_p}(w))$. Combined with Lemma \ref{lem-Fp-Qp}, we obtain $\pi_{\QQ_p}(w)=\pi_{\FF_p}(w)$.
\end{proof}

In particular, we obtain the inequality $\ell(\pi_{\QQ_p}(w))\leq \ell(w)$ for any $w\in W$, since we have established it for $\pi_{\FF_p}(x)$ (Lemma \ref{lem-ineq-0}). The above proposition has also the following immediate consequence:

\begin{proposition} \ \label{prop-sm-Fp-Qp} 
We have $W_{\QQ_p}^{\rm sm}=W_{\FF_p}^{\rm sm}$.
\end{proposition}

We expect that $\Pi_{\QQ_p}=\Pi_{\FF_p}$ in general, but the above results will suffice for the purposes of this article. 

\subsection{Singularities}
Set $G\colonequals \Gcal\otimes_{\ZZ_p}\FF_p$, and write again $\mu\colon \GG_{\textrm{m},\FF_p}\to G$ for the special fiber of $\mu$. Put $\Xcal_0\colonequals \Xscr_{\FF_p}$. For $m\geq 1$, define also $\Xcal_m$ to be the stack of $G$-zips of exponent $m$ attached to the datum $(G,\mu,\varphi^m)$, where $\varphi$ is the $p$-power Frobenius isogeny of $G$. We retain the notations of \S \ref{subsec-notation}. Let $w\in {}^I W$ and $w'\in \Gamma_I(w)$ be a lower neighbor of $w$ with respect to $\preccurlyeq$. Consider the set $\Ucal_m(w,w')\subseteq \Xcal_m$. Define also $\Uscr(w,w')\colonequals \Xscr_w \cup \Xscr_{w'}$ and denote by $\Uscr_{\QQ_p}(w,w')$ its generic fiber. We show the analogue of Proposition \ref{prop-scc-charac} over $\QQ_p$. We state a simplified version below. The proof is completely analogous, so we omit it.

\begin{proposition}\label{prop-scc-Qp}
The following are equivalent:
    \begin{enumerate}
        \item \label{scc1-Qp} $\Uscr_{\QQ_p}(w,w')$ admits a separating canonical cover.
        \item \label{scc2-Qp} One has $w'\notin \pi_{\QQ_p}(\Gamma_{I_w}(w)\setminus \{w'\})$.
        \item \label{scc3-Qp} One has $w'\notin \pi_{\QQ_p}(\Gamma^{\rm sm}_{I_w}(w)\setminus \{w'\})$.
    \end{enumerate}
\end{proposition}

\begin{theorem}\label{thm-main}
The following are equivalent:
\begin{enumerate}
    \item \label{main1} $\Ucal_m(w,w')$ is smooth (equivalently, normal) for some $m\geq 1$.
\item \label{main2} $\Ucal_m(w,w')$ is smooth (equivalently, normal) for all $m\geq 1$.
\item \label{main3} $\Ucal_{0}(w,w')$ is $w$-bounded and admits a separating canonical cover.
\item \label{main4} $\Uscr_{\QQ_p}(w,w')$ is $w$-bounded and admits a separating canonical cover.
\item \label{main5} We have $P_{w'}\subseteq P_w$ and $w'\notin \pi_{\QQ_p}(\Gamma^{\rm sm}_{I_w}(w)\setminus \{w'\})$.
\end{enumerate}
\end{theorem}

\begin{proof}
We already know that \eqref{main1}, \eqref{main2} and \eqref{main3} are equivalent. Since $\pi_{\QQ_p}=\pi_{\FF_p}$, these assertions are equivalent to \eqref{main5}, which is also equivalent to \eqref{main4} using Proposition \ref{prop-scc-Qp}.

\begin{remark}
Suppose the conditions of Theorem \ref{thm-main} are satisfied and define $\Uscr^{(\Pcal_w)}_{\QQ_p}(w,w')$ as the union of the two strata in $\Fscr^{(\Pcal_w)}_{\QQ_p}$. One can show that the map
\begin{equation}
  \pi_{\Pcal_w} \colon  \Uscr^{(\Pcal_w)}_{\QQ_p}(w,w') \to \Uscr_{\QQ_p}(w,w')
\end{equation}
is a bijective, proper, birational and unramified, hence an isomorphism. In particular, $\Uscr_{\QQ_p}(w,w')$ is normal.

%\JS{Conversely, if $\Uscr_{\QQ_p}(w,w')$ is normal, can't we argue that $\Uscr_{\overline{\QQ}}(w,w')$ also, after a choice of a model? Then smoothness means that some Jacobian matrix has maximal rank, which is an open condition : the determinant of some maximal minor is nonzero, therefore it is nonzero mod $\ell$ for some $\ell$, hence $\Ucal_0$ (mod $\ell$) is smooth over $\FF_\ell$.}
\end{remark}
%We first note that the property of admitting a separating canonical cover is satisfied by all the stacks $\Xcal_m$ (for $m\geq 0$) simultaneously or by none of them. The same holds for the property "$w$-bounded", since the canonical parabolics are the same for all involved zip data. Combining this with Theorem \ref{thm-FTZD-normality} (which applies to the stacks $\Xcal_m$ for $m\geq 1$), we deduce that \eqref{main1}, \eqref{main2} and \eqref{main3} are equivalent.

%Property \eqref{main5} is simply a reformulation of \eqref{main4}, as in Proposition \ref{prop-scc-charac}, which is available over $\QQ_p$ too because of Lemma \ref{ineq-Qp}. Finally, \eqref{main3} is equivalent to $w$-boundedness and $w'\notin \pi_{\FF_p}(\Gamma^{\rm sm}_{I_w}(w)\setminus \{w'\})$. Since $\pi_{\QQ_p}$ and $\pi_{\FF_p}$ coincide on small elements (Proposition \ref{prop-sm-Fp-Qp}), this is equivalent to $w'\notin \pi_{\QQ_p}(\Gamma^{\rm sm}_{I_w}(w)\setminus \{w'\})$ and hence also to $w'\notin \pi_{\QQ_p}(\Gamma_{I_w}(w)\setminus \{w'\})$. This terminates the proof.
\end{proof}

Theorem \ref{thm-main} is surprising as it relates the singularities in positive characteristic in the zip stratification (and therefore, in the Ekedahl--Oort stratification - see \S\ref{subsec-Shimura} below) to the characteristic zero object $\Xscr_{\QQ_p}$. We can even extend scalars to an algebraic closure $\overline{\QQ}_p$ and identify this field with the field $\CC$ of complex numbers via a fixed isomorphism $\iota\colon \overline{\QQ}_p\to \CC$. We could then define similarly a map $\pi_{\CC}\colon W\to {}^I W$, which clearly coincides with $\pi_{\QQ_p}$. In theory, we could then (potentially) apply techniques from other fields of mathematics, like complex analysis and differential geometry, to study properties of $\Xscr_{\CC}\colonequals \Xscr_{\QQ_p}\otimes_{\QQ_p,\iota}\CC$. We hope to explore this analytical approach to studying the EO stratification in future research.

\subsection{Shimura varieties}\label{subsec-Shimura}

Let $(\mathbf{G},\mathbf{X})$ be a Shimura datum of abelian type. %\lorenzo{abelian type? you are citing \cite{kisin-hodge-type-shimura}, which gives integral models for abelian type Shimura varieties}. 
In particular, $\mathbf{G}$ is a reductive group over $\QQ$. One can attach to the Shimura datum a well-defined $\mathbf{G}(\CC)$-conjugacy class $[\mu]$ of cocharacters. For a sufficiently small compact open subgroup $K\subseteq \mathbf{G}(\AA_f)$, we have a Shimura variety $\Sh_K(\mathbf{G},\mathbf{X})$, which is a smooth quasi-projective scheme over a number field $\mathbf{E}$ (the field of definition of $[\mu]$). 

\subsubsection{Ekedahl--Oort strata}
Let $p$ be an odd prime\footnote{This assumption is made in \cite{imai2024prismaticrealizationfunctorshimura} for the construction of the zip period map $\zeta$ in \eqref{zeta-eq}.} of good reduction and $v|p$ a place of $\mathbf{E}$. Write \(\mathbf{E}_v\) for the completion of \(\mathbf{E}\) with respect to \(v\). In particular, we have $K=K^pK_p$ where $K^p\subseteq\mathbb{G}(\AA_f^p)$ is open compact and $K_p=\Gcal(\ZZ_p)$ for some $\ZZ_p$-flat reductive model $\Gcal$ of $\mathbf{G}_{\QQ_p}$ such that $G\colonequals \Gcal_{\FF_p}$ is connected. Since \(\mathbf{G}_{\QQ_p}\) is quasi-split, by \cite[Prop.~1.7]{pointsmilne}, after changing $\mu$ to a conjugate cocharacter, we may assume that there is a Borel pair $(\mathbf{B},\mathbf{T})$ of $\mathbf{G}_{\QQ_p}$ satisfying the following conditions
\begin{enumerate}
    \item $\mu$ factors through $\mathbf{T}_{\mathbf{E}_v}$,
    \item $\mu$ is dominant with respect to $\mathbf{B}$.
\end{enumerate}
Let $(\Bcal,\Tcal)$ be the unique Borel pair of $\Gcal$ that extends $(\mathbf{B},\mathbf{T})$. By density, it satisfies the same two properties above. Write $(B,T)$ for the special fiber of $(\Bcal,\Tcal)$. The cocharacter $\mu$ extends to $\Ocal_{\mathbf{E}, v}$, the ring of integers of \(\mathbf{E}_v\). Write \(k(v)\) for the residue field of $\Ocal_{\mathbf{E}, v}$ and again $\mu\colon \GG_{\mathrm{m}, k(v)}\to G_{k(v)}$ for the special fiber of $\mu$.

By results of Kisin \cite{kisin-hodge-type-shimura} and Vasiu \cite{vasiu}, there exists a smooth canonical model $\Scal_K$ over $\Ocal_{\mathbf{E}, v}$. Write $S_K\colonequals \Scal_K\otimes_{\Ocal_{\mathbf{E}, v}}\overline{\FF}_p$ and put $\Xcal\colonequals \GZip^\mu$, the stack of $G$-zips of type $\mu$. By Zhang \cite{zhang} and Imai--Kato--Youcis \cite{imai2024prismaticrealizationfunctorshimura}, %\lorenzo{all of the references we are giving assume that \(p > 2\); do we care? we should state it somewhere}
there is a smooth surjective map 
\begin{equation}\label{zeta-eq}
    \zeta\colon S_K\to \Xcal
\end{equation}
called \emph{the zip period map}, whose fibers are the \emph{Ekedahl--Oort (EO) strata} of $S_K$. For $w\in {}^IW$, we have the corresponding zip stratum $\Xcal_w$ of $\Xcal$ and Ekedahl--Oort stratum $S_{K,w}\colonequals \zeta^{-1}(\Xcal_w)$. For $w\in {}^I W$ and $w'\in \Gamma_I(w)$, set
\begin{equation}
    U_K(w,w')\colonequals S_{K,w}\cup S_{K,w'}=\zeta^{-1}(\Ucal(w,w')).
\end{equation}
It is a locally closed subset of $S_K$, that we endow with the reduced subscheme structure. Locally closed subschemes of $S_K$ of this form will be called \emph{elementary subschemes}.

\subsubsection{Frobenius action} \label{sssec: frob action} %\lorenzo{We should follow the approach of Deligne \& Lusztig to define the Weyl group intrinsically; I think this is similar to what you were saying Wedhorn does. I'll look into this and write it myself, if I can.}
We may identify the root systems $\Phi(\mathbf{G}_{\QQ_p},\mathbf{T})$ and $\Phi(G,T)$ via the reductive $\ZZ_p$-model $\Gcal$, and similarly for the Weyl groups $W(\mathbf{G}_{\QQ_p},\mathbf{T})$ and $W(G,T)$. %If we fix an isomorphism $\iota\colon \overline{\QQ}_p\to \CC$, we may
Since \(\mathbf{G}_{\QQ_p}\) is unramified, the action of \(\Gal(\overline{\QQ}_p/\QQ_p)\) on \(\Phi(\mathbf{G}_{\QQ_p},\mathbf{T})\) is unramified and the action of the $p$-power Frobenius $\sigma\in \Gal(\overline{\FF}_p/\FF_p)$ is the same on \(\Phi(\mathbf{G}_{\QQ_p},\mathbf{T})\) and \(\Phi(G, T)\), via this identification. Moreover, the pair $(\Phi(\mathbf{G}_{\QQ_p},\mathbf{T}),\sigma)$ is independent, up to conjugation, of the choice of a Borel pair $(\mathbf{B},\mathbf{T})$ that satisfies the above assumptions \cite[21.42-43, 24.6]{agrpsmilne}.

The action of $\sigma$ on $\Phi(G,T)$ also uniquely determines its action on the Weyl group $W(G,T)$. Similarly, we obtain a compatible action of $\sigma$ on $W(\mathbf{G}_{\QQ_p},\mathbf{T})$. Recall that the actions of $\sigma$ and $W(\mathbf{G}_{\QQ_p},\mathbf{T})$ on $\Phi(\mathbf{G}_{\QQ_p},\mathbf{T})$ are related by $\sigma(w\cdot \alpha)=\sigma(w)\cdot \sigma(\alpha)$, for all $w\in W(\mathbf{G}_{\QQ_p},\mathbf{T})$ and $\alpha\in \Phi(\mathbf{G}_{\QQ_p},\mathbf{T})$. The following result is an immediate consequence of the results of \S\ref{General algorithm for smoothness}: %\lorenzo{TODO: the previous internal ref was only for the split case, while the proposition here holds in general. Check if there is a better one.} %Theorem \ref{thm-main}:

\begin{proposition}
For any elementary subscheme $U_K(w,w')$ of $S_K$, the smoothness (equivalently, normality) of $U_K(w,w')$ is entirely determined by the triple $(\Phi(\mathbf{G},\mathbf{T}),W(\mathbf{G},\mathbf{T}),\sigma)$.
\end{proposition}

Let $\ell$ be another prime of good reduction for $S_K$. Choose isomorphisms $\iota_p\colon \CC\to \overline{\QQ}_p$ and $\iota_\ell\colon \CC\to \overline{\QQ}_\ell$, and identify the root systems of $\mathbf{G}_{\QQ_p}$ and $\mathbf{G}_{\QQ_\ell}$ with $\mathbf{G}_\CC$. Write $\sigma_p$ and $\sigma_\ell$ for the induced Frobenius actions at $p$ and $\ell$ respectively on $\Phi\colonequals \Phi(\mathbf{G}_\CC,\mathbf{T}_\CC)$ and $\mathbf{W}\colonequals W(\mathbf{G}_\CC,\mathbf{T}_\CC)$. For example, if the maximal tori \(\mathbf{T}\) at $p$ and $\ell$ are both split, then $\sigma_p=\sigma_\ell=\iden_{\Phi}$. We use the notation $U_{K,p}(w,w')$ to denote the elementary subscheme attached to the pair $(w,w')$ in the special fiber at $p$.

\begin{corollary}\label{cor-lp}
Assume that $\sigma_p=\sigma_\ell$ on $\Phi$. For $w\in{}^\mathbf{I} \mathbf{W}$ and $w'\in \Gamma_{\mathbf{I}}(w)$, the following are equivalent:
\begin{enumerate}
\item $U_{K,p}(w,w')$ is smooth (equivalently, normal).
\item $U_{K,\ell}(w,w')$ is smooth (equivalently, normal).
\end{enumerate}
\end{corollary}

\subsubsection{Split primes of good reduction}
We make the following assumption:
\begin{assumption}\label{assum-split}
    The torus $\Tcal$ is split over $\ZZ_p$.
\end{assumption}
In particular, this assumption implies that $\mathbf{E}_v = \QQ_p$ and hence $p$ is split in $\mathbf{E}$. In this setting, we can apply the results of \S\ref{sec-split} and \S\ref{sec-lift-char0}. We retain the notations therein. Let $\Xscr_{\QQ_p}$ be the stack over $\QQ_p$ defined in \eqref{Xscr-def} and $\Xscr_\CC\colonequals \Xscr_{\QQ_p}\otimes_{\QQ_p}\CC$ via the fixed isomorphism $\iota_p\colon \CC\to \overline{\QQ}_p$. The stack $\Xscr_\CC$ only depends on the pair $(\mathbf{G}_\CC, \mu_{\CC})$: it is the stack of $\mathbf{G}_\CC$-zips attached to the triple $(\mathbf{G}_\CC,\mu_{\CC},\iden)$. Write $\mathbf{P}\subseteq \mathbf{G}_{\CC}$ for the parabolic attached to $\mu_\CC$ and denote by $\mathbf{I}\subseteq \Delta$ the type of $\mathbf{P}$. For $w\in {}^{\mathbf{I}}\mathbf{W}$ and $w'\in \Gamma_{\mathbf{I}}(w)$, let
\begin{equation}
    \Uscr_{\CC}(w,w')\colonequals \Xscr_{\CC,w}\cup \Xscr_{\CC,w'}
\end{equation}
be the corresponding elementary $w$-open substack. %For an intermediate parabolic $\mathbf{B}\subset \mathbf{P}_0\subseteq \mathbf{P}$, write $\Fscr_{\CC}^{(\mathbf{P}_0)}$ for the corresponding flag space. 
Write $\mathbf{P}_w$ for the canonical parabolic of $w\in{}^\mathbf{I} \mathbf{W}$. %and the notion of "$w$-boundedness" for $w$-open substacks. Similarly, we have the notion of "admitting a separating canonical cover", using the flag spaces $\Fscr_{\CC}^{(\mathbf{P}_0)}$. 
We can define the map $\pi_\CC\colon \mathbf{W}\to {}^{\mathbf{I}}\mathbf{W}$, similarly to \eqref{piR-maps}. We can similarly define a subset of small elements $\mathbf{W}^{\rm sm}\subseteq \mathbf{W}$.

\begin{theorem}\label{thm-UKp}
The following are equivalent:
\begin{enumerate}
    \item\label{UKp1} $U_{K,\FF_p}(w,w')$ is smooth (equivalently, normal),
    \item \label{UKp2} $\Ucal(w,w')$ is smooth (equivalently, normal),
    \item \label{UKp3} $\Uscr_{\CC}(w,w')$ is $w$-bounded and admits a separating canonical cover.
\item \label{UKp4} We have $\mathbf{P}_{w'}\subseteq \mathbf{P}_w$ and $w'\notin \pi_{\CC}(\Gamma^{\rm sm}_{\mathbf{I}_w}(w)\setminus \{w'\})$.
\end{enumerate}
\end{theorem}
\begin{proof}
The equivalence of the \eqref{UKp2}, \eqref{UKp3} and \eqref{UKp4} simply follows from Theorem \ref{thm-main}. The equivalence of \eqref{UKp1} and \eqref{UKp2} is a consequence of the smoothness and surjectivity of the map $\zeta\colon S_K\to \zeta$ (faithfully flat descent).
\end{proof}

\section{Hasse invariants in characteristic zero}\label{sec-Hasse-0}
In positive characteristic for a Frobenius zip-datum, it was proved in \cite{Goldring-Koskivirta-zip-flags} that when $(G,\mu)$ is of maximal type (for example, in the context of Shimura varieties), Hasse invariants exist on all zip strata. In general, the condition for their existence is that $p^m$ (where $m$ is the exponent of the zip datum) is sufficiently large (Theorem \ref{thm-Hasse}% \lorenzo{TODO: what theorem where aaaaa helppppp}
).

In what follows, we examine the possible extension of these results for $m=0$ (i.e.\ when we replace the Frobenius with the identity map). In this situation, there is no restriction on the characteristic of the field, so we will work over a field of arbitrary characteristic. The equations for Hasse invariants in this setting (when they exist) are usually %\lorenzo{often? sometimes? always?}
obtained by setting $p=1$ in the usual characteristic $p$ Hasse invariants (Theorem \ref{thm-Hasse}), if this substitution makes sense. Therefore, we heuristically think of such sections as ``Hasse invariants over the field $\FF_1$'', even when the base field is $\CC$.

\subsection{Setting}
Let $K$ denote the field $\CC$ or $\overline{\FF}_p$. Fix a cocharacter datum $(G,\mu)$ where $G$ is a connected, reductive group over $K$ and $\mu\colon \GG_{\textrm{m},K}\to G$ is a cocharacter. In this section, we only consider the zip datum $\Zcal$ attached to the triple $(G,\mu,\iden_G)$. Therefore, we simply use the notation $\Xcal$ for the associated stack of $G$-zips, which was previously denoted by $\Xscr_K$. Retain the previous notation for the groups $P,Q,L$ and the zip group $E$. 
%i.e.\ $P\colonequals P_-(\mu)$, $Q\colonequals P_+(\mu)$, with common Levi subgroup $L=\Cent(\mu)$. Write $\Xcal$ for the associated stack of $G$-zips and $E\subseteq P\times Q$ for the attached zip group. 
Fix again $(B,T)$ a Borel pair satisfying:
\begin{enumerate}
    \item $\mu$ factors through $T$,
    \item $B\subset P$.
\end{enumerate}
Let $I\subseteq \Delta$ be the type of $P$. For $w\in {}^IW$, we have the locally closed subset $G_w=E\cdot (BwBz^{-1})$ and the substack $\Xcal_w=[E\backslash G_w]$.

\subsection{\texorpdfstring{Reduction to characteristic $p$}{}} \label{sec-reduction-p}
When $K$ has characteristic $p$, the results of this section can easily be proved by reducing to the Frobenius-type case by the exponent raising method, \S\ref{sec-expo-raising-split}. Hence, we assume here that $K$ has characteristic $0$. In particular, we view $\overline{\QQ}$ as a subfield of $K$. The group $G$ admits a smooth, reductive model over the ring $R=\Ocal_{F_0}[\frac{1}{N_0}]$, for some integer $N_0 \geq 1$ and some number field $F_0$. We may assume that $G$ is split over $R$ and that $\mu,B,T$ are also defined over $R$. The parabolic subgroups $P, Q$, as well as the zip group $E$, are then defined over $R$. Similarly, for any $w\in {}^I W$, the canonical parabolic $P_w$ extends to a parabolic subgroup over $R$. For a fixed choice of $R$ and a maximal ideal $\pfr\subseteq R$, we denote by $\widetilde{g}$ the reduction modulo $\pfr$ of an $R$-point $g$.

\begin{lemma}\label{lem-stab0-char-shifting}
Let $w\in {}^I W$ and $x\in P_w (wz^{-1})Q_w$. One has $\Stab_{E}(x)\subseteq P_w$.
\end{lemma}
% \lorenzo{Couldn't a similar argument be used to prove the opposite inclusion? Lifting points of \(L_w(\FF_p)\)}
% \JS{The reverse inclusion is not true over $\FF_p$, so how do you use characteristic $p$?}
% \lorenzo{I know that the reverse inclusion does not hold over \(\FF_p\), but the number of \(\FF_p\) points of \(\Stab_E(x)\) }
% \JS{ok I understand, that's a good idea}
% \JS{Sure! I am free qll dqy}
%\lorenzo{I've replaced \(\Ocal_F\) by \(\Ocal_F[\frac{1}{N}]\). The arguments still work, but I couldn't see how to avoid inverting some \(N\).}\JS{Good, thanks}
\begin{proof}
It suffices to show the result for all $x\in (P_w (wz^{-1})Q_w)(\overline{\QQ})$, and even for all $x$ with coefficients in $\Ocal_F[\frac{1}{N}]$, for all number fields $F$ containing $F_0$ and all \(N > 0\). Moreover, it suffices to show that for all such $F$ and all $(a,b)\in \Stab_E(x)$ with coefficients in $\Ocal_F[\frac{1}{N}]$, we have $a\in P_w$. To show this, let $\pfr\subseteq \Ocal_F[\frac{1}{N}]$ denote a place of $F$ lying over some prime \(p\) that is totally split in $F$. Since $k(\pfr)\colonequals\Ocal_F[\frac{1}{N}]/\pfr\simeq \FF_p$, the mod $\pfr$ reduction $(\widetilde{a},\widetilde{b})$ lies in $E(\FF_p)$. This implies that $(\widetilde{a},\widetilde{b})$ lies in the stabilizer of $\widetilde{x}$ with respect to the Frobenius-type zip datum attached to $(G_{\FF_p},\mu_{\FF_p})$. Since $G$ is split over $R$, the canonical parabolic subgroups attached to $w$ over $K$ and over $\overline{\FF}_p$ coincide. Therefore $\overline{a}\in P_w(\FF_p)$ by Proposition \ref{stab-contained-Pw}. Since this holds for infinitely many $\pfr$, we deduce that $a\in P_w$.
%\par Conversely, consider \(a \in L_w(\Ocal_F)\) and let \(\pfr\) be a prime as above. The reduction \(\overline{a}\) modulo \(\pfr\) of \(a\) is contained in \(\Stab_E(x)\) if and only if \(\overline{a} \in \Cent_L(\overline x)\). We do know that \(\overline{x} \in E \cdot (wz^{-1})\).
\end{proof}

In the following result, we see another illustration of the same method of proof, by reducing to a Frobenius-type zip datum.

% \lorenzo{TODO: introduce the notation \(G^{(P_w)}_w\) for flag strata in the review section about flag spaces. We use it in \S\ref{sec-split}, but it could be introduced before and used more generally.}\JS{I added this (alternative) notation in section \ref{subsub-flag}}
\begin{lemma}\label{lem-orbit-Pw}
Let $w\in {}^I W$ and $x\in BwBz^{-1}$. Then, we have 
\[\{(a,b)\in E \ | \ axb^{-1}\in G^{(P_w)}_w\}\subseteq P_w\times G.\]
\end{lemma}

\begin{proof}
First, note that the analogous statement in characteristic $p$ for a Frobenius-type zip datum $\Zcal$ is an immediate consequence of the fact that $G^{(P_w)}_{w}$ is a single $E'_{\Zcal,P_w}$-orbit, and that $\Stab_E(x)\subseteq P_w\times G$ (Proposition \ref{stab-contained-Pw}).

We prove the result in characteristic zero by reducing modulo $p$. Again, choose a number field $F$ and assume that all objects are defined over $\Ocal_F[\frac{1}{N}]$. Let $\pfr$ denote a place of $F$ not dividing \(N\) such that \(k(\pfr) \cong \FF_p\) (for some prime $p$). By Corollary \ref{cor-Gw-P0}, the element $\widetilde{x}$ lies in $G^{(P_w)}_w(\FF_p)$, which coincides with the set of $\FF_p$-points of the $P_w$-zip stratum of $w$, with respect to the Frobenius-type zip datum attached to $(G_{\FF_p},\mu_{\FF_p})$. Since the result is true in this case, we deduce that $\widetilde{a}\in P_w(\FF_p)$. Since this holds for infinitely many places $\pfr$, we deduce that $a\in P_w$.
\end{proof}

\subsection{Canonical flag strata}\label{subsec-flag-strata-char0}
In this section, we extend the results of Proposition \ref{prop-Pw} to the case considered here, which is not of Frobenius-type. For an arbitrary parabolic subgroup $B\subseteq P_0 \subseteq P$, recall that $\pi_{P_0}(\Fcal^{(P_0)}_w)=\Xcal_w$ and $\pi_{B,P_0}(\Fcal^{(B)}_w)=\Fcal^{(P_0)}_w$ by \eqref{pi-image-minimal-relative}.

\begin{corollary}
The map $\pi_{P_w}\colon \Fcal^{(P_w)}_w\to \Xcal_w$ is an isomorphism.
\end{corollary}

\begin{proof}
We already know that it is proper and surjective. It suffices to show that it is bijective and unramified, i.e.\ that all geometric fibers are isomorphic to $\spec(K)$. Each geometric point of $\Xcal_w(K)$ can be represented as the $E$-orbit of an element of $x\in (BwBz^{-1})(K)$. We consider the fiber above $x$, which can be computed after base change to $BwBz^{-1}$, as shown in the diagram below.
\begin{equation}
    \xymatrix@M=7pt{
\Fcal^{(P_w)}_w \ar[d]_{\pi_{P_w}} \ar@{=}[r] & \left[ E'_{P_w} \backslash G^{(P_w)}_w \right] \ar[d] & E'_{P_w} \backslash (E\times G^{(P_w)}_w) \ar[l] \ar[d]^{\widetilde{s}} & E'_{P_w} \backslash  Y_w \ar@{_{(}->}[l] \ar[d] \\
\Xcal_w \ar@{=}[r] & \left[ E \backslash G_w \right]  & G_w \ar[l] & BwBz^{-1} \ar@{_{(}->}[l]
}
\end{equation}
In this diagram, $E'_{P_w}$ acts on $E\times G^{(P_w)}_w$ by: $\alpha'\cdot (\alpha,g)=(\alpha'\alpha,\alpha'\cdot g)$ for all $\alpha'\in E'_{P_w}$, $\alpha\in E$, $g\in G^{(P_w)}_w$, and the vertical map $\widetilde{s}$ is induced by the $E'_{P_w}$-invariant map $s\colon E\times G^{(P_w)}_w \to G_w$, $(\alpha,g)\mapsto \alpha^{-1}\cdot g$. Finally, $Y_w$ in the rightmost column of the diagram is defined as the inverse image of $BwBz^{-1}$ by $s$, i.e.
\begin{equation}
    Y_w=\{(\alpha,g)\in E\times G^{(P_w)}_w \ | \ \alpha^{-1}\cdot g \in BwBz^{-1}\}.
\end{equation}
By Lemma \ref{lem-orbit-Pw}, this set is simply isomorphic to $E'_{P_w}\times BwBz^{-1}$ via the map $(\alpha, h)\mapsto (\alpha,\alpha\cdot h)$. Hence $E'_{P_w} \backslash  Y_w$ is simply $BwBz^{-1}$ and the fiber over \(x\) is \(\spec(K)\). This finishes the proof.
\end{proof}

Finally, we prove also the analogue of Proposition \ref{prop-Pw}~\eqref{Pw-3}:

\begin{lemma}\label{lem-Bruhat0}
Let $w\in {}^I W$. The stratum $\Fcal^{(P_w)}_w$ is Bruhat. 
\end{lemma}

\begin{proof}
By definition, the canonical parabolic of $w$ with respect to $\Zcal$ is the same as the canonical parabolic of $w$ with respect to $\Zcal_{P_w}$. By \eqref{Gw-withPw} applied to $\Zcal_{P_w}$, we see that $G^{(P_w)}_w$ contains $P_w wz^{-1}$, hence also $P_w wz^{-1}Q_w$. It follows that $G^{(P_w)}_w$ coincides with the $P_w\times Q_w$-orbit of $wz^{-1}$. 
%For the general case, we consider the stack $\Xcal_{P_0}$. The canonical parabolic of $w$ with respect to $\Zcal_{P_0}$ is $P_0$ (this can be checked on root data, hence follows from the usual case) \lorenzo{this is unrelated, but does this imply that \(\varphi_w\) fixes \(I_w \subseteq I\) not only globally, but also pointwise, in general?}. By the above, $G^{(P_0)}_w$ coincides with the $P_0\times Q_0$-orbit of $wz^{-1}$. This terminates the proof.
\end{proof}

\subsection{Line bundles}
Let $\lambda\in X^*(L)$ be a character of $L$. Identify $X^*(L)$ with $X^*(E)$ via the first projection $\pr_1\colon E\to P$ and the projection $P\to L$ onto the Levi. Consider the stack $\Xcal=[E\backslash G_k]$. For each intermediate parabolic $B\subset P_0 \subset P$, we have a flag space $\pi_{P_0}\colon \Fcal^{(P_0)} \to \Xcal$ such that $\Fcal^{(P_0)}=[E'_{P_0}\backslash G]$. We may identify $X^*(E'_{P_0})=X^*(P_0)$ via the first projection $\pr_1\colon E'_{P_0}\to P_0$. Hence, any character $\lambda\in X^*(L_0)$ gives rise to a line bundle $\Lcal(\lambda)$ on $\Fcal^{(P_0)}$. If $P_1\subseteq P_0$ is another standard parabolic, the construction of the line bundles $\Lcal(\lambda)$ is compatible with the projection $\pi_{P_1,P_0}\colon \Fcal^{(P_1)}\to \Fcal^{(P_0)}$ and the natural inclusion $X^*(P_0)\subseteq X^*(P_1)$. For each intermediate parabolic $P_0$, set $\Bru^{(P_0)}\colonequals [(P_0\times Q_0)\backslash G]$, called the Bruhat stack of $P_0$. We have a natural projection map 
\begin{equation}
    \beta_{P_0}\colon \Fcal^{(P_0)}\to \Bru^{(P_0)}.
\end{equation}
given by the natural inclusion $E'_{P_0}\subseteq P_0\times Q_0$. This gives rise to a coarser stratification on $\Fcal^{(P_0)}$, called the \emph{Bruhat stratification}. A stratum $\Fcal^{(P_0)}_w$ is Bruhat (as defined in \cite{Koskivirta-LaPorta-Reppen-singularities}) if and only if it coincides with a fiber of $\beta_{P_0}$. For two standard intermediate parabolic subgroups $P_1\subseteq P_0$, we have a commutative diagram:
\[
\xymatrix@1@M=8pt{
\Fcal^{(P_1)} \ar[d]_{\pi_{P_1,P_0}} \ar[r]^-{\beta_{P_1}} & \Bru^{(P_1)} \ar[d] \\
\Fcal^{(P_0)} \ar[r]^-{\beta_{P_0}} & \Bru^{(P_0)}.
}
\]
For any $\lambda_1,\lambda_2\in X^*(L_0)$, we have an associated line bundle $\Lcal(\lambda_1,\lambda_2)$ on the Bruhat stack $\Bru^{(P_0)}$. Its pullback via the map $\beta_{P_0}$ is given by
\begin{equation}
    \beta_{P_0}^*(\Lcal(\lambda_1,\lambda_2))=\Lcal(\lambda_1+z^{-1}\lambda_2).
\end{equation}
For $w\in W$, let $\Bru^{(B)}_w$ the corresponding stratum, i.e.\ $[(B\times B)\backslash BwB]$. Recall (\cite[Th. 2.2.1]{koskgold}) that the space $H^0(\Bru_w^{(B)},\Lcal(\lambda_1,\lambda_2))$ is nonzero if and only if $\lambda_2=-w^{-1}\lambda_1$. Define $E_w\subseteq \Phi^+$ as the set
\begin{equation}
    E_w\colonequals \{\alpha\in \Phi^+ \ | \ ws_\alpha\leq w \textrm{ and } \ell(ws_\alpha)=\ell(w)-1\}.
\end{equation}
Then, for any $f\in H^0(\Bru_w^{(B)},\Lcal(\lambda,-w^{-1}\lambda))$, we have Chevalley's formula\footnote{The reference given here contains a sign error.} (\loccitn):
\begin{equation}
    \divi(f)=\sum_{\alpha\in E_w} \langle \lambda,w\alpha^\vee \rangle [\overline{\Bru}^{(B)}_{ws_\alpha}].
\end{equation}
Let $P_0$ be an intermediate parabolic and $\Fcal^{(P_0)}_w$ a Bruhat stratum. We say that a section $f\in H^0(\Fcal^{(P_0)}_w,\Lcal(\lambda))$ (for $w\in {}^{I_0}W$) is \emph{Bruhat} if it arises by pullback from the corresponding stratum of $\Bru^{(P_0)}$, equivalently, if the corresponding function $f\colon G^{(P_0)}_w\to \AA^1$ is $P_0\times Q_0$-equivariant. A Bruhat section is always uniquely determined up to scalar by its weight $\lambda$. By Rosenlicht's theorem (\cite[Proposition 1.1]{Picard-group-KKV}), we have:
\begin{lemma}\label{bru-sections}
Let $\Fcal^{(P_0)}_w$ be a Bruhat stratum (for $w\in {}^{I_0}W$). A section $f\in H^0(\Fcal^{(P_0)}_w,\Lcal(\lambda))$ is Bruhat if and only if it is nowhere vanishing on $\Fcal^{(P_0)}_w$.
\end{lemma}

\begin{comment}
    
Let $w\in {}^IW$. Since $\Fcal^{(P_w)}_w$ is a Bruhat stratum, the set $G^{(P_w)}_w$ coincides with the set $P_w wz^{-1}Q_w$ (the $P_w\times Q_w$-orbit of $wz^{-1}$).

\begin{lemma}
Let $w\in {}^I W$ and $\chi\in X^*(L_w)$. Up to a nonzero scalar, there exists a unique nonzero function $f_w\colon P_w wz^{-1}Q_w\to \GG_{\mathrm{m}}$ satisfying $f_w(ax)=\lambda(a)f_w(x)$ for all $a\in P_w$ and $x\in P_w wz^{-1}Q_w$.
\end{lemma}

\begin{proof}
Since $\pi_{B,P_w}(\Fcal^{(B)}_w)=\Fcal^{(P_w)}_w$, any element of $P_w wz^{-1}Q_w$ is of the form $axb^{-1}$ with $x\in BwBz^{-1}$ and $(a,b)\in E'_{P_w}$. Hence, we must define
\begin{equation}
    f_w(axb^{-1})=\lambda()
\end{equation}
\end{proof}
\end{comment}

\subsection{Hasse invariants}\label{subsec-Hasse-inv}

We say that a section $\Ha_w\in H^0(\overline{\Xcal}_w,\Lcal(\lambda))$ is a \emph{Hasse invariant} on $\overline{\Xcal}_w$ if the non-vanishing locus of $\Ha_w$ coincides with $\Xcal_w$. We consider again the maps
\begin{equation}
\Fcal^{(B)}_w \xrightarrow{\pi_{B,P_w}} \Fcal^{(P_w)}_{w} \xrightarrow{\pi_{P_w}}\Xcal_w
\end{equation}
and similarly between the Zariski closures of these strata. If $\Ha_w\in H^0(\overline{\Xcal}_{w},\Lcal(\lambda))$ (for $\lambda\in X^*(L)$) is a Hasse invariant, its pullback $\pi_{P_w}^*(\Ha_w)$ is a section over $\overline{\Fcal}^{(P_w)}_w$ which is nowhere vanishing on $\Fcal^{(P_w)}_w$. Since $\Fcal^{(P_w)}_w$ is a Bruhat stratum, the section $\pi_{P_w}^*(\Ha_w)$ on the stratum $\Fcal^{(P_w)}_w$ is Bruhat, by Lemma \ref{bru-sections}. Similarly, $\pi_{B}^*(\Ha_w)$ on the stratum $\Fcal^{(B)}_w$ is a Bruhat section. We recall two geometric lemmas:

\begin{lemma}\label{geom-lem1}
Let $f\colon X\to Y$ be a surjective morphism of normal schemes. Let $\Lcal$ be a line bundle on $Y$ and $U\subseteq Y$ a dense open subset %\lorenzo{do we need dense open subset?}
and let $s\in H^0(U,\Lcal)$ be a nonzero section. Then, $s$ extends to a global section over $Y$ if and only if $f^* s\in H^0(f^{-1}(U),f^*\Lcal)$ extends to a global section over $X$.
\end{lemma}

\begin{proof}
If $s$ extends to $Y$, then clearly $f^*s$ extends to $X$. Conversely, assume that $f^*s$ extends to $X$.
Since $Y$ is normal, it suffices to show that the divisor of $s$ does not have any poles. If $s$ had a pole along a codimension one subscheme $Z\subseteq Y$, then $f^*s$ would have a pole along any codimension one subscheme $Z'\subseteq f^{-1}(Z)$, which is impossible. The result follows.
\end{proof}

\begin{lemma}[{\cite[Lemma 5.3.2]{boxer}}]\label{geom-lem2}
Let $f\colon X\to Y$ be a proper surjective map of reduced noetherian schemes and let $\Lcal$ be a line bundle on $Y$. Let $s\in H^0(X,f^*\Lcal)$ be a section. Suppose that there exists a open subset $U\subseteq Y$ such that $f^{-1}(U)\to U$ is an isomorphism and $s$ vanishes set-theoretically on $f^{-1}(X \setminus U)$. Then, there exists an integer $n\geq 1$ and a section $t\in H^0(Y,\Lcal^n)$ such that $t$ pulls back to $s^n$.
\end{lemma}

We can now prove the main result of this section.
\begin{theorem}\label{thm-Hasse-char0}
Let $w\in {}^{I}W$ and $\lambda\in X^*(L)$. The following are equivalent:
\begin{enumerate}
\item\label{cond1} There exists a Hasse invariant $\Ha_w\in H^0(\overline{\Xcal}_w,\Lcal(m\lambda))$ for some $m\geq 1$.
\item\label{cond2} There exists $\lambda_0\in X^*(T)$ such that $\lambda = w\lambda_0-z\lambda_0$ and $\langle \lambda_0,\alpha^\vee \rangle <0$ for all $\alpha\in E_w$.
\end{enumerate}
\end{theorem}

\begin{proof}
Assume that there exists a Hasse invariant $\Ha_w\in H^0(\overline{\Xcal}_w,\Lcal(\lambda))$. The pullback via the map $\pi_B\colon \overline{\Fcal}^{(B)}_w\to \overline{\Xcal}_w$ is a Bruhat section, hence arises by pullback from the Bruhat stack $[(B\times B)\backslash G]$. This implies that there exists $\chi\in X^*(T)$ such that $\Lcal(\chi,-w^{-1}\chi)$ admits a section $f_w\in H^0(\overline{\Bru}_w^{(B)},\Lcal(\chi,-w^{-1}\chi))$ which pulls back to $\pi_B^*(\Ha_w)$ under $\beta_B$. In particular, we have $\lambda=\chi-zw^{-1}\chi$. Since the preimage of $\Xcal_w$ by the map $\pi_B\colon \overline{\Fcal}^{(B)}_w\to \overline{\Xcal}_w$ coincides with $\Fcal_w$, the non-vanishing locus of $f_w$ is $\Bru^{(B)}_w$. By Chevalley's formula, we have $\langle w\chi, \alpha^\vee \rangle <0$ for all $\alpha\in E_w$, hence we may take $\lambda_0=w^{-1}\chi$.

Conversely, assume \eqref{cond1} and set $\chi=w\lambda_0$. Then, there exists $f_w\in H^0(\overline{\Bru}_w^{(B)},\Lcal(\chi,-w^{-1}\chi))$ with non-vanishing locus $\Bru^{(B)}_w$. Put $h_w\colonequals \beta_{B}^*(f_w)\in H^0(\overline{\Fcal}^{(B)}_w,\Lcal(\lambda))$ for $\lambda=\chi-zw^{-1}\chi$. Since $\pi_{B,P_w}\colon \Fcal^{(B)}_w\to \Fcal^{(P_w)}_w$ is finite \'{e}tale, we may consider the norm $\Nm(h_w)$, which is a section over $\Fcal^{(P_w)}_w$ of $\Lcal(m\lambda)$ for some $m\geq 1$. Moreover, $\Nm(h_w)$ is again nowhere vanishing, hence is Bruhat by Lemma \ref{bru-sections}. Since the pullback $\pi_{B,P_0}^*(\Nm(h_w))$ is then also Bruhat, it must coincide, up to a scalar, with $h_w^m$. After relabelling the section $h_w$, we may thus assume that $h_w$ descends to a section over $\Fcal^{(P_w)}_w$, that we continue to denote by $h_w$. Moreover, by Lemma \ref{geom-lem1}, this section extends to $\overline{\Fcal}^{(P_w)}_w$. It is clear that the non-vanishing locus of this extension is exactly $\Fcal^{(P_w)}_w$. Finally, since $\pi_{P_w}\colon \Fcal^{(P_w)}_w\to \Xcal_w$ is an isomorphism, Lemma \ref{geom-lem2} shows that some power of $h_w$ descends to $\overline{\Xcal}_w$. This section is clearly a generalized Hasse invariant.
\end{proof}

\subsection{\texorpdfstring{Example: \texorpdfstring{$\GL_{n}$}{general linear groups}}{}}\label{sub-exampleGL}
Let $n\geq 2$ and $G=\GL_{n,K}$, endowed with the cocharacter $\mu \colon t \mapsto \diag(t\mathbf{1}_{r},\mathbf{1}_{s})$ where $r,s$ are positive integers such that $r+s=n$. Let $(B,T)$ denote the usual lower-triangular Borel pair. Identify $X^*(T)=\ZZ^n$ such that $(a_1,\dots,a_n)\in \ZZ^n$ corresponds to the character $\diag(x_1,\dots,x_n)\mapsto \prod_{i=1}^n x_i^{a_i}$. Identify $W$ with the permutation group $\Sfr_n$. The simple roots are given by $\alpha_i\colonequals e_i-e_{i+1}$ for $1\leq i \leq n-1$, where $(e_1,\dots,e_n)$ is the standard basis of $\ZZ^n$. We make the stratification $(\Xcal_w)_{w\in {}^I W}$ explicit for certain signatures $(r,s)$.

\subsubsection{Schur complement}
We give a list of natural invariants that make it possible to partially describe the stratification on the stack $\Xcal$. Let $M\in \GL_{n,K}$ be a matrix, that we write in block form
\begin{equation}
    g=\left(\begin{matrix}
       A&B\\C&D 
    \end{matrix}\right), \quad A\in \Mat_r(K), \ B\in \Mat_{r,s}(K), \ C\in \Mat_{s,r}(K), \ D\in \Mat_{s}(K).
\end{equation}
We sometimes write $A(g)$, $B(g)$, $C(g)$ and $D(g)$ for the above submatrices. Define
\begin{equation}
    \Delta(g)\colonequals A, \quad \Delta'(g)\colonequals \det(A)D-C\Adj(A)B
\end{equation}
where $\Adj(A)$ is the adjugate matrix of $A$, i.e.\ the transpose of the cofactor matrix of $A$. The function $\Delta'$ is related to the Schur complement of $M$. The following lemma is an easy computation:

\begin{lemma}\label{lem-Delta-conj}
For all $(x,y)\in E$ and $M\in \GL_{n,K}$, one has
\begin{equation}
\Delta(xMy^{-1})=\Delta(x)\Delta(M)\Delta(x)^{-1}\quad \textrm{and} \quad \Delta'(xMy^{-1})=D(x)\Delta'(M)D(x)^{-1}.
\end{equation}
\end{lemma}

For any integer $m\geq1$, let $\Jcal_m$ be the quotient stack $\GL_{m,K} \backslash \Mat_{m,K}$, where $\GL_{m,K}$ acts on $\Mat_{m,K}$ by conjugation. It follows from the above lemma that the functions $\Delta,\Delta'$ provide a well-defined map of $K$-stacks:
\begin{equation}
    (\Delta,\Delta')\colon \Xcal\to \Jcal_r\times \Jcal_r.
\end{equation}
In many cases, it is possible to describe the stratification $(\Xcal_w)_{w\in {}^I W}$ in terms of the invariants $\Delta,\Delta'$. For example, this is the case for $(r,s)=(n-1,1)$ (see \S\ref{subsec-n-1} below). However, these invariants are in general not sufficent to describe the stratification. For example, when $(r,s)=(3,2)$, the unique length $0$ and length $1$ strata in $\Xcal$ contain elements which are sent to the same elements of $\Jcal_r\times \Jcal_r$. Thus, it is not possible to distinguish these two strata using only $\Delta$, $\Delta'$. 

\subsubsection{Signature $(n-1,1)$}\label{subsec-n-1}
Assume that $(r,s)=(n-1,1)$. The set ${}^I W$ consists of $n$ elements, numbered $x_0,\dots x_{n-1}$, such that $\ell(x_i)=n-1-i$. In other words, the codimension of $\Xcal_{x_i}$ in $\Xcal$ is $i$. Explicitly,
\begin{equation}
x_i=\begin{pNiceArray}{ccccccc}
    \Block[borders={bottom, right}]{2-2}{\mathbb{1}_i} &&&&& \\
    &&&&& \\
&&0&\Block[borders={bottom, right,left,top}]{4-4}{\mathbb{1}_{n-i-1}} &&& \\
    &&& \\
    &&& \\
    &&&&& \\ \hline
    &&1&&& 
\end{pNiceArray}
\end{equation}
We have $E_{x_i}=\{e_{i+1}-e_{i+2}, \dots, e_{i+1}-e_n\}$. Set $\lambda_i\colonequals (0,\dots,0,1,\dots,1)$, where $1$ appears $n-i-1$ times and $0$ appears $i+1$ times. One sees that $x_i\lambda_i-z\lambda_i=0$ and $\langle \lambda_i,\alpha^\vee\rangle<0$ for all $\alpha\in E_{x_i}$. This shows that the trivial line bundle $\Ocal_\Xcal$ admits a generalized Hasse invariant $\Ha_i$ on the stratum $\Xcal_{x_i}$.

Concretely, we can express $\Ha_i$ explicitly as follows. For $g\in \GL_n$, denote by $\chara_{\Delta(g)}(X)=\det(X \mathbb{1}_{n-1}-\Delta(g))$ the characteristic polynomial of $\Delta(g)$, and write
\begin{equation}
    \chara_{\Delta(g)}(X)=X^{n-1}+\Ha_{n-2}(g)X^{n-2}+\dots+\Ha_{1}(g)X+\Ha_{0}(g).
\end{equation}
For the stratum $\Xcal_{n-1}$, simply set $\Ha_{n-1}(g)=1$. For each $0\leq i \leq n-1$, we view $\Ha_i$ as a function on the stratum $G_{x_i}$. One easily sees that it is $E$-invariant and that $\Ha_i$ is a generalized Hasse invariant for $\Xcal_{x_i}$ (of weight $\lambda=0$).

\subsubsection{\texorpdfstring{The map $\pi$}{}}
As we saw in many results, such as Theorem \ref{thm-main}, the map $\pi\colon W\to {}^IW$ encodes the singularities of the Ekedahl--Oort stratification. In general, the combinatorial determination of this map is a difficult problem. In the case $(r,s)=(n-1,1)$, we give an
explicit formula for $\pi$. For a polynomial $P(X)$ with coefficients in an integral domain, denote by $\val_X(P)$ the valuation of $P$. Identify ${}^IW$ with the set $\{0,1,\dots,n-1\}$ via the map $i\mapsto x_i$. For $w\in \Sfr_n$ and $b\in B_L$, consider the polynomial $\chara_{b\Delta(wz^{-1})}(X)$, viewed as a polynomial with coefficients in the $K$-algebra of regular functions on $B_L$.

\begin{proposition}\label{prop-pi-GLn}
For any $w\in \Sfr_n$, we have $\pi(w)=\val_X(\chara_{b\Delta(wz^{-1})}(X))$.
\end{proposition}

\begin{proof}
We know that the set $\pi(w)$ corresponds to the zip stratum of largest dimension which intersects $B_Lwz^{-1}$. By the construction of Hasse invariants, $\val_X(\chara_{b\Delta(w)}(X))$ returns the smallest value of $i$ such that $\Ha_i$ is nonzero on $B_Lwz^{-1}$. The result follows.
\end{proof}

\subsubsection{The case $(r,s)=(2,2)$} \label{sec-U22}
In this section, we choose $(r,s)=(2,2)$. We can describe all strata of $\Xcal$ in terms of the (matrix-valued) invariants $\Delta$, $\Delta'$. They can be interpreted as sections of a certain a vector bundle of rank $4$ over the stack $\Xcal$. Denote also by $\Ha_i$, $\Ha'_i$ the invariants obtained as the coefficients of the characteristic polynomials of $\Delta$, $\Delta'$ respectively. Specifically, for $g\in G$ we set:
\begin{align}
    \Ha_0(g)\colonequals \det(\Delta(g)) \quad & \quad \Ha'_0(g)\colonequals \det(\Delta'(g)), \\
        \Ha_1(g)\colonequals \Tr(\Delta(g)) \quad & \quad \Ha'_1(g)\colonequals \Tr(\Delta'(g)).
\end{align}
The functions $\Ha_i,\Ha'_i$ are elements of $H^0(\Xcal,\Ocal_\Xcal)$. We describe below the strata $\Xcal_w$ explicitly using these invariants. We organize the strata from top to bottom in decreasing order of dimensions. Closure relations are indicated by a line. We write $[a_1 \ a_2 \ a_3 \ a_4]$ to denote the permutation $i\mapsto a_i$.
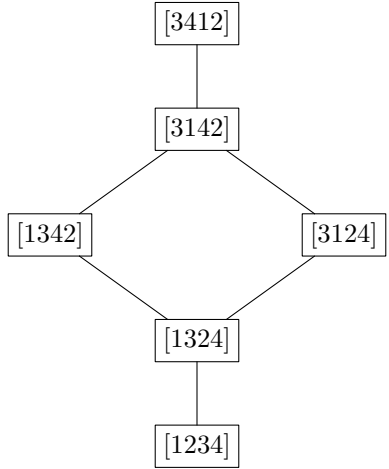
$$\xymatrix{
& *+[F]{[3412]} \ar@{-}[d] & \\
& *+[F]{[3142]} \ar@{-}[dr] \ar@{-}[dl] & \\
*+[F]{[1342]}\ar@{-}[dr]  && *+[F]{[3124]}\ar@{-}[dl]  \\
& *+[F]{[1324]} \ar@{-}[d] & \\
& *+[F]{[1234]} & \\
}
$$
\medskip

For each element $w\in {}^I W$, we indicate in the table below the equation of each stratum $G_w$ in terms of the invariants defined above. We also record the type of the canonical parabolic $I_w$ in the second column. Write $\mathbf{0}_2$ for the zero matrix of size $2\times 2$.

\begin{center}
    \begin{tabular}{|c|c|l|} 
\hline
$w$ & $I_w$  & Equations for $G_w$ \\ \hline
    $[3412]$& $I$  & $\Ha_0(g)\neq 0$ \\ \hline
    $[3142]$  & $\emptyset$ &$\Ha_0(g)=0, \ \Ha_1(g)\neq 0, \ \Ha'_1(g)\neq 0$ \\ \hline
    $[1342]$ & $\emptyset$ &$\Ha_0(g)=0, \ \Ha_1(g)\neq 0, \ \Ha'_1(g) =0$  \\ \hline
    $[3124]$ & $\emptyset$ &$ \Ha_0(g)=0, \ \Ha_1(g)= 0, \ \Ha'_1(g)\neq 0$\\ \hline
    $[1324]$ & $\emptyset$ &$\Delta(g)\neq \mathbf{0}_2, \ \Ha_0(g)=0, \ \Ha_1(g)= 0, \ \Ha'_1(g) =0$  \\ \hline
    $[1234]$ & $I$ &$\Delta(g)= \mathbf{0}_2$  \\ \hline
\end{tabular}

\end{center}
Let us consider the stratum $w_2=[3124]$. It corresponds to the element $w_2$ defined in \eqref{w12-eq}. Since $\gcd(r,s)=2$, we know by Theorem \ref{thm-Urs-smooth} that the union of the strata parametrized by $w_2$ and $w'=[1324]$ is smooth. By the above equations, one sees that $\overline{G}_{w_2}$ is the set of matrices $g\in \GL_4$ such that $\Delta(g)$ is nilpotent. It is easy to see that the surface
\begin{equation}
    \{A\in \Mat_{2,K} \ | \ A^2=\mathbf{0}_2\}
\end{equation}
is singular at the point $\mathbf{0}_2$ only. Hence, one sees that the singular locus of $\overline{G}_{w_2}$ coincides with the stratum $G_{\iden}=\{g\in \GL_{4,K} \ | \ \Delta(g)=\mathbf{0}_2\}$ (since $\Delta$ is a smooth map $\GL_{4,K}\to \Mat_{2,K}$). This answers the analogue of Remark \ref{rmk-Urs-Uibar} in the non-Frobenius-type case. By analogy, we expect that $\overline{\Ucal}_{w_2}$ is singular along $\Xcal_{\iden}$ for the usual stack of $G$-zips (and hence for the corresponding unitary Shimura variety).

\begin{proposition} \ 
\begin{enumerate}
    \item For all $w\in {}^IW\setminus\{ [1324] \}$, the stratum $\Xcal_w$ admits a Hasse invariant.
    \item The stratum $\Xcal_{[1324]}$ does not admit a Hasse invariant. %\lorenzo{of weight \(\lambda = 0\) maybe?}.
\end{enumerate}
\end{proposition}

\begin{proof}
One sees immediately from the above table that for each $w\in {}^IW\setminus\{ [1324] \}$, one the sections $\Ha_i, \Ha'_i$ (for $i=0,1$) or a product thereof gives a Hasse invariant for $\Xcal_w$. Assume now that $w=[1324]$, for which $E_w=\{ \alpha_2 \}$. By Theorem \ref{thm-Hasse-char0}, the existence of a Hasse invariant for $\Xcal_w$ amounts to the existence of $\lambda_0=(a_1,a_2,a_3,a_4)\in \ZZ^4$ such that $w\lambda_0-z\lambda_0\in X^*(L)$ %\lorenzo{why can't we take \(\lambda = w\lambda_0-z\lambda_0\) more generally?} 
and $\langle\lambda_0,\alpha_2^\vee\rangle <0$. %\lorenzo{I guess it's just a sign, but shouldn't this be \(<0\)?}.
The first condition translates to $a_1-a_3=a_3-a_4$ and $a_2-a_1=a_4-a_2$, hence $a_2=a_3$. In particular $\langle\lambda_0,\alpha_2^\vee\rangle =0$. 
\end{proof}

%\lorenzo{I think one can do this more generally for \(\mu(t) = \diag(t^{m-1}\mathbf{1}_{r_1}, t^{m-2}\mathbf{1}_{r_2}, \ldots, \mathbf{1}_{r_m})\), for \(r_1 + r_2 + \ldots + r_m = n\).}\JS{For most signatures, Hasse invariants do not exist, so I decided to restrict to $(n-1,1)$}

\bibliographystyle{alpha}
\bibliography{bib.bib}

\end{document}